\documentclass[reqno, 11pt]{amsart}
\usepackage[top=4cm, bottom=3cm, left=3.2cm, right=3.2cm]{geometry}
\usepackage{enumerate}
\usepackage{amssymb,amsmath}
\usepackage{graphicx}
\usepackage{subfigure}

\usepackage{amsthm}
\usepackage{amsmath}
\usepackage{amssymb}

\usepackage{bm}
\usepackage{geometry} 
\usepackage{graphicx} 

\usepackage{booktabs} 
\usepackage{array} 
\usepackage{paralist} 
\usepackage{verbatim} 
\usepackage{caption} 
\usepackage{cases}

\usepackage{multirow}
\usepackage{graphicx}
\usepackage{float}
\usepackage{subfigure}
\usepackage{algorithmicx,algorithm}
\usepackage{booktabs}
\usepackage{psfrag}
\usepackage[noend]{algpseudocode}

\newtheorem{thm}{Theorem}

\newtheorem{lem}{Lemma}
\newtheorem{prop}{Proposition}

\allowdisplaybreaks

\numberwithin{equation}{section} \numberwithin{lem}{section}
\numberwithin{thm}{section} \numberwithin{prop}{section}
\numberwithin{cor}{section} \numberwithin{rem}{section}
\title[Field Model for Complex Ionic Fluids]{Field Model for Complex Ionic Fluids: Analytical Properties and Numerical Investigation}

\author{Jian-Guo Liu$^{\,1}$, Jinhuan Wang$^{\,2}$, Yu Zhao$^{\,3}$ and Zhennan Zhou$^{\,4}$}
\thanks{The work of Jian-Guo Liu was partially supported by KI-Net NSF RNMS grant No.1107291, NSF	DMS  grant No. 1514826.}
\thanks{The work of Jinhuan Wang was partially supported by Key Project of Education Department of Liaoning Province (Grant No. LZD201701).}
\thanks{Corresponding author: Yu Zhao.  }
\begin{document}
	\maketitle
	\begin{center}
		{\footnotesize
			1-Department of Physics and Department of Mathematics, Duke University, \\
			Durham, NC 27708. USA.
			email: jliu@phy.duke.edu \\
			\smallskip
			2-School of Mathematics, Liaoning University, Shenyang, 110036, P. R. China. \\
			email: wangjh@lnu.edu.cn \\
			\smallskip
			3-School of Mathematical Sciences, Peking University, Beijing, 100871, P.R.China. \\  
			email: y.zhao@pku.edu.cn \\
			\smallskip
			4-Beijing International Center for Mathematical Research, Peking University, Beijing,\\ 100871, P.R.China. 
			email: zhennan@bicmr.pku.edu.cn 
		}
	\end{center}
	\maketitle
	\date{}
	\begin{abstract}
	In this paper, we consider the field model for complex ionic fluids with an energy variational structure, and analyze the well-posedness to this model with regularized kernels. Furthermore, we deduce the estimate of the maximal density function to quantify the finite size effect. On the numerical side, we adopt a finite volume scheme to the field model, which satisfies the following properties: positivity-preserving, mass conservation and energy dissipation. Besides, series of numerical experiments are provided to demonstrate the properties of the steady state and the finite size effect  by showing the equilibrium profiles with different values of the parameter in the kernel.
	\end{abstract}
	
	
	\section{Introduction}
	\label{sec1}
	Nearly all biological processes are related to ions \cite{liu2010}. The electrokinetic system for ion transport in solutions is an important model in medicine and biology \cite{Alverts1994,Boron2008}. The transport and distribution of charged particles are crucial in the study of many physical and biological problems, such as ion particles in the electrokinetic fluids \cite{Jin2007}, and ion channels in cell membranes \cite{Bazant2004,Eisenberg1996}. In this paper, we consider the field equations for complex ionic fluids derived from an energetic variational method EnVarA (energy variational analysis) which combines Hamilton's least action and Rayleigh's dissipation principles to create a variational field theory \cite{liu2010}.
	
	In EnVarA, the free energy of the field systems which is denoted by $\mathcal{F}$ for complex ionic fluids is written in the Eulerian framework 
	\begin{equation}
	\label{free1}
	\mathcal{F}(c_m(\cdot,  t)) = \int_{\Omega} \left\{ \sum_{m = 1}^M c_m \log c_m + \phi_{\text{ES}}(\cdot) + \psi_{\text{FSE}}(\cdot) \right\} \,\mathrm{d} \boldsymbol{x},
	\end{equation}
	where $c_m = c_m(\boldsymbol{x}, t)$, $m=1,...,M$, be the concentration of the m-th ionic species where $\boldsymbol{x} \in \Omega \subset \mathbb{R}^d$ indicates the location and $t > 0$ indicates the time \cite{liu2010}. The first part of the right hand of (\ref{free1}) is the entropy term which describes the particle Brownian motion of the ions. And the second part $\phi_{\text{ES}}(\cdot)$ is the electrostatic potential where the electric field is created by the charge on different ionic species in most cases we considered. In addition, we focus on the steric repulsion arising from the finite size of solid ions\cite{fawcett2004liquids,Lee2008,KS1991,barthel1998physical}, which is the last term of (\ref{free1}). Here all physical parameters are set as 1 for simplicity in representation. Furthermore, additional free energy due to physical effects such as screening\cite{jnc1999} can also be included in (\ref{free1}), which leads to different field equations. The field equations might either be defined on the whole domain $\mathbb{R}^n$ or a bounded domain $\Omega$ equipped with certain physical boundary conditions. However, proposing an appropriate boundary condition is a task of great difficulty as well as an interesting research subject. In this paper, we consider only the unbounded domain $\mathbb{R}^n$ and focus on the generalized field model. We remark that, there have been other ways of modeling ionic and water flows when considering voids, polarization of water, and ion-ion and ion-water correlations \cite{Eisenberg2020, Eisenberg2017}.
	
	The chemical potential $\psi_m$ of the m-th ionic species is described by the variational derivative
	\begin{equation}
	\label{potential}
	\psi_m = \dfrac{\delta \mathcal{F}(c_m(\cdot, t))}{\delta c_m}
	\end{equation}
	and is referred to in channel biology as the "driving force" for the current of the m-th ionic species\cite{liu2010}. Then EnVarA gives us both the equilibrium and the non-equilibrium (time dependent) equations for complex ionic fluids as follows,
	\begin{align}
	& \text{equilibrium : } 0 = \nabla \cdot (c_m \nabla \psi_m), ~~ m = 1, \cdots, M, 
	\label{equ} \\
	& \text{nonequilibrium (time dependent): } \partial_t c_m = \nabla \cdot (c_m \nabla \psi_m), ~~ m = 1, \cdots, M.
	\label{nonequ} 
	\end{align}
	
	In this paper, we consider the steric repulsion in the following form $$\psi_{\text{FSE}}(\cdot) = \frac{1}{2} \theta(\boldsymbol{x})(\mathcal{W}*\theta)(\boldsymbol{x}),$$
	where the total density $\theta(\boldsymbol{x}) := \sum_{m = 1}^{M} c_m$ and the electrostatic potential 
	$$\phi_{\text{ES}}(\cdot) = \frac{1}{2} \rho(\boldsymbol{x}) (\mathcal{K}*\rho)(\boldsymbol{x}),$$ 
	where the charge density $\rho(\boldsymbol{x}) := \sum_{m = 1}^{M} z_m c_m$ with $z_m \in \mathbb{Z}$ being the valence of the $m$-th ionic species. Then the free energy (\ref{free1}) of the field model for complex ionic fluids is given by the following functional
	\begin{align}
	\label{free2}
	\mathcal{F}(c_m(\cdot,t))
	=&\sum_{m=1}^M \int_{\mathbb{R}^d} c_m\log c_m\,\mathrm{d} \boldsymbol{x}
	+\frac{1}{2}\int_{\mathbb{R}^d}\rho(\boldsymbol{x}) (\mathcal{K}*\rho)(\boldsymbol{x})\,\mathrm{d} \boldsymbol{x}
	+\frac{1}{2}\int_{\mathbb{R}^d}\theta(\boldsymbol{x})(\mathcal{W}*\theta)(\boldsymbol{x})\,\mathrm{d}\boldsymbol{x}.
	\end{align}
	Kernel $\mathcal{K}(\boldsymbol{x})$ in (\ref{free2}) represents the effect of the  electrostatic potential while kernel $\mathcal{W}(\boldsymbol{x})$ represents the effect of the steric repulsion arising from the small size. Here, explicit write out (\ref{nonequ}) using $\mathcal{K}(\boldsymbol{x}), \mathcal{W}(\boldsymbol{x}), \theta(\boldsymbol{x})$ and $\rho(\boldsymbol{x})$ etc. In fact, the convolution terms make it difficult to derive explicit differential equations between the field functions and the charge density $\rho(\boldsymbol{x})$ except when the kernel is Newtonian, i.e. 
	\begin{equation}
	\label{K}
	\mathcal{K}(\boldsymbol{x}) = 
	\left\{
	\begin{array}{ll}
	-\dfrac{1}{2 \pi} \ln |\boldsymbol{x}|, & d = 2, \\
	\dfrac{1}{d (d-2) \alpha(d) |\boldsymbol{x}|^{d - 2}}, & d \geqslant 3,
	\end{array}
	\right.
	\end{equation}
	in which case, the electrostatic field potential function $\Phi_{\mathcal{K}}(\boldsymbol{x})$ related to the concentrations of the ions is determined by Gauss's law, i.e.
	\begin{align}
	\label{phik}
	-\Delta \Phi_{\mathcal{K}}(\boldsymbol{x}) = \rho(\boldsymbol{x}).
	\end{align}
	Whereas, for the steric repulsion there is no clear way to reformulate the convolution with the help of an auxiliary potential equation. Hence, in this work, without loss of generality, we focus on the Cauchy problem of the field model for complex ionic fluids and the initial conditions can be given as follows,
	\begin{equation}
	\label{recini}
	c_m(\boldsymbol{x}, 0)=c^0_{m}(\boldsymbol{x}), ~~m=1,...,M.
	\end{equation}
	We aim to investigate the transport of ions modeled by EnVarA both theoretically and numerically.
	
	It is worth emphasizing that the Poisson-Nernst-Planck (PNP) equations, which are widely used by many channologists \cite{Eisenberg1996} to describe the transport of ions through ionic channels and by physical chemists \cite{Bazant2004}, can be derived by such variational method as well, simply by setting $\psi_{\text{FSE}}(\cdot) = 0$. The PNP equations describe the transport of an ideal gas of point charges. However, due to the lack of incorporating the nonideal properties of ionic solutions, the PNP equations can not describe electrorheological fluids containing charged solid balls or some other complex fluids in biological applications properly.
	
	In contrast to the limited studies and the partial understanding of the Cauchy problem of the field model (\ref{nonequ})(\ref{recini}), on one hand, as for the nonlinear nonlocal equations with a gradient flow structure, Carrillo, Chertock and Huang  \cite{Carrillo2014} proposed a positivity preserving entropy decreasing finite volume scheme. On the other hand, for the initial boundary value problem of the PNP equations without small size effect, there have been quit a few numeric studies in the recent years. 
	For instance, Liu and Wang \cite{Liu2014} have designed and analyzed a free energy satisfying finite difference method for solving PNP equations in a bounded domain that are conservative, positivity preserving and of the first order in time and the second order in space. 
	Later, a discontinuous Galerkin method for the one-dimensional PNP equations \cite{Liu2017} has been proposed. Both of them satisfy the positivity preserving property and the discrete energy decay estimate under a parabolic CFL condition $\Delta t = O((\Delta x)^2)$. 
	Furthermore, Flavell, Kabre and Li \cite{Flavell2017} have proposed a finite difference scheme that captures exactly (up to roundoff error) a discrete energy dissipation and which is of the second order accurate in both time and space. 
	Besides, a finite element method using a method of lines approached developed by Metti, Xu and Liu \cite{Metti2016} enforces the positivity of the computed solutions and obtains the discrete energy decay but works for the certain boundary while the scheme developed by Hu and Huang \cite{hu2019} works for the general boundaries. etc.  
	
	In this paper, besides the basic properties of the equilibrium and non-equilibrium state of the field model (\ref{nonequ})(\ref{recini}), such as positivity-preserving, mass conservation and free energy dissipation, we also analyze the existence of the solutions to this model and the properties of the steady state and estimate the maximal density function to quantify the finite size effect theoretically, while to supplement this, reliable numerical simulations are necessary to explore such phenomenon.
	We consider a finite volume scheme to the field model (\ref{nonequ})(\ref{recini}) to preserve the basic physical properties of the ionic fluid equations. The small size effect can also be demonstrated by such numerical scheme. The scheme is of the first order in time and the first order in space while the generalization to higher order schemes in space is of no difficulty. Also higher order in time can be obtained by the strong stability preserving (SSP) Runge-Kutta methods. Another challenge in numerical simulation of our field model (\ref{nonequ})(\ref{recini}) is the handling for the high order singularity of the kernel $\mathcal{K}(\boldsymbol{x})$ and $\mathcal{W}(\boldsymbol{x})$. Here we just deal with the singularity preliminarily so as to grasp the effect of the finite size effect. More appropriate methods will only be discussed in future papers.
	
	When considering the way of modeling ionic and water flows in \cite{Eisenberg2020}, the correlated electric potential is obtained by making some modifications to the electrostatic potential function $\Phi_{\mathcal{K}}(\boldsymbol{x})$ and taking $\mathcal{W} = 0$ at the same time. We show the proposed scheme also applies to such a modeling scenario with little extra effort, and some preliminary numerical explorations are provided.
	
	The rest of the paper is organized as follows. The field model for complex ionic fluids we considered in this paper is recalled and analyzed in Chapter 2. We show the basic properties of the Cauchy problem of the field model (\ref{equ}) or (\ref{nonequ})(\ref{recini}). Furthermore, the well-posedness of the model (\ref{nonequ})(\ref{recini}) is captured when we take the electrostatic potential $\phi_{\text{ES}}(\cdot)$ and the repulsion of finite size effect $\psi_{\text{FSE}}(\cdot)$ as Newtonian and Lennard-Jones form respectively. In Chapter 3, we consider a finite volume scheme to the field system (\ref{nonequ})(\ref{recini}) in 1D in the semi-discrete level, and prove its properties : positivity-preserving, mass conservation and discrete free energy dissipation. In the same section, the fully discrete scheme, and the extension to the 2D cases are also discussed. In Chapter 4 we verify the properties of our numerical method with numerous test examples and provide series of numerical experiments to demonstrate the small size effect in the model. Concluding remarks and the expectations in the following research are given in Chapter 5.

	\section{The Field System And Its Properties}
\label{sec2}
Considering the free energy functional (\ref{free2}), the chemical potential $\psi_m$ of the m-th ionic species is described by the variational derivative (\ref{potential}) which is calculated as follows:
\begin{align}
\label{potential1}
\psi_m = \dfrac{\delta \mathcal{F}(c_m(\cdot, t))}{\delta c_m} &= 1 + \log c_m + z_m \Phi_{\mathcal{K}}(\boldsymbol{x}) + \Phi_{\mathcal{W}}(\boldsymbol{x}), \\
&= 1 + \log c_m + z_m (\mathcal{K}*\rho)(\boldsymbol{x}) + (\mathcal{W}*\theta)(\boldsymbol{x}), ~~m = 1, \cdots, M.
\end{align} 
With proper initial conditions (\ref{recini}), which we recall here for convenience, the field system for complex ionic fluids are given by
\begin{align}
\partial_t c_m(\boldsymbol{x}, t) &= \nabla \cdot (c_m \nabla \left( 1 + \log c_m + z_m (\mathcal{K}*\rho)(\boldsymbol{x}) + (\mathcal{W}*\theta)(\boldsymbol{x}) \right)),~~m=1,\cdots,M, \label{model1} \\
\rho(\boldsymbol{x}) &= \sum_{m = 1}^{M} z_m c_m, ~~\theta(\boldsymbol{x}) = \sum_{m = 1}^{M} c_m, \\
c_m(\boldsymbol{x},0) &= c^0_{m}(\boldsymbol{x}), ~~m=1,\cdots,M. \label{model2}
\end{align} 

\subsection{Basic Properties for the Multi-ionic Species Case}
Here we show some properties of the field model (\ref{model1})-(\ref{model2}) for complex ionic fluids.

The first two properties are related to the positivity-preserving and the conservation of mass.
\begin{prop}\label{prop0}(Positivity-Preserving)
	Let initial data $c_m^0$, $m = 1,\cdots,M$, be non-negative functions. Then solutions $c_m$ to (\ref{model1})-(\ref{model2}) are still non-negative.
\end{prop}
\begin{prop}\label{prop1}(Mass conservation)
	Let $c_m$, $m = 1,\cdots,M$, be non-negative solutions to (\ref{model1})-(\ref{model2}). Then the field model has the following conservation of mass
	\begin{align}\label{masscon}
	\int_{\mathbb{R}^d}c_m(\boldsymbol{x},t) \,\mathrm{d} \boldsymbol{x} \equiv\int_{\mathbb{R}^d}c^0_m(\boldsymbol{x}) \,\mathrm{d} \boldsymbol{x} = :\bar{m}^m_0, ~~\sum_{m=1}^M \bar{m}^m_0=:\bar{m}_0.
	\end{align}
\end{prop}
Here the notation $\bar{m}^m_0$ represents the mass of the $m$-th ionic species and $\bar{m}_0$ represents the total mass of all kinds of the ionic species for $m = 1, \cdots, M$.

The proofs for these properties above are standard, which we omit in this paper.
And the third property is to give the energy-dissipation relation for total free energy.
\begin{prop}\label{prop2}(Free energy-dissipation relation)
	Let $c_m$, $m=1,\cdots,M,$ be solutions to (\ref{model1})-(\ref{model2}). Then the following energy-dissipation relation holds that
	\begin{align}
	\frac{\mathrm{d}}{\mathrm{d}t} \mathcal{F}(c_m(\cdot, t))(t) + D = 0.\label{free0}
	\end{align}
	where the dissipation
	\begin{align}
	\label{disspation}
	D = \sum_{m=1}^M \int_{\mathbb{R}^d} c_m \big| \nabla \psi_m \big|^2 \,\mathrm{d} \boldsymbol{x}.
	\end{align}
\end{prop}
\begin{proof}
	In fact, we only need to take $\frac{\delta \mathcal{F}(c_m(\cdot, t))}{\delta c_m}$ as a test function in the both side of (\ref{nonequ}). Consequently, using integration by parts, we have
	\begin{equation}
	\frac{\mathrm{d}}{\mathrm{d}t}\mathcal{F}(c_m(\cdot, t))(t) = \sum_{m = 1}^{M} \int_{\mathbb{R}^d} \dfrac{\partial c_m}{\partial t} \psi_m \,\mathrm{d} \boldsymbol{x} = - \sum_{m = 1}^{M} \int_{\mathbb{R}^d} c_m
	\big| \nabla \psi_m \big|^2 \,\mathrm{d} \boldsymbol{x} \leqslant 0.
	\end{equation}
\end{proof}
Next four equivalent statements for the steady solutions are shown.
\begin{prop}\label{prop3}(Four equivalent statements for the positive steady state)
	Assuming that $\bar C_m \in L^1\cap L\log L$ is bounded with $\int_{\mathbb{R}^d} \bar C_m \,\mathrm{d} \boldsymbol{x} = M$, $\bar C_m \in C(\mathbb{R}^d)$, $\bar C_m > 0$ in $\mathbb{R}^d$ and $\bar C_m$ decays at infinity for all $m$. Then the following four statements are equivalent:
	\begin{itemize}
		\item Equilibrium (definition of weak steady solutions): $\bar\psi_m \in \dot{H}^1(\mathbb{R}^d)$ and $\nabla \cdot (\bar C_m \nabla \bar \psi_m)$ $= 0$ in $H^{-1}(\mathbb{R}^d)$, $\forall ~m = 1, \cdots, M$, where $\bar\psi_m = 1 + \log \bar C_m + z_m \mathcal{K}*\bar \rho + \mathcal{W}* \bar\theta$, $~~ \bar \rho = \sum_{m = 1}^M z_m \bar C_m $, $~~ \bar \theta = \sum_{m = 1}^M \bar C_m$.
		\item No dissipation: $\sum_{m = 1}^{M}\int_{\mathbb{R}^d} \bar C_m|\nabla \bar\psi_m|^2 \,\mathrm{d} \boldsymbol{x} = 0$.
		\item $\left( \bar C_1, \cdots, \bar C_m \right)$ is a critical point of $\mathcal{F}(c_m(\cdot,t))$.
		\item $\bar\psi_m$ is a constant, $\forall~ m = 1, \cdots, M$.
	\end{itemize}
	
\end{prop}

\begin{proof}
	At first, we prove (i)$\Rightarrow$(ii). Since $\bar{\psi}_{m} \in H^{1}(\mathbb{R}^{d}), \nabla \cdot\left( \bar C_{m} \nabla \bar{\psi}_{m} \right)=0$ in $H^{-1}\left(\mathbb{R}^{d}\right)$, $ C_{0}^{\infty}(\mathbb{R}^{d})$ is dense in $\dot{H}^1(\mathbb{R}^{d})$ and $\bar C_m$ is bounded, one has
	\begin{equation}
	0 = \int_{\mathbb{R}^{d}} \bar{\psi}_{m} \nabla \cdot\left( \bar{C}_{m} \nabla \bar{\psi}_{m} \right) \,\mathrm{d} \boldsymbol{x} = -\int_{\mathbb{R}^{d}} \bar{C}_{m}\left| \nabla \bar{\psi}_{m}\right|^{2} \,\mathrm{d} \boldsymbol{x}, \quad \forall ~m = 1, \cdots, M.
	\end{equation}
	Hence (ii) holds.
	
	Next we prove (iii)$\Leftrightarrow$(iv). Notice that $\bar C_m $ is a critical point of $\mathcal{F}(c_m(\cdot,t))$ if and only if
	\begin{equation}
	\left.\frac{\mathrm{d}}{\mathrm{d} \varepsilon}\right|_{\varepsilon=0} \mathcal{F}\left(\bar C_{m} + \varepsilon \phi\right)=0, \quad \forall \phi \in C_{0}^{\infty}(\mathbb{R}^{d}) \text {  with } \int_{\mathbb{R}^{d}} \phi(x) \,\mathrm{d} \boldsymbol{x} = 0.
	\end{equation}
	Equivalently,
	\begin{equation}
	\int_{\mathbb{R}^{d}} \bar{\psi}_{m} \phi \,\mathrm{d} \boldsymbol{x} = 0, \quad \forall \phi \in C_{0}^{\infty}(\mathbb{R}^{d}),
	\end{equation}
	which implies $\bar\psi_m $ is a constant, $\forall~ m = 1, \cdots, M$.
	
	Then we prove (ii)$\Rightarrow$(iv). Suppose $\sum_{m = 1}^{M} \int_{\mathbb{R}^d}\bar C_m|\nabla \bar\psi_m|^2 \,\mathrm{d} \boldsymbol{x} = 0$. It follows from $\bar C_m > 0$ at any point
	$\boldsymbol{x}_{0} \in \mathbb{R}^{d}$ that $\nabla \bar{\psi}_{m}=0$ in $\mathbb{R}^{d}$ and thus $\bar {\psi}_m$ is a constant for all $m = 1, \cdots, M$.
	
	Hence we complete the proof for (ii)$\Rightarrow$(iii) and (iii)$\Rightarrow$(iv).
	
	Finally we prove (iv) $\Rightarrow$ (i). Since $\bar{\psi}_m$ is a constant in $\mathbb{R}^{d}$, (i) is a direct consequence of (iv).
\end{proof}

\subsection{Well-posedness with the Regularized Kernels}

In this subsection, the well-posedness of the field model (\ref{model1})-(\ref{model2}) is presented provided that we describe the inter-ion repulsive force by the regularized Lennard-Jones type potential and the electrostatic force by the regularized Newtonian potential. Specifically, with constant parameters $a > 0$ and $\eta > 0$, we set the kernel $\mathcal{K}(\boldsymbol{x})$ and $\mathcal{W}(\boldsymbol{x})$ in the following form
\begin{equation}
\label{kandw11}
\mathcal{K}(\boldsymbol{x}) = \mathcal{K}_a(\boldsymbol{x}) :=
\left\{
\begin{array}{ll}
-{\dfrac{1}{2}}\log\left(|\boldsymbol{x}|^2+a^2\right), & d = 2, \\
\dfrac{1}{\left(|\boldsymbol{x}|^2 + a^2 \right)^{\frac{d - 2}{2}}},
&d>2,
\end{array}
\right.
\end{equation}
and
\begin{equation}
\label{kandw}
\mathcal{W}(\boldsymbol{x}) = \mathcal{W}_a(\boldsymbol{x}) :=
\dfrac{\eta}{\left(|\boldsymbol{x}|^2 + a^2 \right)^{\frac{k}{2}}},  d>2, ~~d-2 \leqslant k < d.
\end{equation}
By classical parabolic theory, we know that there is a global smooth solution for the field model (\ref{model1})-(\ref{model2}) with the kernels $\mathcal{K}(\boldsymbol{x})$ and $\mathcal{W}(\boldsymbol{x})$ defined by (\ref{kandw11}) or (\ref{kandw}),  which is given by the following theorem without the proof.
\begin{thm}(Existence for the multi-ionic species case)\label{thm3}
	Assume that $c_m^0 \in L^1_+ \cap L\log L(\mathbb{R}^d)$, $\sigma_2(0) < \infty$ and $\mathcal{F}(0)<\infty$. Then for any $T>0$, there is a global smooth solution $(c_1, \cdots, c_M)$ to the field model (\ref{model1})-(\ref{model2}).
\end{thm}

The next property for the regularized field model (\ref{model1})-(\ref{model2}) is concerned with the boundedness of the second moment which is essential for showing the tightness of $c_m$, $m = 1,\cdots,M$. Here
$$
\sigma_2(t)=\sum_{m = 1}^ M \sigma_2^m(t) = \sum_{m = 1}^M \int_{\mathbb{R}^d}|\boldsymbol{x}|^2 c_m \,\mathrm{d} \boldsymbol{x}.
$$
\begin{prop}(Boundness of the second moment)\label{boundednessm2}
	Let $c_m$, $m = 1,\cdots,M$, be non-negative solutions to (\ref{model1})-(\ref{model2}).
	If $\sigma_2(0)<\infty$, then we have
	\begin{equation}
	\sigma_2(t) \leqslant Ct, ~~ \mathrm{for} ~~d \geqslant 2, \label{mom2}
	\end{equation}
	where $C$ is a constant that only depends on $d, k, \eta, a, z_m, m = 1, \cdots, M,$ and the initial data.
\end{prop}

\begin{proof}
	In fact, taking $|\boldsymbol{x}|^2$ as a test function in the equations of $c_m$, integrating them in $\mathbb R^d$, we have
	\begin{align}
	&\frac{\mathrm{d}}{\mathrm{d}t} \int_{\mathbb{R}^d} |\boldsymbol{x}|^2 c_m \,\mathrm{d} \boldsymbol{x} = 2 d \bar{m}_0 ^m - 2 \int_{\mathbb{R}^d} \boldsymbol{x} \cdot (\nabla\mathcal{K}_a*\rho )(\boldsymbol{x}) z_m c_m \, \mathrm{d} \boldsymbol{x} - 2 \int_{\mathbb{R}^d} \boldsymbol{x} \cdot (\nabla\mathcal{W}_a*\theta )(\boldsymbol{x}) c_m \,\mathrm{d} \boldsymbol{x},\label{c+sec}
	\end{align}
	where $\mathcal{K}_a(\boldsymbol{x})$ and $\mathcal{W}_a(\boldsymbol{x})$ are defined in (\ref{kandw11}) or (\ref{kandw}).
	Summing (\ref{c+sec}) for $m$ from 1 to $M$, we obtain
	\begin{align}
	\frac{\mathrm{d}}{\mathrm{d}t}  \sigma_2(t) = 2 d \bar{m}_0 - 2 \int_{\mathbb{R}^d} \boldsymbol{x} \cdot (\nabla\mathcal{K}_a*\rho)(\boldsymbol{x}) \rho(\boldsymbol{x}) \,\mathrm{d} \boldsymbol{x} - 2 \int_{\mathbb{R}^d} \boldsymbol{x} \cdot (\nabla\mathcal{W}_a*\theta )(\boldsymbol{x}) \theta(\boldsymbol{x}) \,\mathrm{d} \boldsymbol{x}.
	\end{align}
	
	(1) For $d \geqslant 3$, notice that
	\begin{align}
	\nabla \mathcal{K}_a(\boldsymbol{x}) = - (d-2) \frac{\boldsymbol{x}}{(|\boldsymbol{x}|^2 + a^2)^{\frac{d}{2}}},\quad \nabla \mathcal{W}_a(\boldsymbol{x}) = -k \eta \frac{\boldsymbol{x}}{(|\boldsymbol{x}|^2+a^2)^{\frac{k}{2}+1}}.
	\end{align}
	By the symmetry of the potentials, it follows
	\begin{align*}
	-2\int_{\mathbb{R}^d} \boldsymbol{x} \cdot (\nabla\mathcal{K}_a*\rho)(\boldsymbol{x}) \rho(\boldsymbol{x}) \,\mathrm{d} \boldsymbol{x} &= (d-2) \int_{\mathbb{R}^d} \int_{\mathbb{R}^d} \frac{|\boldsymbol{x} - \boldsymbol{y}|^2 \rho(\boldsymbol{x}) \rho(\boldsymbol{y})}{(|\boldsymbol{x} - \boldsymbol{y}|^2+a^2)^\frac{d}{2}} \,\mathrm{d} \boldsymbol{y} \,\mathrm{d} \boldsymbol{x}, \\
	-2\int_{\mathbb{R}^d} \boldsymbol{x} \cdot (\nabla\mathcal{W}_a*\theta)(\boldsymbol{x}) \theta(\boldsymbol{x}) \,\mathrm{d} \boldsymbol{x} &= k \eta \int_{\mathbb{R}^d}\int_{\mathbb{R}^d} \frac{|\boldsymbol{x} - \boldsymbol{y}|^2 \theta(\boldsymbol{x}) \theta(\boldsymbol{y})}{(|\boldsymbol{x} - \boldsymbol{y}|^2+a^2)^{\frac{k}{2}+1}} \,\mathrm{d} \boldsymbol{y} \,\mathrm{d} \boldsymbol{x}.
	\end{align*}
	Thus,
	\begin{align}\label{11}
	\frac{\mathrm{d}}{\mathrm{d}t} \sigma_2(t)\leqslant 2d \bar{m}_0 + \left(k \eta a^{-k}+(d-2)a^{2-d} \max\{|z_1|,..., |z_M|\}^2 \right) \left( \bar{m}_0 \right)^2 \leqslant C.
	\end{align}
	
	(2) For $d = 2$, noticing that
	\begin{align*}
	-2 \int_{\mathbb{R}^2} \boldsymbol{x} \cdot (\nabla \mathcal{K}_a*\rho)(\boldsymbol{x}) \rho(\boldsymbol{x}) \,\mathrm{d} \boldsymbol{x} &= \int_{\mathbb{R}^2}\int_{\mathbb{R}^2} \frac{|\boldsymbol{x} - \boldsymbol{y}|^2 \rho(\boldsymbol{x}) \rho(\boldsymbol{y})}{(|\boldsymbol{x} - \boldsymbol{y}|^2 + a^2)} \,\mathrm{d} \boldsymbol{y} \,\mathrm{d} \boldsymbol{x} \\
	-2 \int_{\mathbb{R}^d} \boldsymbol{x} \cdot (\nabla \mathcal{W}_a*\theta)(\boldsymbol{x}) \theta(\boldsymbol{x}) \,\mathrm{d} \boldsymbol{x} &= k \eta \int_{\mathbb{R}^d} \int_{\mathbb{R}^d} \frac{|\boldsymbol{x} - \boldsymbol{y}|^2 \theta(\boldsymbol{x})\theta(\boldsymbol{y})}{(|\boldsymbol{x} - \boldsymbol{y}|^2+a^2)^{\frac{k}{2}+1}} \,\mathrm{d} \boldsymbol{y} \,\mathrm{d} \boldsymbol{x},
	\end{align*}
	we have
	\begin{align}\label{12}
	\frac{\mathrm{d}}{\mathrm{d}t} \sigma_2(t) \leqslant 4 \bar{m}_0 + (k \eta + 1) \max \{|z_1|,..., |z_M|\}^2 \left(\bar{m}_0\right)^2.
	\end{align}
	Hence (\ref{11}) and (\ref{12}) imply that (\ref{mom2}) holds.
\end{proof}

Using Proposition \ref{boundednessm2} and the free energy-dissipation relation (\ref{free0}), we can also provide the estimate on the maximal density function. A maximal density function defined by
DiPerna and Majda \cite{DM} associated to a measure $u(\boldsymbol{x})$ is given by
$$
M_{r} (u) = \sup_{\boldsymbol{x}, t} \int_{B(\boldsymbol{x},r)} u(\boldsymbol{y}) \,\mathrm{d}\boldsymbol{y}.
$$
Then the estimate on the maximal density functions $M_{r}(c_m)$, $m = 1,\cdots,M$, is in the following lemma.
\begin{lem}\label{lemma1}(Estimate on the maximal density function)
	Assume that $c_m^0 \in L^1_+ \cap L\log L(\mathbb{R}^d)$, $m=1, \cdots, M$, $\mathcal{F}(0) < \infty$ and $\sigma_2(0) < \infty$, then we have
	\begin{align}
	\sum_{m = 1}^{M} M_{r}(c_m(\cdot, t))\leqslant C 
	((2r)^2 + a^2)^{\frac{k}{4}},
	\end{align}
	where $C$ is a constant dependent on the initial data.
\end{lem}
\begin{proof}
	Since
	\begin{align}
	\frac{1}{((2r)^{2} + a^2)^{\frac{k}{2}}}\left(\int_{B(\boldsymbol{x}, r)} c_m(\boldsymbol{y}, t)\,\mathrm{d}\boldsymbol{y}\right)^2
	&\leqslant \int_{B(\boldsymbol{x}, r)}\int_{B(\boldsymbol{z}, r)}\frac{1}{((2r)^{2} + a^2)^{\frac{k}{2}}} c_m(\boldsymbol{y}, t) c_m(\boldsymbol{z}, t) \,\mathrm{d}\boldsymbol{y} \,\mathrm{d}\boldsymbol{z}\nonumber\\
	&\leqslant \int_{\mathbb{R}^d\times \mathbb{R}^d}\frac{1}{((\boldsymbol{y} - \boldsymbol{z})^{2} + a^2)^{\frac{k}{2}}}c_m(\boldsymbol{y},t) c_m(\boldsymbol{z}, t) \,\mathrm{d}\boldsymbol{y} \,\mathrm{d}\boldsymbol{z}.
	\label{c2}
	\end{align}
	Using the energy-dissipation relation (\ref{free0}) and the property of the second moment, we have
	$$
	\int_{\mathbb{R}^d\times \mathbb{R}^d}\frac{1}{((\boldsymbol{y} - \boldsymbol{z})^{2} + a^2)^{\frac{k}{2}}}c_m(\boldsymbol{y},t) c_m(\boldsymbol{z}, t) \,\mathrm{d}\boldsymbol{y} \,\mathrm{d}\boldsymbol{z} \leqslant C,
	$$
	where $C$ is a constant dependent on the initial data. Hence by (\ref{c2}), we have
	\begin{align*}
	\int_{B(\boldsymbol{x},r)} c_m(\boldsymbol{y},t)d\boldsymbol{y} \leqslant C 
	((2r)^2 + a^2)^{\frac{k}{4}}.
	\end{align*}
\end{proof}
Lemma \ref{lemma1} shows that qualitatively as k increases, the small size effect is stronger. However, as the estimates are not sharp, nor feasible quantitative measurements, we shall numerically explore such phenomenon.
	
	\section{Numerical Schemes}
\label{sec3}
In this section, we propose the first-order finite volume schemes both in space and time for the field model 
(\ref{model1})-(\ref{model2}) in one-dimension and two-dimension and prove the positivity preserving and entropy dissipation properties.
\subsection{First-order Scheme for One-dimensional Case}
Consider the computational domain as $[-L, L]$ and give the grid arrangement $-L = x_{-M_x - \frac{1}{2}} < x_{-M_x + \frac{1}{2}} < \cdots < x_{M_x - \frac{1}{2}} < x_{M_x + \frac{1}{2}} = L$.
Then we define the cell average of $c_m, \ m = 1, \cdots, M$, on cell $C_j = [x_{j - \frac{1}{2}}, x_{j + \frac{1}{2}}]$ of a small space mesh size $\Delta x_j$ as 
\begin{equation}
\bar{c}_{m, j}(t) = \frac{1}{\Delta x_j} \int_{C_j} c_m(x, t) \,\mathrm{d} x, ~~m = 1, \cdots, M,
\end{equation}
where $\Delta x_j = x_{j + \frac{1}{2}} - x_{j - \frac{1}{2}}$ and we set the maximum mesh size $\Delta x = \max_{j} \Delta x_j$.
A semi-discrete finite volume scheme can be given as
\begin{equation}
\label{sys2}
\frac{\mathrm{d} \bar{c}_{m, j}(t)}{\mathrm{d} t} = - \frac{F_{m, j + \frac{1}{2}}(t) - F_{m, j - \frac{1}{2}}(t)}{\Delta x_j}, ~~m = 1, \cdots, M,
\end{equation}
where the numerical flux $F_{m, j + \frac{1}{2}}$ is defined in the following form
\begin{equation}
\label{sys3}
F_{m, j + \frac{1}{2}}(t) = u^{+}_{m, j + \frac{1}{2}}(t) \bar{c}_{m, j}(t) - u^{-}_{m, j + \frac{1}{2}}(t) \bar{c}_{m, j + 1}(t).
\end{equation}
We denote the velocity $u_{m} = - \nabla \psi_m$,  then the discrete velocity $u_{m, j + \frac{1}{2}}$ in one-dimension can be denoted by the negative difference quotients of the discrete chemical potential $\psi_{m}$, which is
\begin{equation}
\label{sys4}
u_{m, j + \frac{1}{2}}(t) = - \dfrac{\psi_{m, j + 1}(t) - \psi_{m, j}(t)}{\Delta x_j}.
\end{equation}
$u_{m, j + \frac{1}{2}}$ equals its positive part minus the absolute value of its negative part, i.e.
\begin{equation}
u_{m, j + \frac{1}{2}} = u^{+}_{m, j + \frac{1}{2}} - u^{-}_{m, j + \frac{1}{2}}, 
\end{equation}
where the positive and the absolute value of the negative part of $u_{m, j + \frac{1}{2}}$, which are $u^{+}_{m, j + \frac{1}{2}}$ and $u^{-}_{m, j + \frac{1}{2}}$, can be written respectively as
\begin{equation}
\label{eq31d}
u^{+}_{m, j + \frac{1}{2}} = \max\left\{u_{m, j + \frac{1}{2}}, 0\right\}, \quad u^{-}_{m, j + \frac{1}{2}} = - \min\left\{u_{m, j + \frac{1}{2}}, 0\right\}.
\end{equation}
The discrete chemical potential $\psi_{m, j}$, the discrete charge density $\rho_{m, j}$ and the discrete total density $\theta_{m, j}$ are denoted respectively by
\begin{align}
\psi_{m, j} &= \log \bar{c}_{m, j} + 1 + \sum_{i} \Delta x_i \left[ z_m \mathcal{K}_{j - i} \rho_{i} + \mathcal{W}_{j - i} \theta_{i} \right], \\
\rho_{j} &= \sum_{m = 1}^M z_m \bar{c}_{m, j}, \\
\theta_{j} &= \sum_{m = 1}^M \bar{c}_{m, j},
\end{align}
where the discrete kernel $\mathcal{K}_{j - i} = \mathcal{K}(x_j - x_i)$ and $\mathcal{W}_{j - i} = \mathcal{W}(x_j - x_i)$.

It is worth stating that although we present the scheme for arbitrary grids, we implement with only uniform grids.

Next, we show the positivity preserving and entropy-dissipation properties of the one-dimensional semi-discrete finite volume scheme (\ref{sys2})-(\ref{sys4}).
\begin{thm}\label{thm1}(Positivity-Preserving)
	Consider the one-dimensional semi-discrete finite volume scheme (\ref{sys2})-(\ref{sys4}) of the system (\ref{model1})-(\ref{model2}) with initial data $c_{m}^0(x) \geqslant 0, ~\forall ~ m = 1, \cdots, M$. 
	If we discretize the ODEs system (\ref{sys2}) by the forward Euler method, which is
	\begin{align}
	\label{sysode}
	\frac{\bar{c}_{m, j}(t + \Delta t) - \bar{c}_{m, j}(t)}{\Delta t} = - \frac{F_{m, j + \frac{1}{2}}(t) - F_{m, j - \frac{1}{2}}(t)}{\Delta x_j}, ~~m = 1, \cdots, M.
	\end{align}
	Then, the cell averages $\bar{c}_{m, j} \geqslant 0, ~\forall ~ m = 1, \cdots, M, \,\forall ~ j$, provided that the following CFL condition is satisfied 
	\begin{equation}
	\label{cfl}
	\Delta t \leqslant \dfrac{\Delta x}{2 ~U_{\text{max}}}, 
	\end{equation}
	where $U_{\text{max}} = \max_{m, j} \left\{ u^{+}_{m, j + \frac{1}{2}} , u^{-}_{m, j - \frac{1}{2}} \right\}$, with $u^{+}_{m, j + \frac{1}{2}}$ and $u^{-}_{m, j + \frac{1}{2}}$ defined in (\ref{eq31d}).
\end{thm} 

\begin{proof}
	Take $\lambda_j = \Delta t/ \Delta x_j$, then $\lambda = \Delta t/ \Delta x = \min_{j} \lambda_j$.
	For all given $t > 0$, from (\ref{sysode}) we have 
	\begin{equation}
	\begin{aligned}
	\bar{c}_{m, j}(t + \Delta t) &= \bar{c}_{m, j}(t) - \lambda_j \left[ F_{m, j + \frac{1}{2}}(t) - F_{m, j - \frac{1}{2}}(t) \right] \\
	&= \bar{c}_{m, j}(t) - \lambda_j \left[ u^{+}_{m, j + \frac{1}{2}}(t) \bar{c}_{m, j}(t) - u^{-}_{m, j + \frac{1}{2}}(t) \bar{c}_{m, j + 1}(t) \right] \\
	&+ \lambda_j \left[ u^{+}_{m, j - \frac{1}{2}}(t) \bar{c}_{m, j - 1}(t) - u^{-}_{m, j - \frac{1}{2}}(t) \bar{c}_{m, j}(t) \right] \\
	&= \lambda_j u^{-}_{m, j + \frac{1}{2}}(t) \bar{c}_{m, j + 1}(t) + \lambda_j u^{+}_{m, j - \frac{1}{2}}(t) \bar{c}_{m, j - 1}(t) \\
	&+ \left( 1 - \lambda_j u^{+}_{m, j + \frac{1}{2}}(t) - \lambda_j u^{-}_{m, j - \frac{1}{2}}(t) \right) \bar{c}_{m, j}(t).
	\end{aligned}
	\end{equation}
	We can conclude that the cell average $\bar{c}_{m, j}(t + \Delta t) \geqslant 0, \, m = 1, \cdots, M, \,\forall ~j$ from the fact that both $u^{+}_{m, j + \frac{1}{2}}$ and $u^{-}_{m, j + \frac{1}{2}}$ for all $j$ are non-negative and that the CFL condition (\ref{cfl}) is satisfied.
\end{proof}


Next we calculate the discrete form of the free energy $\mathcal{F}$ defined in (\ref{free2}) and of the dissipation $D$ defined in (\ref{disspation}) by
\begin{equation}
\label{disenergy}
\begin{aligned}
E_{\Delta}(t) = \sum_{m = 1}^M \sum_{j} \Delta x_j \bar{c}_{m, j} \log \bar{c}_{m, j} + \dfrac{1}{2} \sum_{i, j} \Delta x_j \Delta x_i \left[ \mathcal{K}_{j - i} \rho_i \rho_j + \mathcal{W}_{j - i} \theta_i \theta_j \right], 
\end{aligned}
\end{equation}
and
\begin{equation}
\begin{aligned}
D_{\Delta}(t) &=  \sum_{m = 1}^{M} \sum_{j} \Delta x_j \left( u_{m, j + \frac{1}{2}} \right)^2 \min \left\{ \bar{c}_{m, j}, \bar{c}_{m, j + 1} \right\}.
\end{aligned}
\end{equation}

\begin{thm}\label{thm2}(Free energy-dissipation estimate)
	Consider the one-dimensional semi-discrete finite volume scheme (\ref{sys2})-(\ref{sys4}) of the system (\ref{model1})-(\ref{model2}) with initial data $c_{m}^0(x) \geqslant 0, \, m = 1, \cdots, M$. Assume that there is no flux boundary conditions on $[-L, L]$, i,e. the discrete boundary conditions satisfy $F_{m, -M_x - \frac{1}{2}} = F_{m, M_x + \frac{1}{2}} = 0, \, m = 1, \cdots, M$, where $L$ can be big enough. 
	Then we have
	\begin{equation}
	\dfrac{d}{d t} E_{\Delta}(t) \leqslant - D_{\Delta}(t) \leqslant 0, \quad \forall t \geqslant 0.
	\end{equation}
\end{thm}

\begin{proof}
	Differentiating (\ref{disenergy}) with respect to time, we have
	\begin{equation}
	\begin{aligned}
	&\dfrac{d}{d t} E_{\Delta}(t) \\
	&= \sum_{m = 1}^{M} \sum_{j} \Delta x_j \left( \log \bar{c}_{m, j} \dfrac{d}{d t} \bar{c}_{m, j} + \dfrac{d}{d t} \bar{c}_{m, j} \right) \\
	&+ \sum_{i, j} \Delta x_j \Delta x_i \left[ \mathcal{K}_{j - i} \left( \sum_{m = 1}^{M} z_m \bar{c}_{m, i} \right) \left( \sum_{m = 1}^{M} z_m \dfrac{d}{d t} \bar{c}_{m, i} \right) + \mathcal{W}_{j - i} \left( \sum_{m = 1}^{M} \bar{c}_{m, i} \right)  \left( \sum_{m = 1}^{M} \dfrac{d}{d t} \bar{c}_{m, i} \right) \right] \\
	&= \sum_{m = 1}^{M} \sum_{j} \Delta x_j \left[ 1 + \log \bar{c}_{m, j} +  \sum_{i} \Delta x_i \left[ z_m \mathcal{K}_{j - i} \left( \sum_{m = 1}^{M} z_m \bar{c}_{m, i} \right) + \mathcal{W}_{j - i} \left( \sum_{m = 1}^{M} \bar{c}_{m, i} \right) \right] \right] \dfrac{d}{d t} \bar{c}_{m, j} \\
	&= \sum_{m = 1}^{M} \sum_{j} \Delta x_j \psi_{m, j} \dfrac{d}{d t} \bar{c}_{m, j}.
	\end{aligned}
	\end{equation}
	According to (\ref{sys2}), we have 
	
	\begin{equation}
	\dfrac{d}{d t} E_{\Delta}(t) = - \sum_{m = 1}^{M} \sum_{j} \left[\psi_{m, j} (F_{m, j + \frac{1}{2}} - F_{m, j - \frac{1}{2}}) \right].
	\end{equation}
	Using Abel's summation formula, we obtain
	\begin{equation}
	\begin{aligned}
	\dfrac{d}{d t} E_{\Delta}(t) 
	=& -  \sum_{m = 1}^{M} \sum_{j} \left[(\psi_{m, j} - \psi_{m, j + 1}) F_{m, j + \frac{1}{2}} \right] \\
	=& - \sum_{m = 1}^{M} \sum_{j} \left[(\psi_{m, j} - \psi_{m, j + 1}) (u^{+}_{m, j + \frac{1}{2}} \bar{c}_{m, j} - u^{-}_{m, j + \frac{1}{2}} \bar{c}_{m, j + 1})\right] \\
	=& - \sum_{m = 1}^{M} \sum_{j} \Delta x_j \left[ u_{m, j + \frac{1}{2}} (u^{+}_{m, j + \frac{1}{2}} \bar{c}_{m, j} - u^{-}_{m, j + \frac{1}{2}} \bar{c}_{m, j + 1})\right], \\
	\leqslant& - \sum_{m = 1}^{M} \sum_{j} \Delta x_j \left( u_{m, j + \frac{1}{2}} \right)^2 \min \left\{ \bar{c}_{m, j}, \bar{c}_{m, j + 1} \right\} = - D_{\Delta}(t)
	\end{aligned}
	\end{equation}
	that is to say
	\begin{equation}
	\dfrac{d}{d t} E_{\Delta}(t) \leqslant - D_{\Delta}(t) \leqslant 0, \quad \forall t \geqslant 0.
	\end{equation}
\end{proof}

\subsection{First-order Scheme for Two-dimensional Case}
Similarly, define the cell average of $c_m, ~m = 1, \cdots, M$, on cell $C_{j, k} = [x_{j - \frac{1}{2}}, x_{j + \frac{1}{2}}] \times [y_{j - \frac{1}{2}}, y_{j + \frac{1}{2}}]$ of a small space mesh size $\Delta x_j$ and $\Delta y_k$ as 
\begin{equation}
\bar{c}_{m, j, k}(t) = \frac{1}{\Delta x_j \Delta y_k} \int_{C_{j, k}} c_m(x, y, t) \,\mathrm{d} x \mathrm{d} y, ~~m = 1, \cdots, M,
\end{equation}
where $\Delta x_j = x_{j + \frac{1}{2}} - x_{j - \frac{1}{2}}, \Delta y_k = y_{k + \frac{1}{2}} - y_{k - \frac{1}{2}}$ and we set the maximum mesh size $\Delta x = \max_{j} \Delta x_j, \Delta y = \max_{k} \Delta y_k$.
A semi-discrete finite volume scheme for two-dimensional case can be given as
\begin{equation}
\label{sys20}
\dfrac{d \bar{c}_{m, j, k}(t)}{d t} = - \dfrac{F^x_{m, j + \frac{1}{2}, k}(t) - F^x_{m, j - \frac{1}{2}, k}(t)}{\Delta x_j} - \dfrac{F^y_{m, j, k + \frac{1}{2}}(t) - F^y_{m, j, k - \frac{1}{2}}(t)}{\Delta y_k}, ~~ m = 1, \cdots, M,
\end{equation}
where the upwind flux is as follows,
\begin{equation}
\label{sys30}
\begin{aligned}
F^x_{m, j + \frac{1}{2}, k}(t) &= u^{+}_{m, j + \frac{1}{2}, k}(t) \bar{c}_{m, j, k}(t) - u^{-}_{m, j + \frac{1}{2}, k}(t) \bar{c}_{m, j + 1, k}(t),  \\
F^y_{m, j, k + \frac{1}{2}}(t) &= v^{+}_{m, j, k + \frac{1}{2}}(t) \bar{c}_{m, j, k}(t) - v^{-}_{m, j, k + \frac{1}{2}}(t) \bar{c}_{m, j, k + 1}(t).
\end{aligned}
\end{equation}
The discrete velocity $\left( u_{m, j + \frac{1}{2}, k}, v_{m, j, k + \frac{1}{2}} \right)^T$ are denoted by
\begin{equation}\label{sys40}
\begin{aligned}
u_{m, j + \frac{1}{2}, k} &= u^{+}_{m, j + \frac{1}{2}, k} - u^{-}_{m, j + \frac{1}{2}, k} = - \dfrac{\psi_{m, j + 1, k} - \psi_{m, j, k}}{\Delta x_j}, \\
v_{m, j, k + \frac{1}{2}} &= v^{+}_{m, j, k + \frac{1}{2}} - v^{-}_{m, j, k + \frac{1}{2}} = - \dfrac{\psi_{m, j, k + 1} - \psi_{m, j, k}}{\Delta y_k }, \\
\end{aligned}
\end{equation}
where the positive and the abstract of the negative part of $u_{m, j + \frac{1}{2}, k}$ and $v_{m, j, k + \frac{1}{2}}$ are described as before.

The two-dimensional discrete chemical potential $\psi_{m, j, k}$, the two-dimensional discrete charge density $\rho_{m, j, k}$ and the two-dimensional discrete total density $\theta_{m, j, k}$ are denoted respectively by
\begin{align}
\psi_{m, j, k} &= 1 + \log \bar{c}_{m, j, k} + \sum_{i,l} \Delta x_j \Delta y_k \left[ z_m \mathcal{K}_{j - i, k - l} \rho_{i, l} + \mathcal{W}_{j - i, k - l} \theta_{i, l} \right], \\
\rho_{j, k} &= \sum_{m = 1}^M z_m \bar{c}_{m, j, k}, \\
\theta_{j, k} &= \sum_{m = 1}^M \bar{c}_{m, j, k},
\end{align}
where the discrete kernel $\mathcal{K}_{j - i, k - l} = \mathcal{K}(x_j - x_i, y_k - y_l)$ and $\mathcal{W}_{j - i, k - l} = \mathcal{W}(x_j - x_i, y_k - y_l)$.

For the two-dimensional case, we give the positivity preserving and entropy dissipation properties. Since the proofs are quite similar to the one dimensional cases, we thus omit the proofs here.

\begin{thm}\label{thm33}(Positivity-Preserving)
	Consider the two-dimensional semi-discrete finite volume scheme (\ref{sys20})-(\ref{sys40}) of the system (\ref{model1})-(\ref{model2}) with initial data $c_{m}^0(x, y) \geqslant 0, \, m = 1, \cdots, M$. 
	If we discretize the ODEs system (\ref{sys20}) by the forward Euler method, which is for all $m = 1, \cdots, M,$
	\begin{align}
	\label{sysode0}
	\frac{\bar{c}_{m, j, k}(t + \Delta t) - \bar{c}_{m, j, k}(t)}{\Delta t} = - \dfrac{F^x_{m, j + \frac{1}{2}, k}(t) - F^x_{m, j - \frac{1}{2}, k}(t)}{\Delta x_j} - \dfrac{F^y_{m, j, k + \frac{1}{2}}(t) - F^y_{m, j, k - \frac{1}{2}}(t)}{\Delta y_k}.
	\end{align}
	Then, the cell averages $\bar{c}_{m, j, k} \geqslant 0, \, m = 1, \cdots, M, \,\forall ~ j, \forall ~ k$, provided that the following CFL condition is satisfied 
	\begin{equation}
	\label{cfl2}
	\Delta t \leqslant \max \left\{\dfrac{\Delta x}{4 ~ U_{\text{max}}}, \dfrac{\Delta y}{4 ~ V_{\text{max}}} \right\}
	\end{equation}
	where $U_{\text{max}} = \max_{m, j, k} \left\{ u^{+}_{m, j + \frac{1}{2}, k} , u^{-}_{m, j - \frac{1}{2}, k} \right\}, V_{\text{max}} = \max_{m, j, k} \left\{ v^{+}_{m, j, k + \frac{1}{2}} , v^{-}_{m, j, k - \frac{1}{2}} \right\}$.
\end{thm} 

Define the discrete form of the free energy $\mathcal{F}$ of the two-dimensional case defined in (\ref{free2}) and of the dissipation $D$ defined in (\ref{disspation}) as 
\begin{equation}
\begin{aligned}
E_{\Delta}(t) &= \sum_{m = 1}^M \sum_{j,k} \Delta x_j \Delta y_k \left[ \bar{c}_{m, j, k} \log \bar{c}_{m, j, k} \right] \\
&+ \dfrac{1}{2} \sum_{i, j, k, l} \Delta x_j \Delta x_i \Delta y_k \Delta y_l \mathcal{K}_{j - i, k - l} \rho_{i, l} \rho_{j, k} \\
&+ \dfrac{1}{2} \sum_{i, j, k, l} \Delta x_j \Delta x_i \Delta y_k \Delta y_l \mathcal{W}_{j - i, k - l} \theta_{i, l} \theta_{j, k}, 
\end{aligned}
\end{equation}

and

\begin{equation}
\begin{aligned}
D_{\Delta}(t) &= \sum_{m = 1}^{M} \sum_{j, k} \Delta x_j \Delta y_k \left[ \left( u_{m, j + \frac{1}{2}, k} \right)^2 + \left( v_{m, j, k + \frac{1}{2}} \right)^2 \right] \min \left\{ \bar{c}_{m, j, k}, \bar{c}_{m, j + 1, k},  \bar{c}_{m, j, k + 1} \right\}.
\end{aligned}
\end{equation}

\begin{thm}\label{thm4}(Free energy-dissipation estimate)
	Consider the two-dimensional semi-discrete finite volume scheme (\ref{sys20})-(\ref{sys40}) of the system (\ref{model1})-(\ref{model2}) with initial data $c_{m}^0(x) \geqslant 0, \, m = 1, \cdots, M$. Assume that there is no flux boundary conditions on $[-L_x, L_x] \times [-L_y, L_y]$ i.e. the discrete boundary conditions satisfy \\
	$F_{m, -M_x - \frac{1}{2}, \pm \left( M_y + \frac{1}{2} \right)} = F_{m, M_x + \frac{1}{2}, \pm \left(M_y + \frac{1}{2}\right)} = 0, \, m = 1, \cdots, M$, where $L_x$ and $L_y$ can be big enough. 
	Then we have
	\begin{equation}
	\dfrac{d}{d t} E_{\Delta}(t) \leqslant - D_{\Delta}(t) \leqslant 0, \quad \forall t \geqslant 0.
	\end{equation}
\end{thm}

	\section{Numerical Experiments}
\label{sec4}
In this section, we give several one- and two-dimensional numerical examples and verify the properties of the numerical schemes and explore the finite size effect numerically. In the following numerical examples, we add some additional external field into the original model (\ref{model1})-(\ref{model2}) to make sure the steady states are effectively localized, which is to say, the system we consider becomes for $m=1,\cdots,M$,
\begin{align}
\label{model3}
\partial_t c_m (x, t) &= \nabla \cdot \left[ c_m \nabla \left( 1 + \log c_m + z_m \mathcal{K}*\rho + \mathcal{W}*\theta + V_{\text{ext}} \right) \right], 
\\
c_m(\boldsymbol{x}, 0) &= c^0_{m}(\boldsymbol{x}), \label{model4}
\end{align}
where $V_{\text{ext}}$ is the added external potential. 

\subsection{Comparison of Kernel Functions}
At first, we plot the two- and three-dimensional kernel functions $\mathcal{K}(r)$ and $\mathcal{W}(r)$ in the form (\ref{kandw11}) and (\ref{kandw}) with the parameters $\eta = 1, a = \frac{1}{2}, k = 2$ and the grid size $\Delta x = 0.0097656$ in Figure \ref{plotkernel} respectively, which is, for the two-dimensional case
\begin{equation}
\begin{aligned}
&\mathcal{K}(r) = -{\frac{1}{2}}\log\left(r^2+\frac{1}{4}\right), \\
&\mathcal{W}(r) = \frac{1}{r^2+\frac{1}{4}},
\end{aligned}
\end{equation}
in Figure \ref{plotkernel}(a) and for the three-dimensional case
\begin{equation}
\begin{aligned}
&\mathcal{K}(r) = \frac{1}{\sqrt{r^2 + \frac{1}{4}}}, \\
&\mathcal{W}(r) = \frac{1}{r^2 + \frac{1}{4}},
\end{aligned}
\end{equation}
in Figure \ref{plotkernel}(b). 
\begin{figure}[htp] 
	\hspace*{0.05\textwidth}
	\subfigure[]{ 
		\centering
		\includegraphics[width=0.45\textwidth]{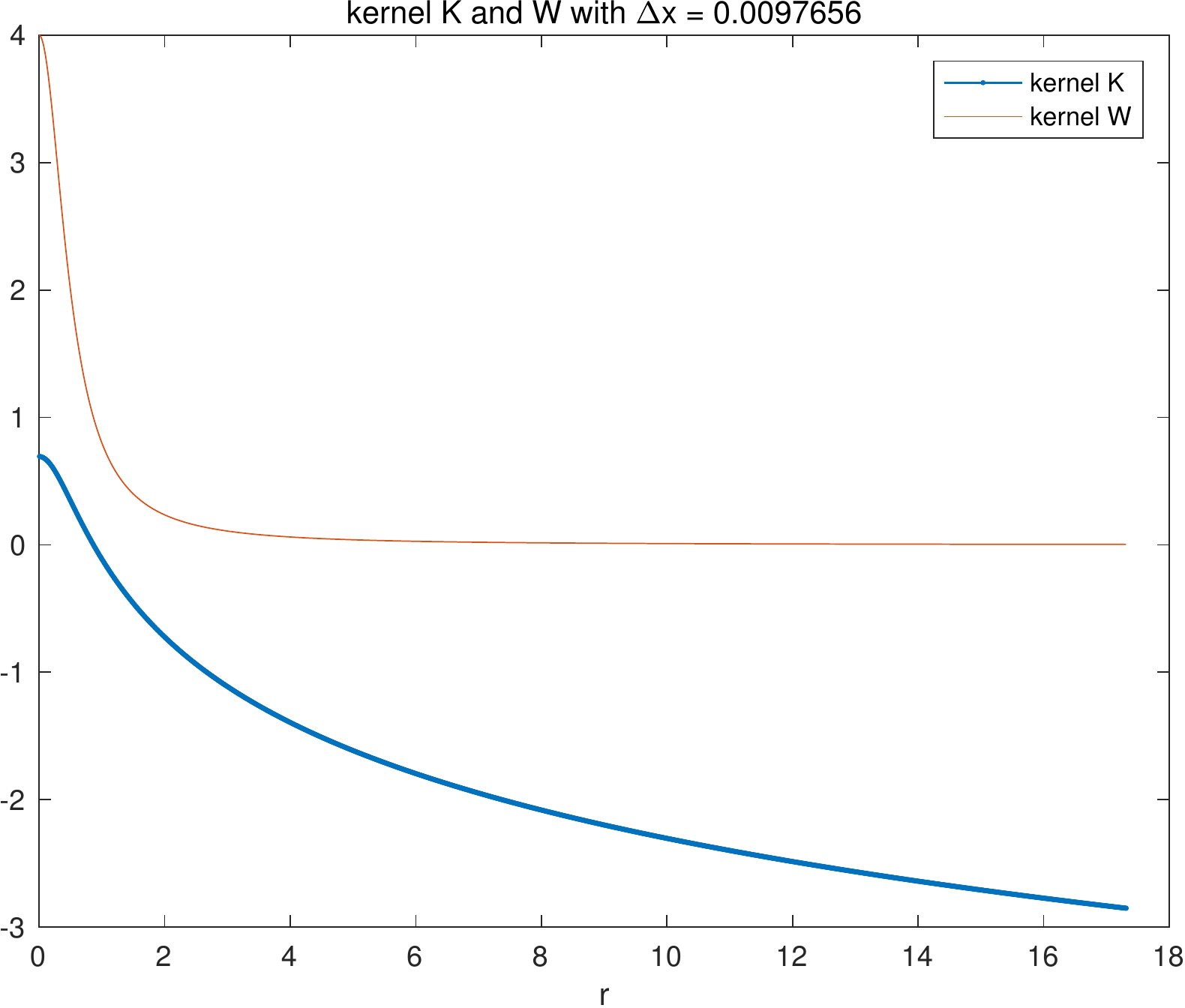}} 
	\subfigure[]{ 
		\centering
		\includegraphics[width=0.45\textwidth]{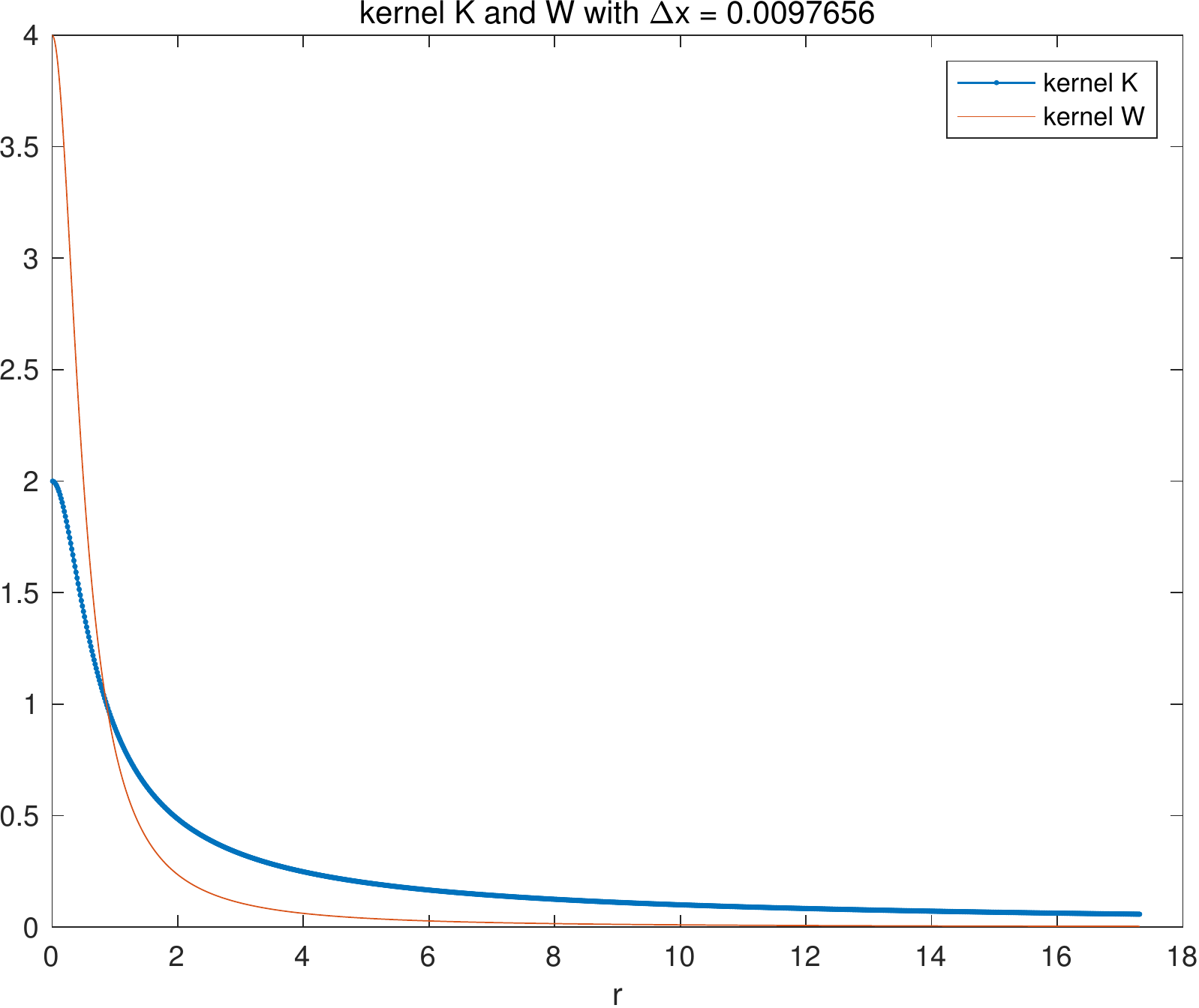} 
	}
	\caption{Comparison of Kernel Functions: (a): The two-dimensional kernel functions $\mathcal{K}(r)$ and $\mathcal{W}(r)$ with the mesh size $\Delta x$ being 0.0097656. (b): The three-dimensional kernel functions $\mathcal{K}(r)$ and $\mathcal{W}(r)$ with the mesh size $\Delta x$ being 0.0097656.}
	\label{plotkernel}
\end{figure}

\subsection{Convergence Test}
Consider the equations (\ref{model3}) for complex ionic fluids in one-dimension with the kernel $\mathcal{W}(x) = \frac{1}{x^2 + \epsilon^2}, \epsilon = \frac{1}{2}$, $\mathcal{K}(x) = \exp(-|x|)$, $V_{\text{ext}}(x) = \frac{1}{2} x^2$. Note that, the electrostatic kernel $\mathcal{K}(x)$ in one-dimension is not physically relevant, and thus this numerical example is a toy model, which only serves the purpose of the  convergence test. The initial conditions (\ref{model4}) with which the equations (\ref{model3}) equipped are given by
\begin{equation*}
\left\{
\begin{array}{lll}
c_1(x, 0) = \dfrac{1}{\sqrt{2 \pi}} \exp \left(-\dfrac{(x - 2)^2}{2} \right) &\text{with} &z_1 = 1, \\
c_2(x, 0) = \dfrac{1}{\sqrt{2 \pi}} \exp \left(-\dfrac{(x + 2)^2}{2} \right) &\text{with} &z_2 = -1.
\end{array}
\right.
\end{equation*}
Here, take the computation domain as $[-2L, 2L], \ L = 10$, then the results of the convergence of error in $l^{\infty}, \ l^1$ and $l^2$ norms at time $t = 1$ is shown in Figure \ref{im1} where we take the uniform mesh size $\Delta x$ be $1.2500,    0.6250,    0.3125,    0.15625,    0.078125,    0.0390625,
0.01953125(N$ be $2^5, 2^6, 2^7, 2^8, 2^9, 2^{10}, 2^{11})$, $\Delta t$ is determined by (\ref{cfl}), here $\Delta t = {\Delta x}/{(2 ~U_{\text{max}})}$. And we omit it in other one-dimensional examples. Meanwhile, we define the errors of numerical solutions
\begin{equation}
\|\boldsymbol{e}\|_{l^{\infty}} := \max_{m, j}|c_{m,j} - c^{\text{ref}}_{m,j}|, ~~ \|\boldsymbol{e}\|_{l^p} := \left( \sum_{m = 1}^{M} \Delta x \sum_{j} |c_{m,j} - c^{\text{ref}}_{m,j}|^p \right)^{\frac{1}{p}}, p = 1, 2.
\end{equation}
Here $c^{\text{ref}}_{m,j}$ is the reference solution of the $m$-th species on mesh with mesh size $\Delta x = 0.0048828125\ (N = 2^{13})$. The first order convergence in space can be observed in Figure \ref{im1}.

\begin{figure}[htp] 
	\centering 
	\includegraphics[width=.6\textwidth]{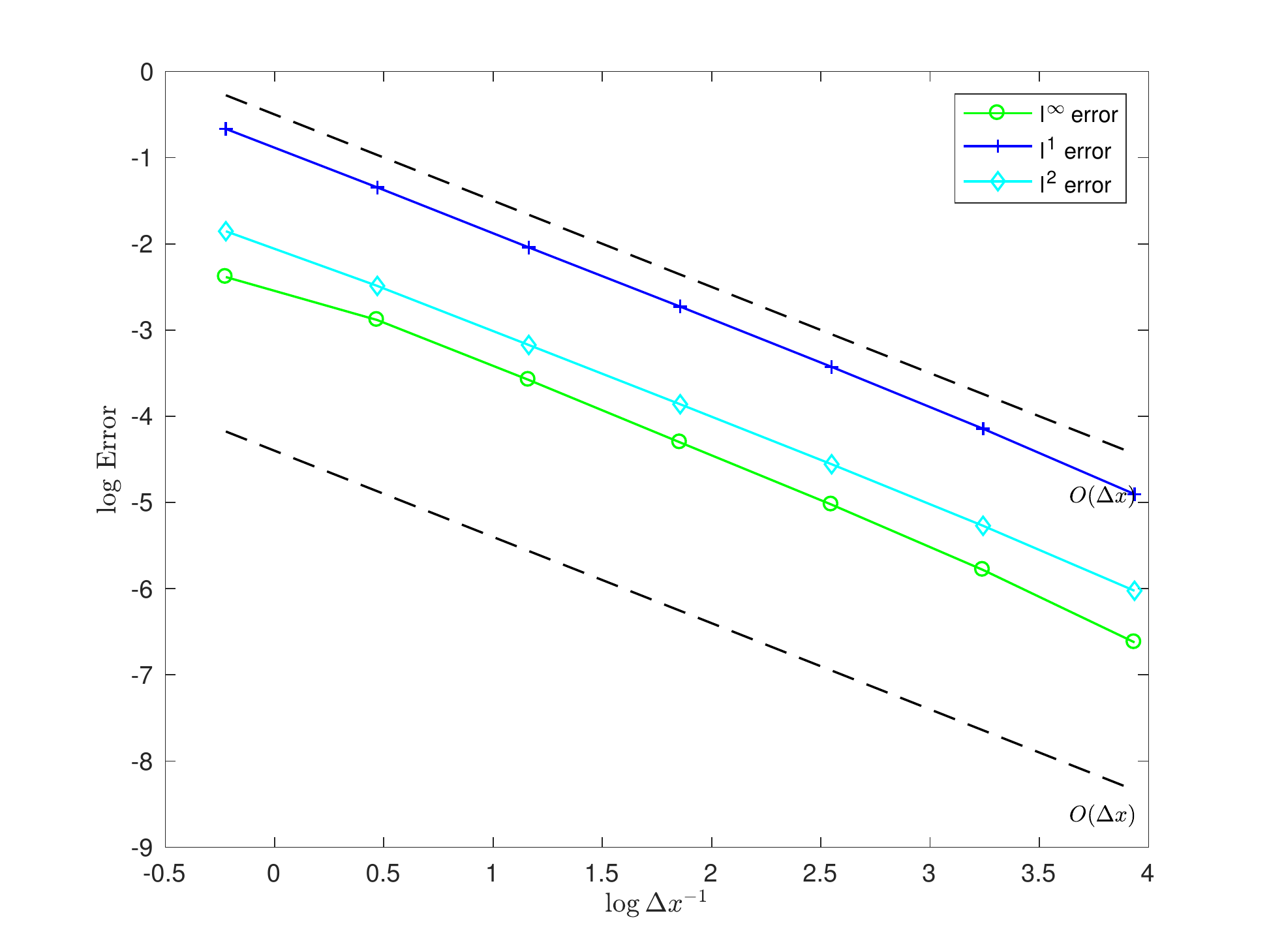} 
	\caption{Convergence Test: The loglog plot of errors with the mesh size $\Delta x$ being $1.2500,    0.6250,    0.3125,    0.15625,    0.078125,    0.0390625,
		0.01953125$ at time $t = 1$} 
	\label{im1} 
\end{figure}

\subsection{Multiple Species in One-dimension}
\subsubsection{Steady State} In this part, we study the steady state of a one-dimensional example. Consider the equations (\ref{model3}) in one-dimension with the kernel $\mathcal{W}(x) = \frac{\eta}{x^2 + \epsilon^2}, \epsilon = \frac{1}{10}$, $\mathcal{K}(x) = \exp(-|x|)$, $V_{\text{ext}}(x) = \frac{1}{2} x^2$ and the initial conditions (\ref{model4}) are given by the following form
\begin{equation}
\label{iniex1}
\left\{
\begin{array}{lll}
c_1(x, 0) = \dfrac{1}{2 \sqrt{\pi}} \exp \left(-\dfrac{(x - 2)^2}{2} \right) &\text{with} &z_1 = 1, \\
c_2(x, 0) = \dfrac{1}{\sqrt{2 \pi}} \exp \left(-\dfrac{(x + 2)^2}{2} \right) &\text{with} &z_2 = -1.
\end{array}
\right.
\end{equation}
In this part, we take the parameter $\eta = 1$, the computation domain as $[-2L, 2L], \ L = 10$ and the uniform mesh size $\Delta x = 0.01953125\ (N = 2^{11})$. 
The results with which we are concerned are on the domain $[-L, L]$. Then Figure \ref{3c} shows the transport of the ionic species: the concentrations of the positive ions and the negative ions move towards each other due to the electrostatic attraction with time $t$ and the concentrations converge to the equilibrium.
\begin{figure}[htp] 
	\centering
	\subfigure{
		\begin{minipage}[t]{0.33\linewidth}
			\centering
			\includegraphics[width=1.0\linewidth]{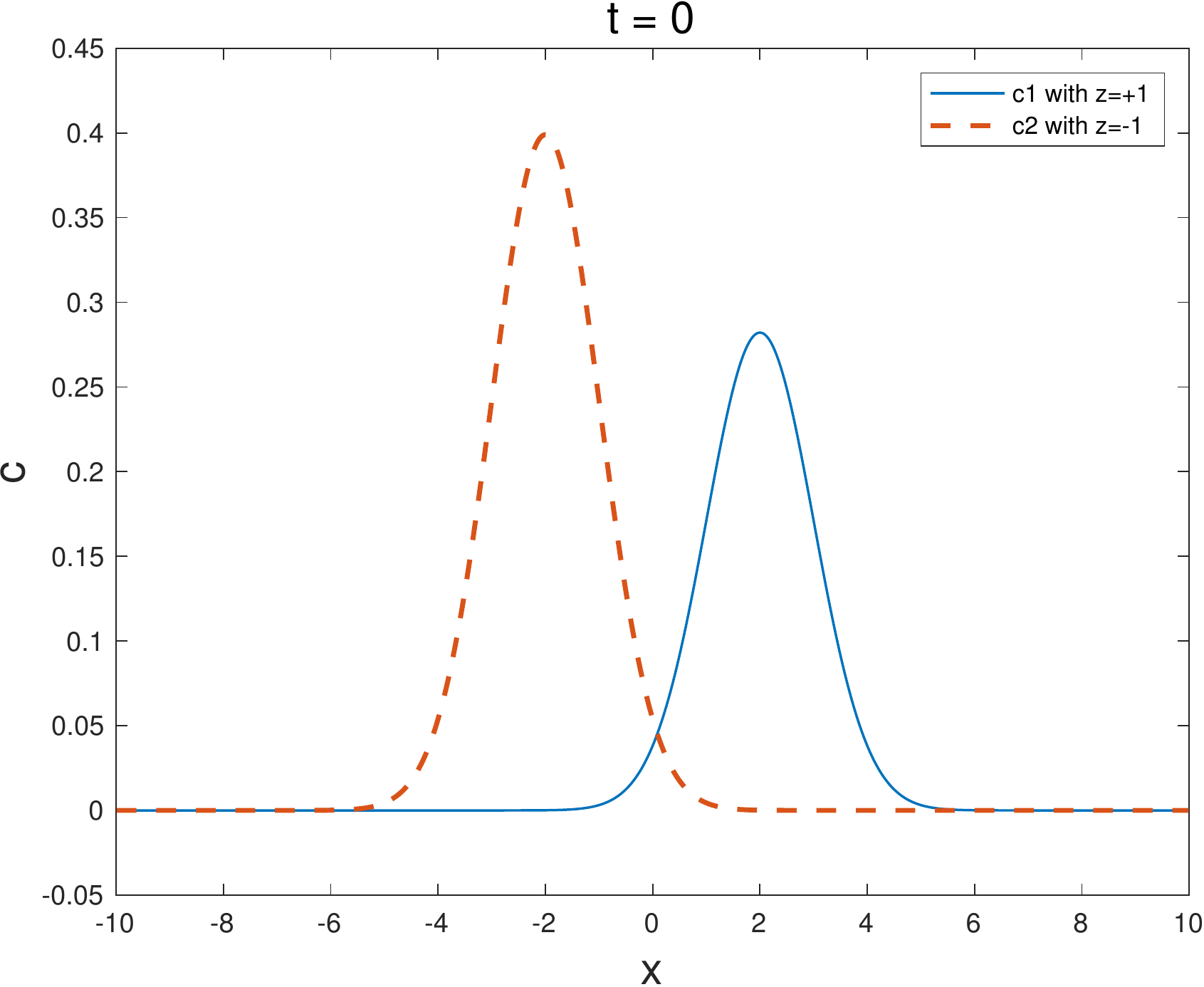}
		\end{minipage}
		\begin{minipage}[t]{0.33\linewidth}
			\centering
			\includegraphics[width=1.0\linewidth]{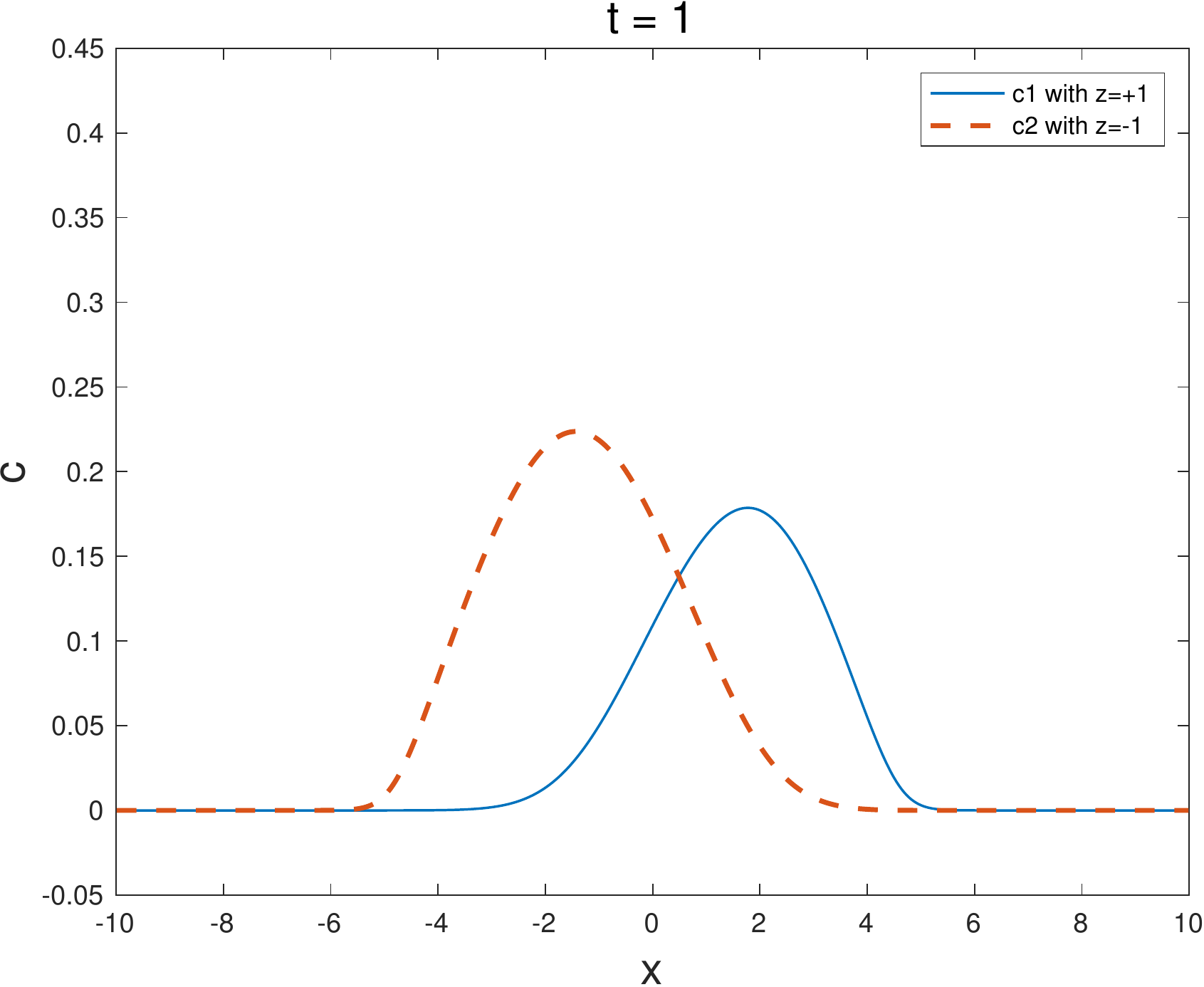}
		\end{minipage}
		\begin{minipage}[t]{0.33\linewidth}
			\centering
			\includegraphics[width=1.0\linewidth]{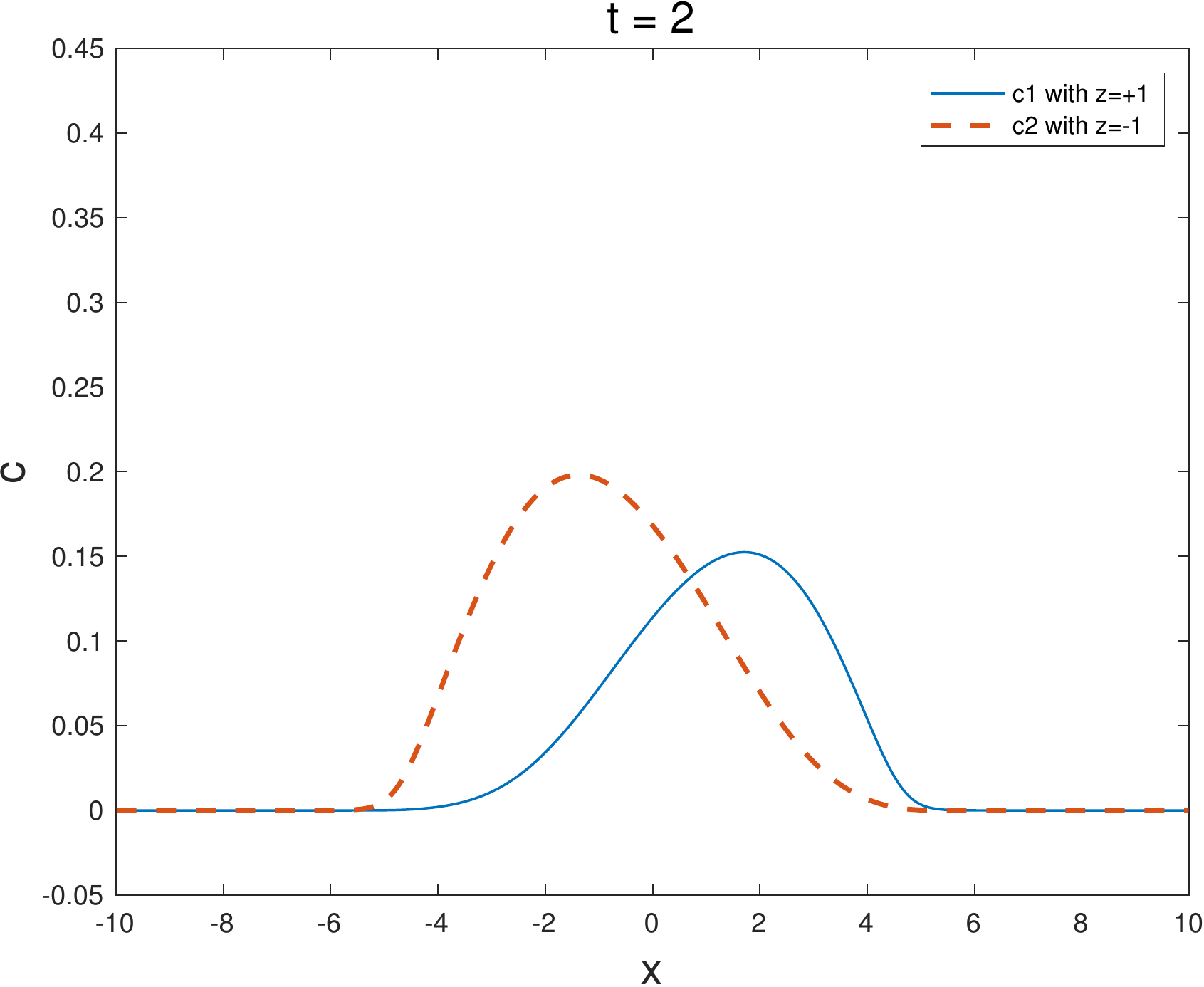}
		\end{minipage}
	}
	\subfigure{
		\begin{minipage}[t]{0.33\linewidth}
			\centering
			\includegraphics[width=1.0\linewidth]{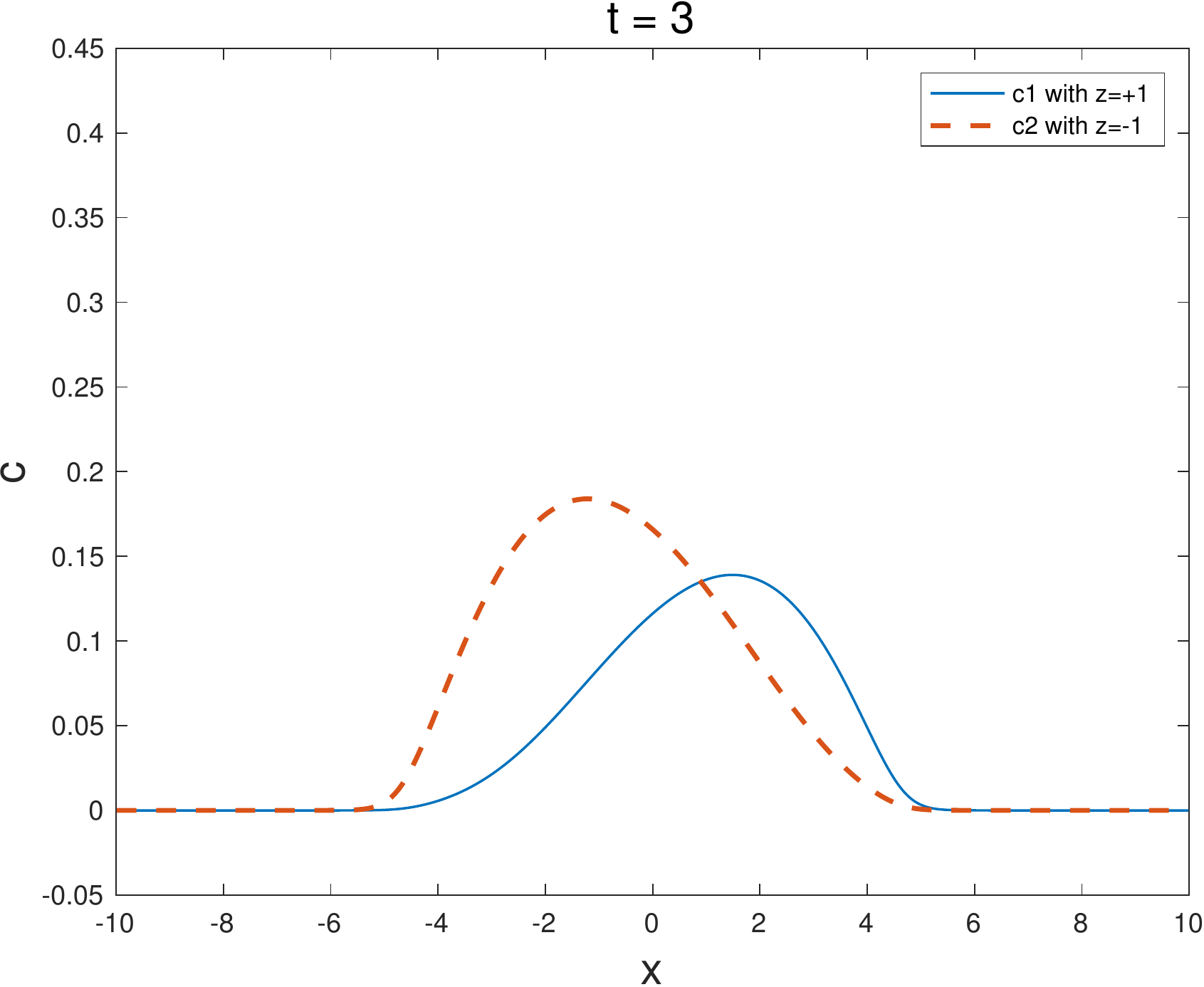}
		\end{minipage}
		\begin{minipage}[t]{0.33\linewidth}
			\centering
			\includegraphics[width=1.0\linewidth]{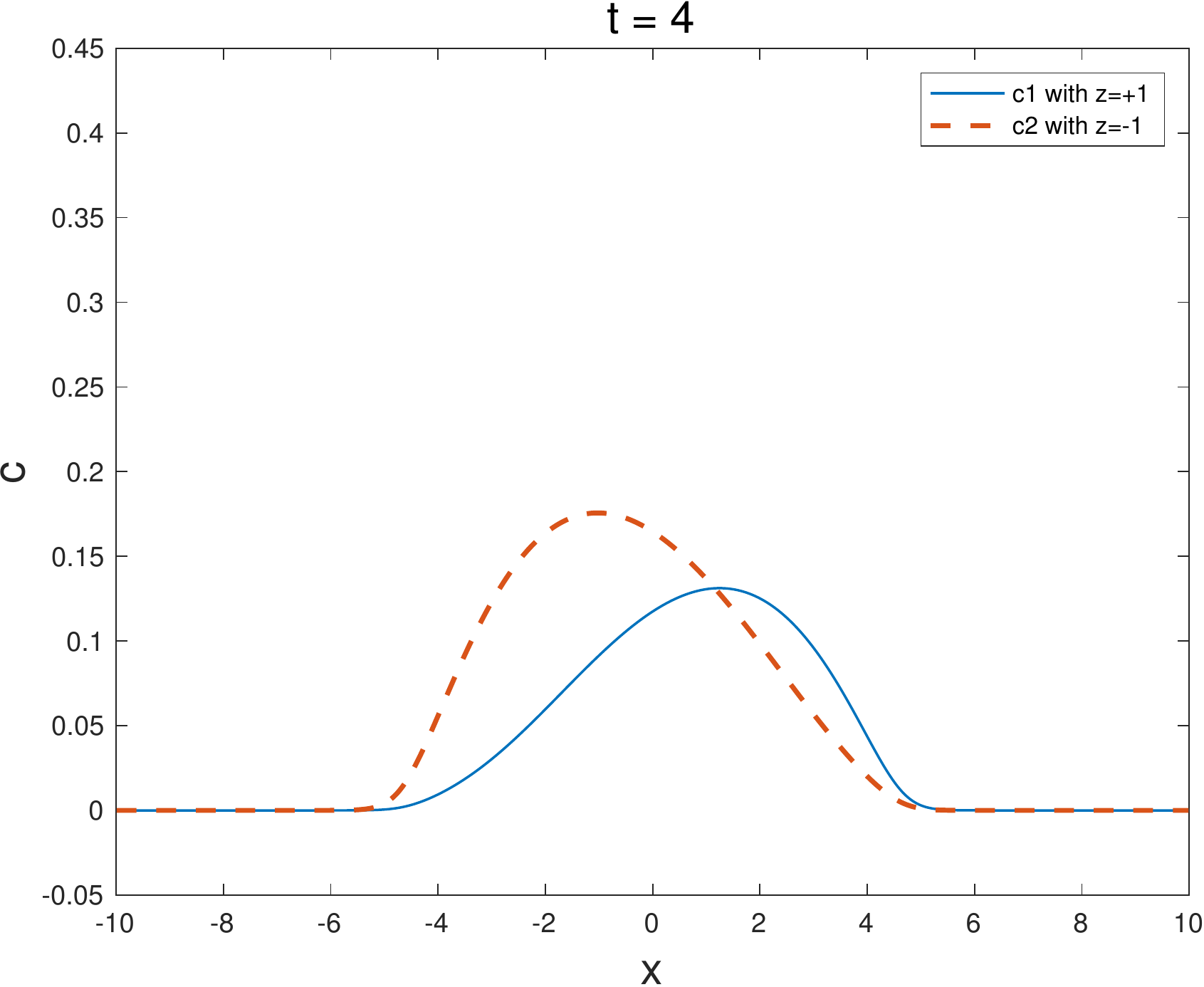}
		\end{minipage}
		\begin{minipage}[t]{0.33\linewidth}
			\centering
			\includegraphics[width=1.0\linewidth]{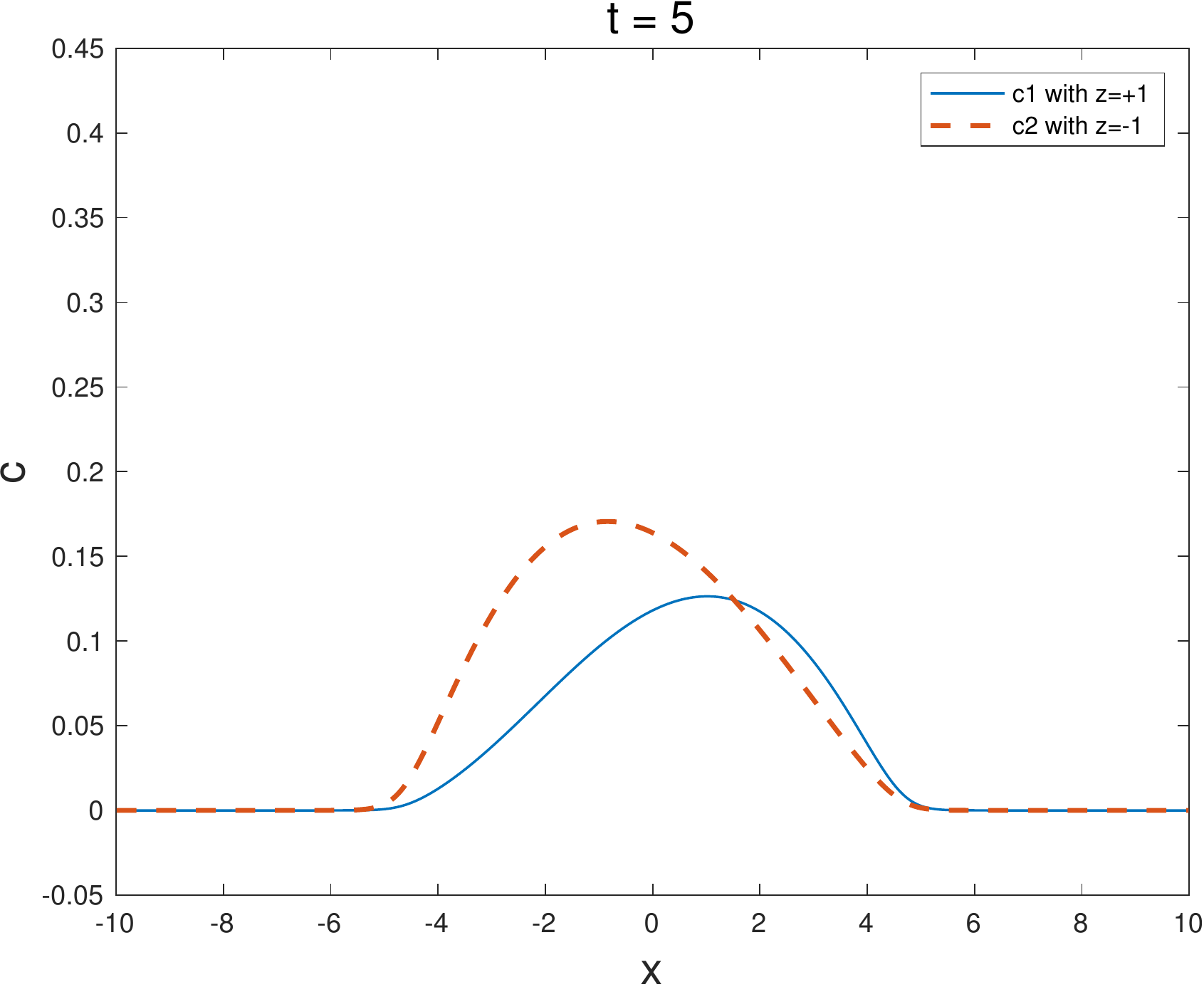}
		\end{minipage}
	}
	\subfigure{
		\begin{minipage}[t]{0.33\linewidth}
			\centering
			\includegraphics[width=1.0\linewidth]{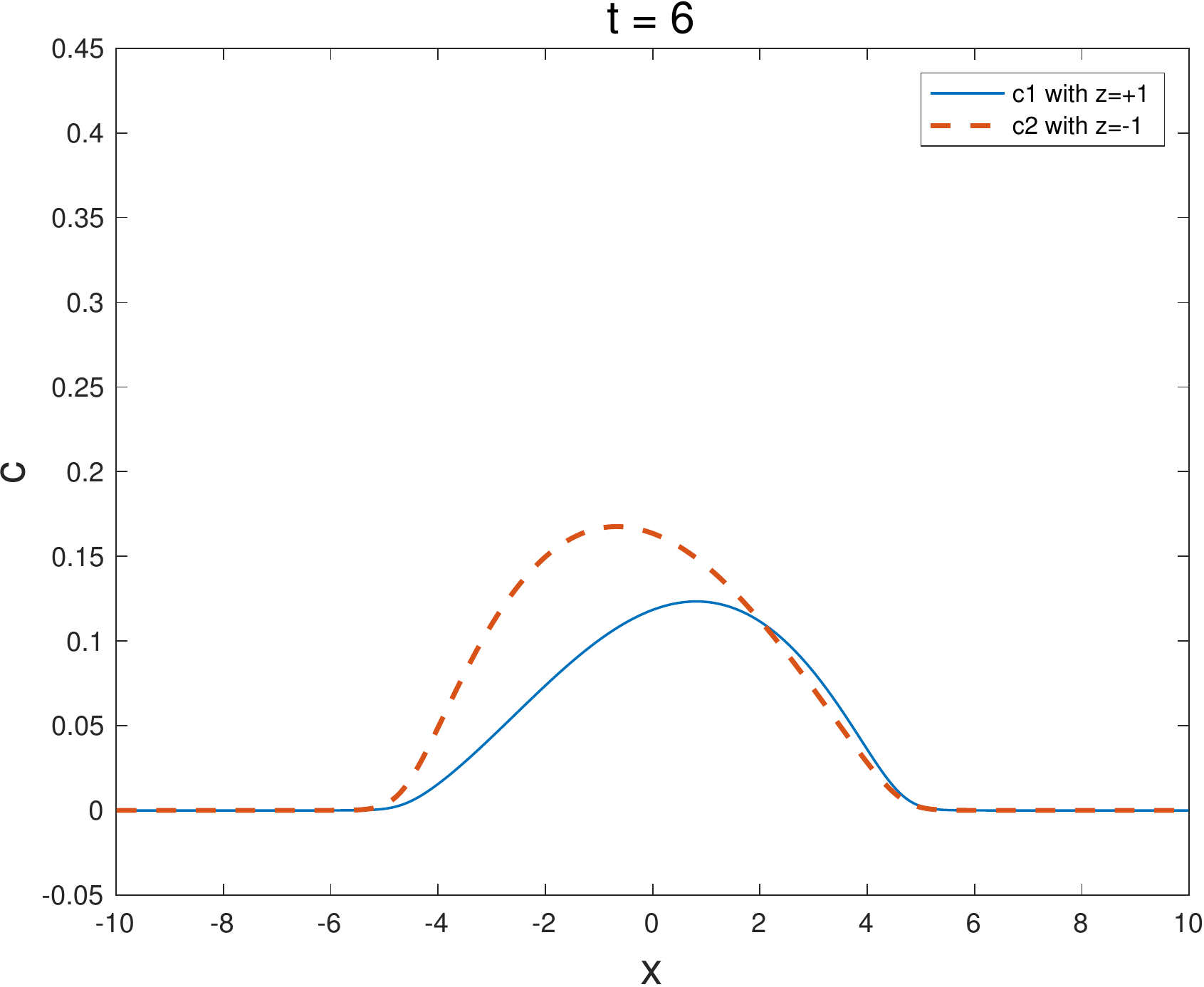}
		\end{minipage}
		\begin{minipage}[t]{0.33\linewidth}
			\centering
			\includegraphics[width=1.0\linewidth]{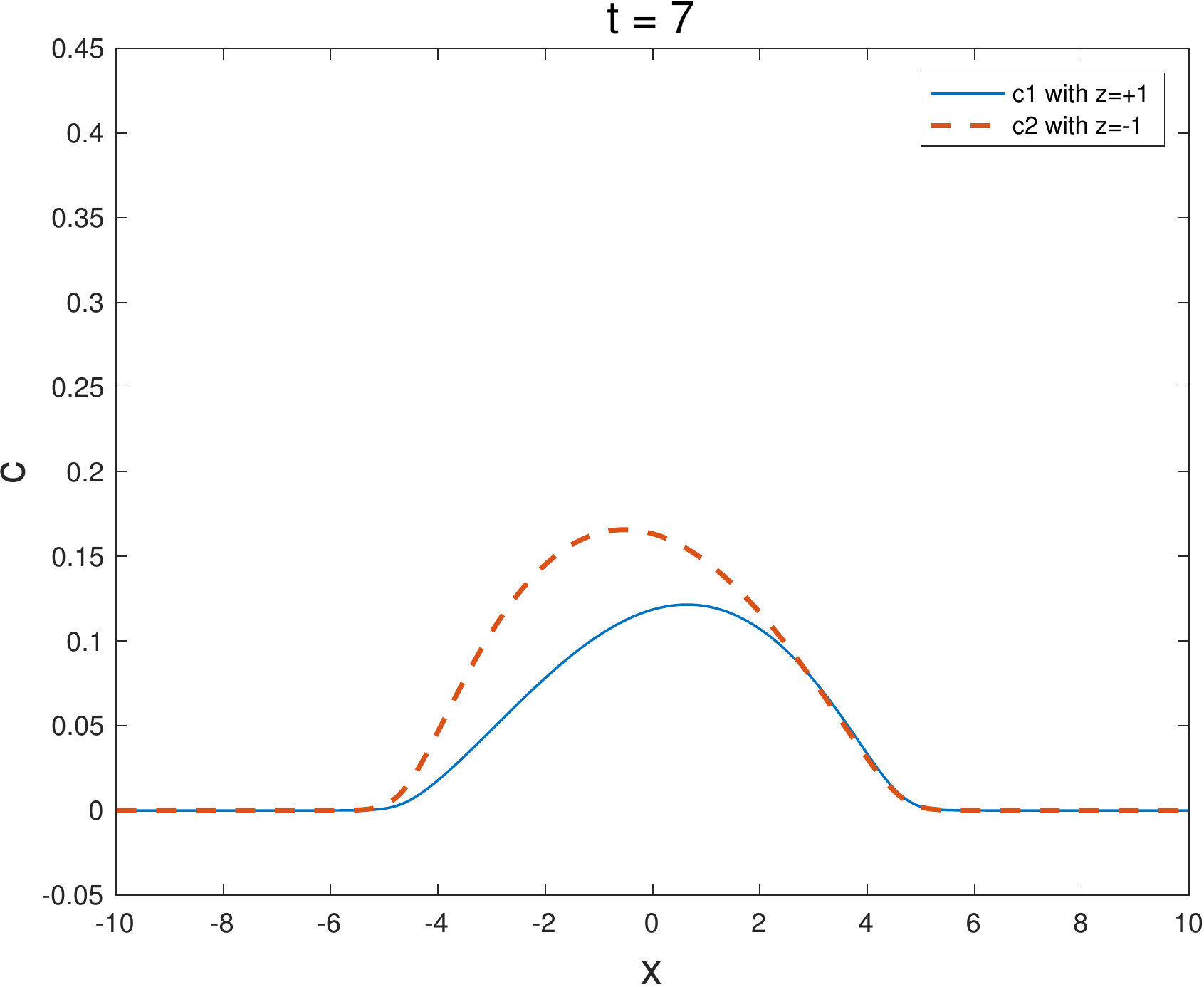}
		\end{minipage}
		\begin{minipage}[t]{0.33\linewidth}
			\centering
			\includegraphics[width=1.0\linewidth]{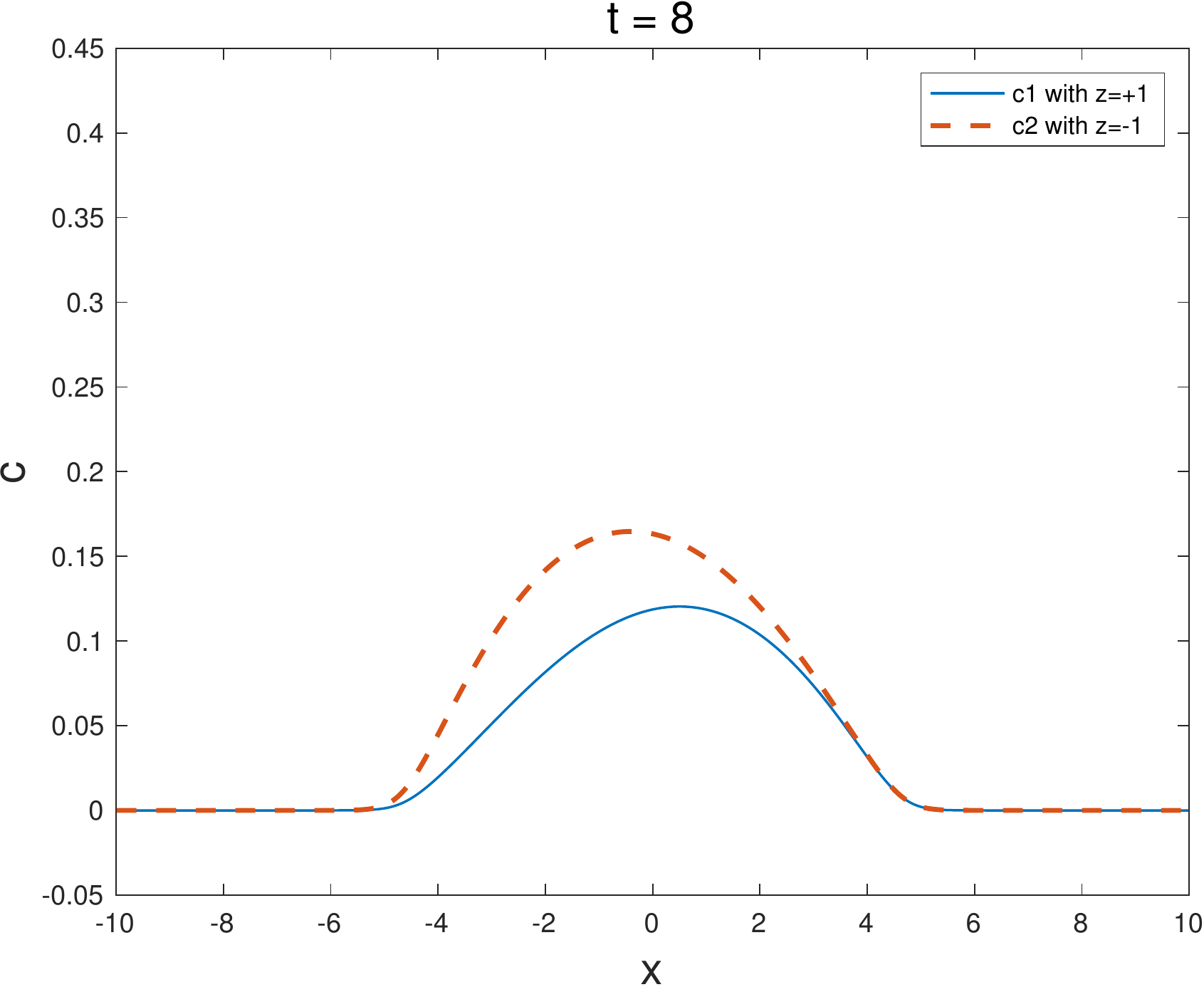}
		\end{minipage}
	}
	\subfigure{
		\begin{minipage}[t]{0.33\linewidth}
			\centering
			\includegraphics[width=1.0\linewidth]{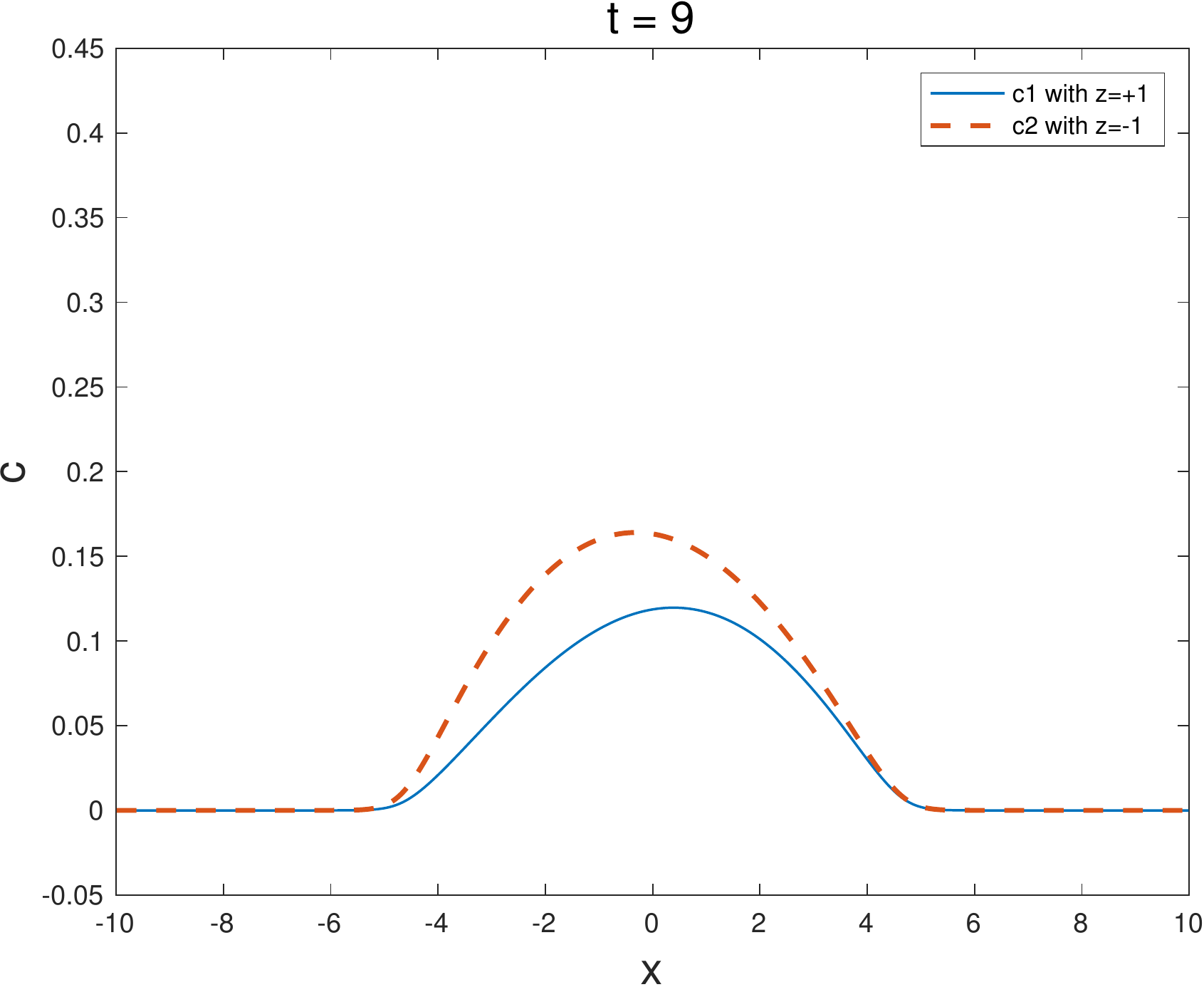}
		\end{minipage}
		\begin{minipage}[t]{0.33\linewidth}
			\centering
			\includegraphics[width=1.0\linewidth]{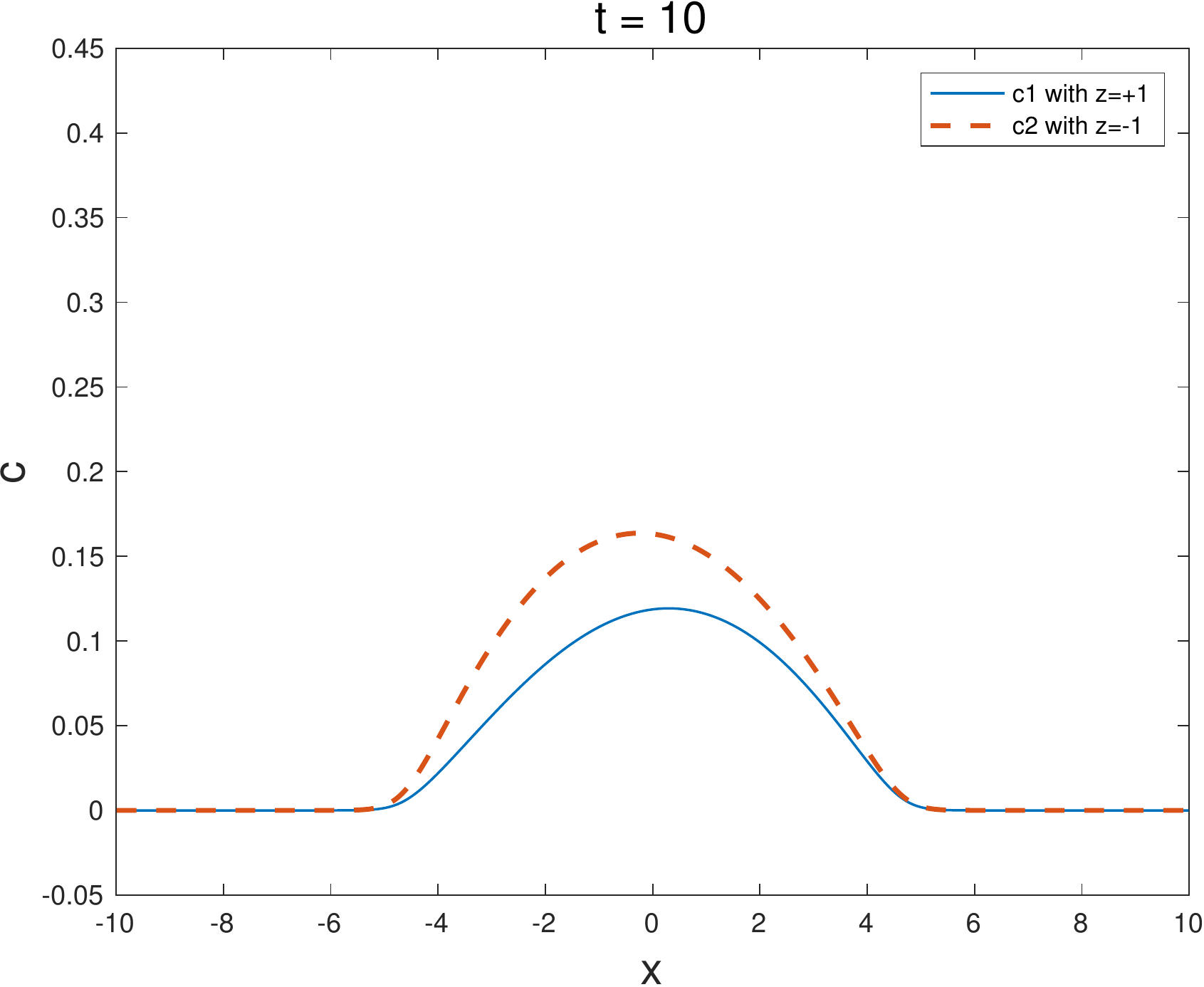}
		\end{minipage}
		\begin{minipage}[t]{0.33\linewidth}
			\centering
			\includegraphics[width=1.0\linewidth]{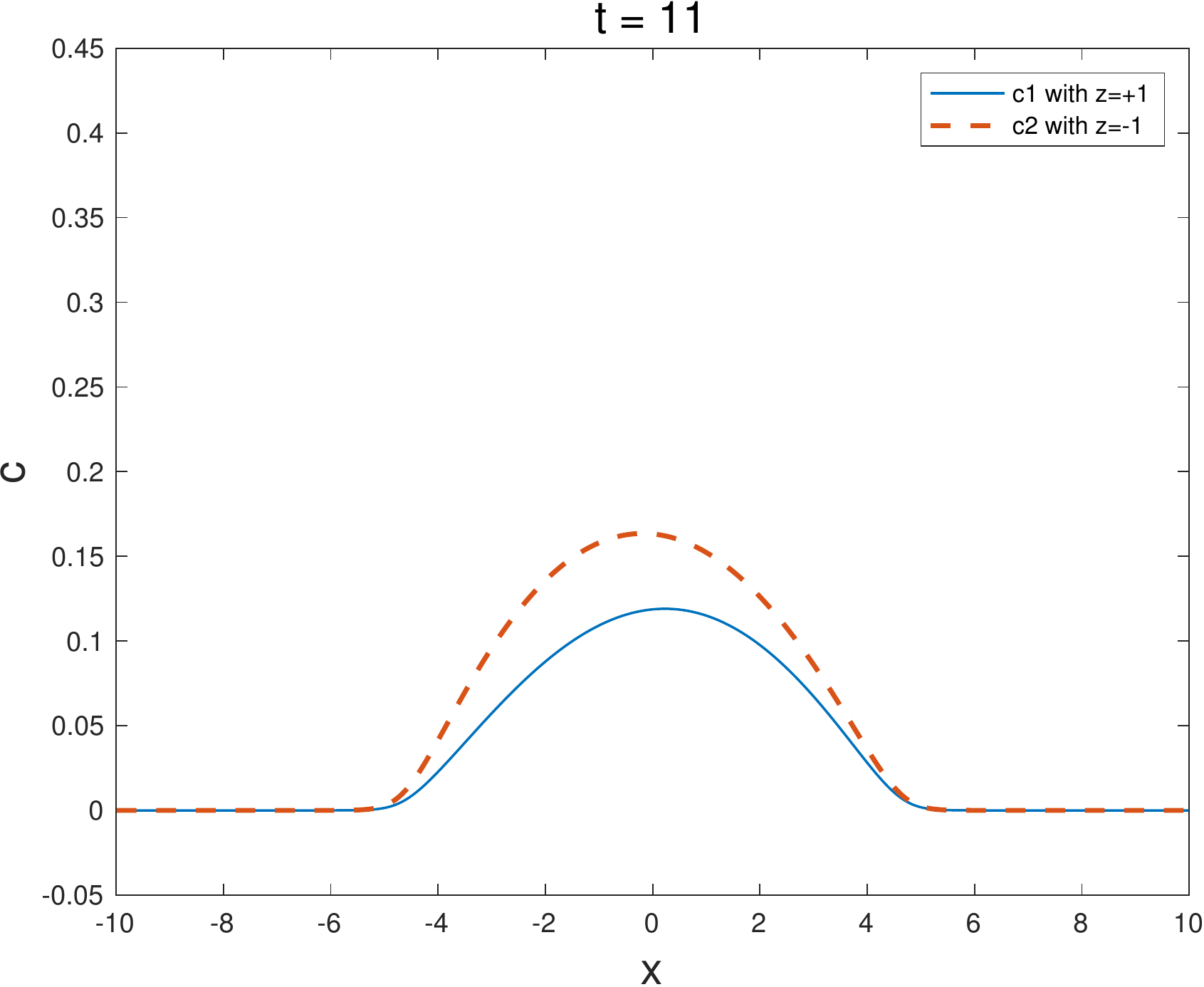}
		\end{minipage}
	}
	\subfigure{
		\begin{minipage}[t]{0.33\linewidth}
			\centering
			\includegraphics[width=1.0\linewidth]{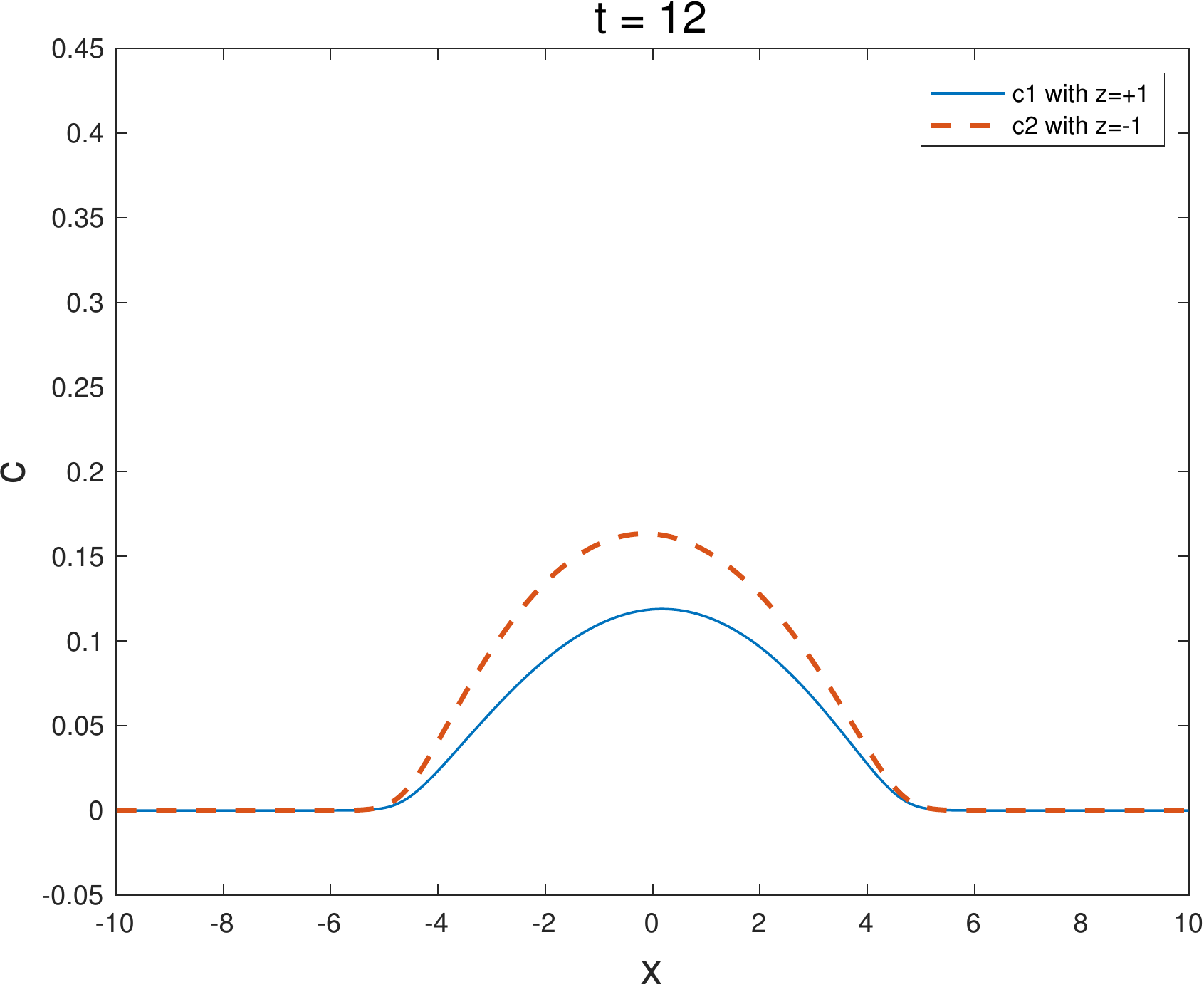}
		\end{minipage}
		\begin{minipage}[t]{0.33\linewidth}
			\centering
			\includegraphics[width=1.0\linewidth]{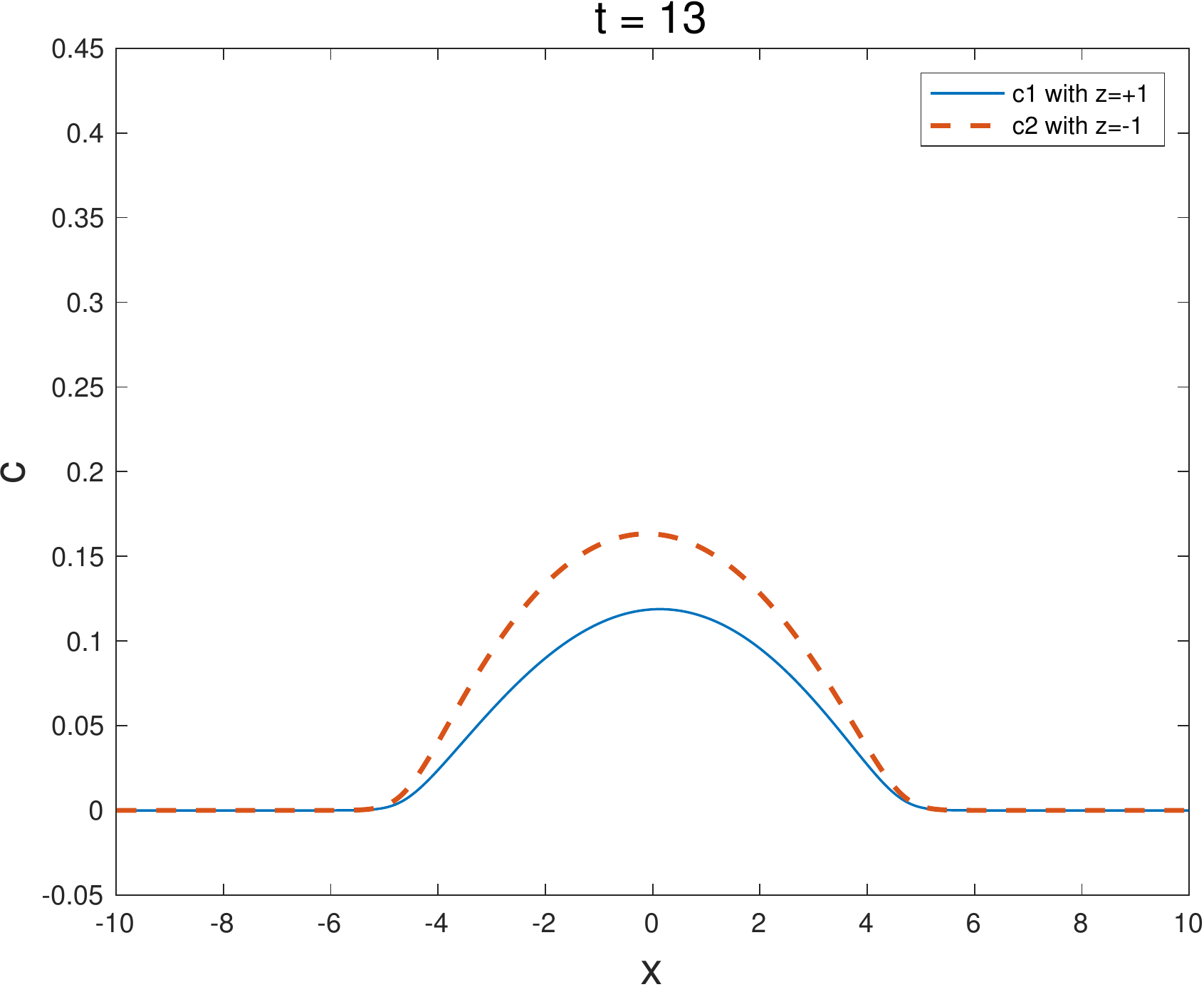}
		\end{minipage}
		\begin{minipage}[t]{0.33\linewidth}
			\centering
			\includegraphics[width=1.0\linewidth]{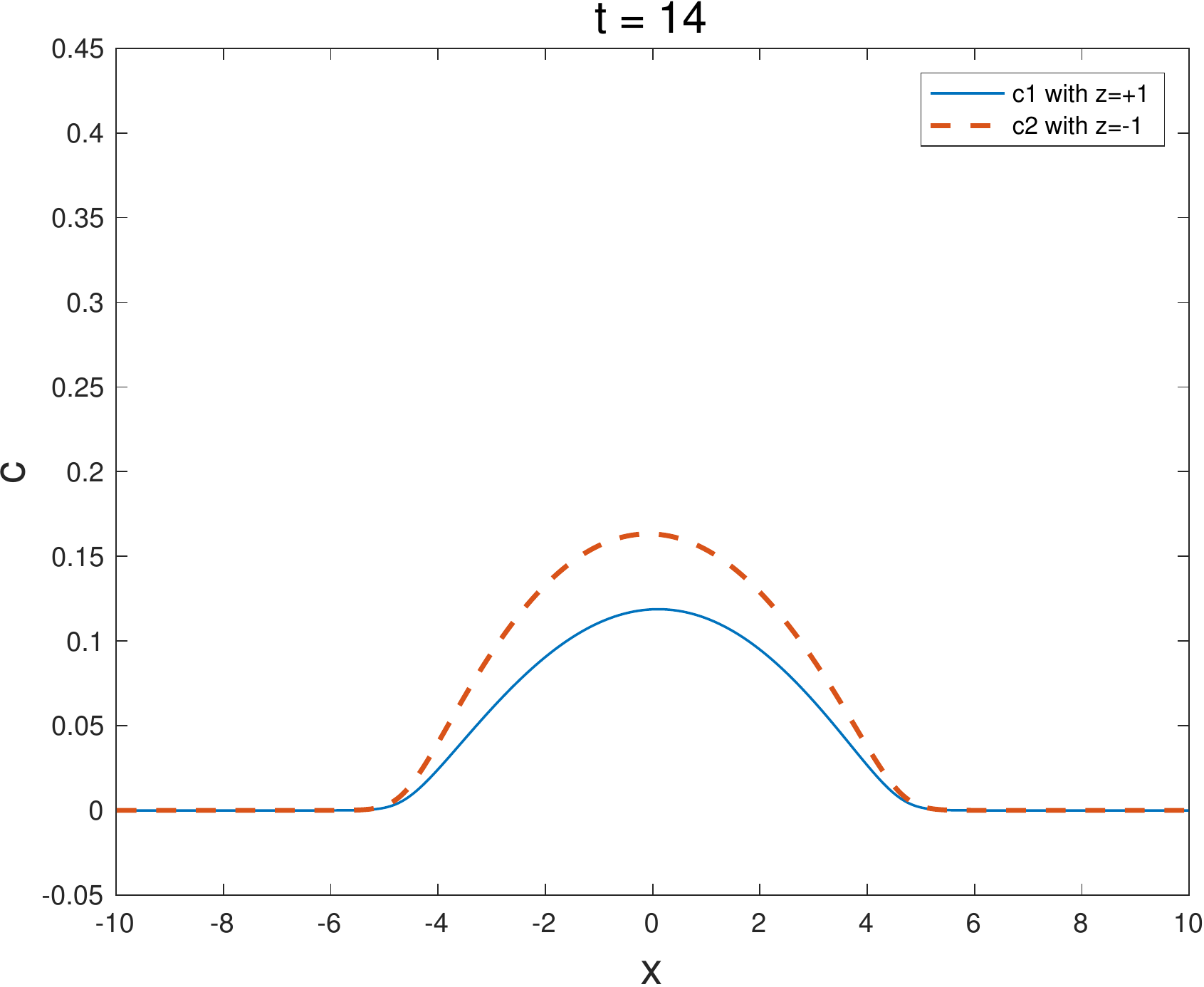}
		\end{minipage}
	}
	\caption{Multiple Species in One-dimension: The space-concentration curves with the mesh size $\Delta x$ being 0.01953125 and the time $t$ changing from 0 to 14}
	\label{3c}
\end{figure}
According to the fact that every individual part of the free energy $\mathcal{F}$ has its own physical effect, we can define the internal energy $\mathcal{F}_1$, the field energy $\mathcal{F}_2$, interaction energy $\mathcal{F}_3$, external field energy $\mathcal{F}_4$ of the model (\ref{model3})-(\ref{model4}) respectively as
\begin{equation}
\begin{aligned}
\mathcal{F}_1(t) &= \int_{\mathbb{R}^d} \sum_{m = 1}^{M} c_m \log c_m \,d x, \\
\mathcal{F}_2(t) &= \dfrac{1}{2} \int_{\mathbb{R}^d} \int_{\mathbb{R}^d} \mathcal{K}(x - y) \rho(x) \rho(y) \,d x \,d y, \\
\mathcal{F}_3(t) &= \dfrac{1}{2} \int_{\mathbb{R}^d} \int_{\mathbb{R}^d} \mathcal{W}(x - y) \theta (x) \theta(y) \,\mathrm{d} x \,d y, \\
\mathcal{F}_4(t) &= \sum_{m = 1}^{M} \int_{\mathbb{R}^d} V_{\text{ext}}(x) c_m \,\mathrm{d} x,
\end{aligned}
\end{equation}
then the total free energy $\mathcal{F}$ defined by (\ref{free2}) equals the sum of the energy $\mathcal{F}_1, \mathcal{F}_2, \mathcal{F}_3$ and $\mathcal{F}_4$,
\begin{equation}
\begin{aligned}
\mathcal{F}(t) = \mathcal{F}_1(t) + \mathcal{F}_2(t) + \mathcal{F}_3(t) + \mathcal{F}_4(t). 
\end{aligned}
\end{equation}

Figure \ref{im2}(a) shows how the discrete forms of the energy $\mathcal{F}_1, \mathcal{F}_2, \mathcal{F}_3, \mathcal{F}_4, \mathcal{F}$ change with time $t$ and 
Figure \ref{im2}(b) shows the chemical potential $\psi_m, \ m = 1, 2,$ at time $t = 23$. It's observed that $\psi_m$ goes to a constant while the model goes to the equilibrium for all $m$ and the discrete form of $\mathcal{F}$ decays with time $t$. The results are consistent with our conclusions in this paper.
\begin{figure}[htp] 
	\hspace*{0.05\linewidth}
	\subfigure[]{ 
		\includegraphics[width=0.45\linewidth]{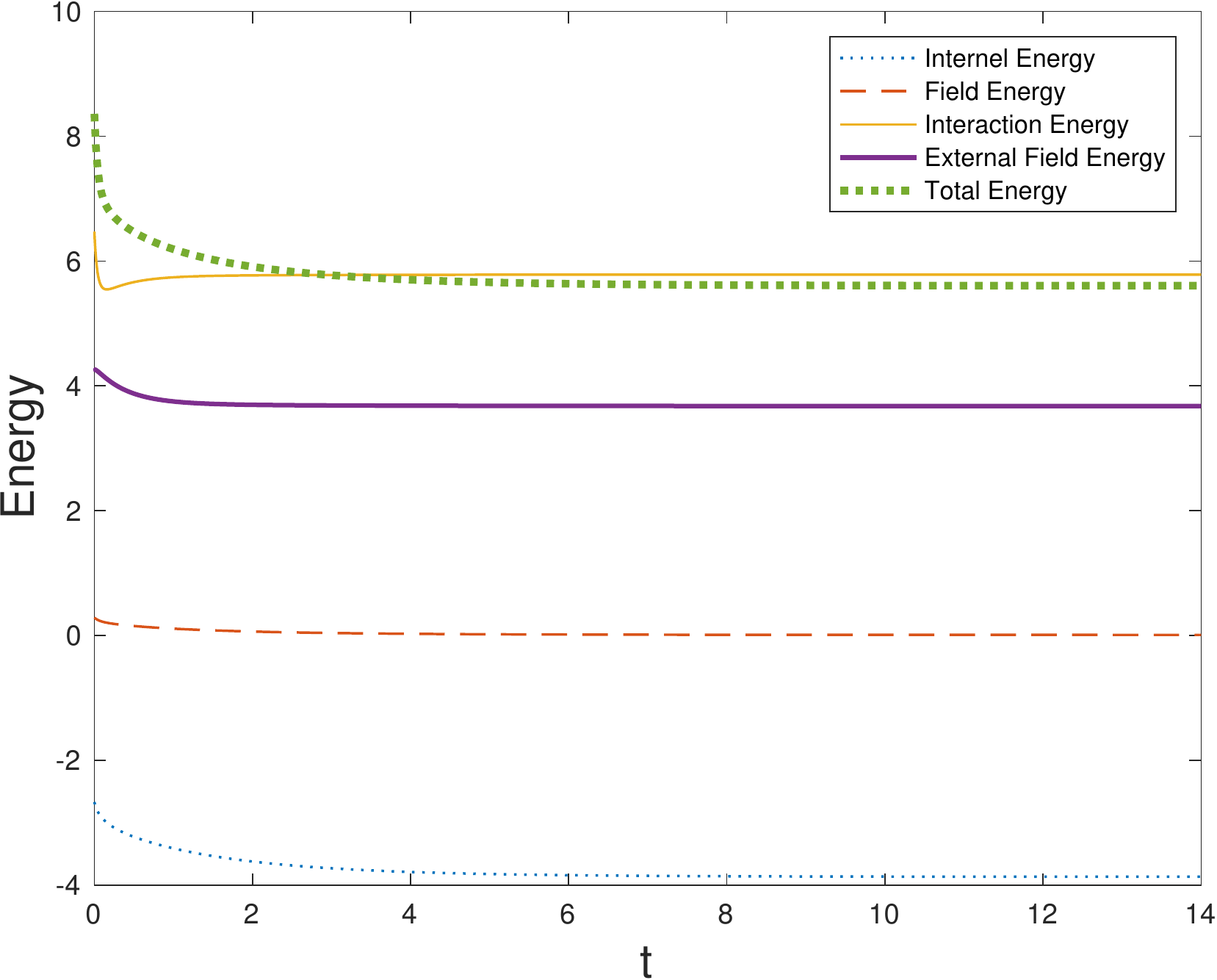}} 
	\subfigure[]{ 
		\includegraphics[width=0.45\linewidth]{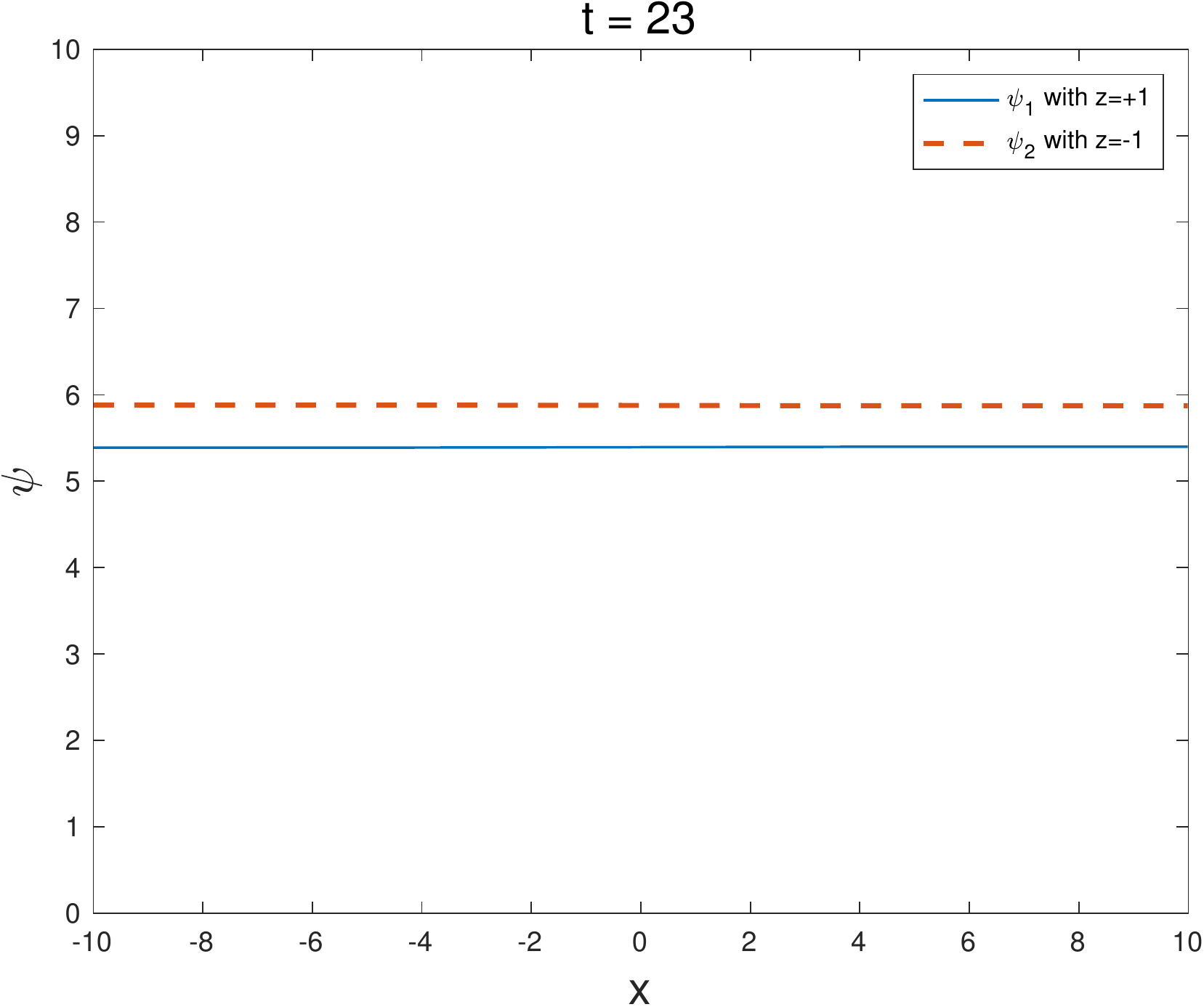} 
	}
	\caption{Multiple Species in One-dimension: (a): The time-energy plot of the model (\ref{model3}) equipped with the initial conditions (\ref{iniex1}) with the mesh size $\Delta x$ being 0.01953125. (b): Discrete chemical potential at time $t = 23$ with the mesh size $\Delta x$ being 0.01953125.}
	\label{im2}
\end{figure}

\subsubsection{Finite Size Effect} 
As we mentioned before, the kernel $\mathcal{W}(x)$ represents the steric repulsion arising from the finite size, the strength of which 
is indicated by the parameter $\eta$. The larger $\eta$ is, the stronger the nonlocal steric repulsion effect is, and thus the less peaked the concentrations of the steady state are. And $\eta = 0$ means steric repulsion vanishes. Here we aim to explore this phenomenon by different values of the parameter $\eta$.
Let $\eta = \dfrac{1}{2}, \dfrac{1}{4}, \cdots, \dfrac{1}{256}, 0$ and the mesh size $\Delta x = 0.0390625$, Figure \ref{eta} shows different steady state solutions with different values of $\eta$, here the density solutions $c_m, \ m = 1, 2,$ of time $t = 14$ approximate steady state solution, where we can find that the finite size effect makes the concentrations $c_m, \ m = 1, 2,$ not overly peaked.

\begin{figure}[htp] 
	\subfigure{
		\begin{minipage}[t]{0.33\linewidth}
			\centering
			\includegraphics[width=1.0\linewidth]{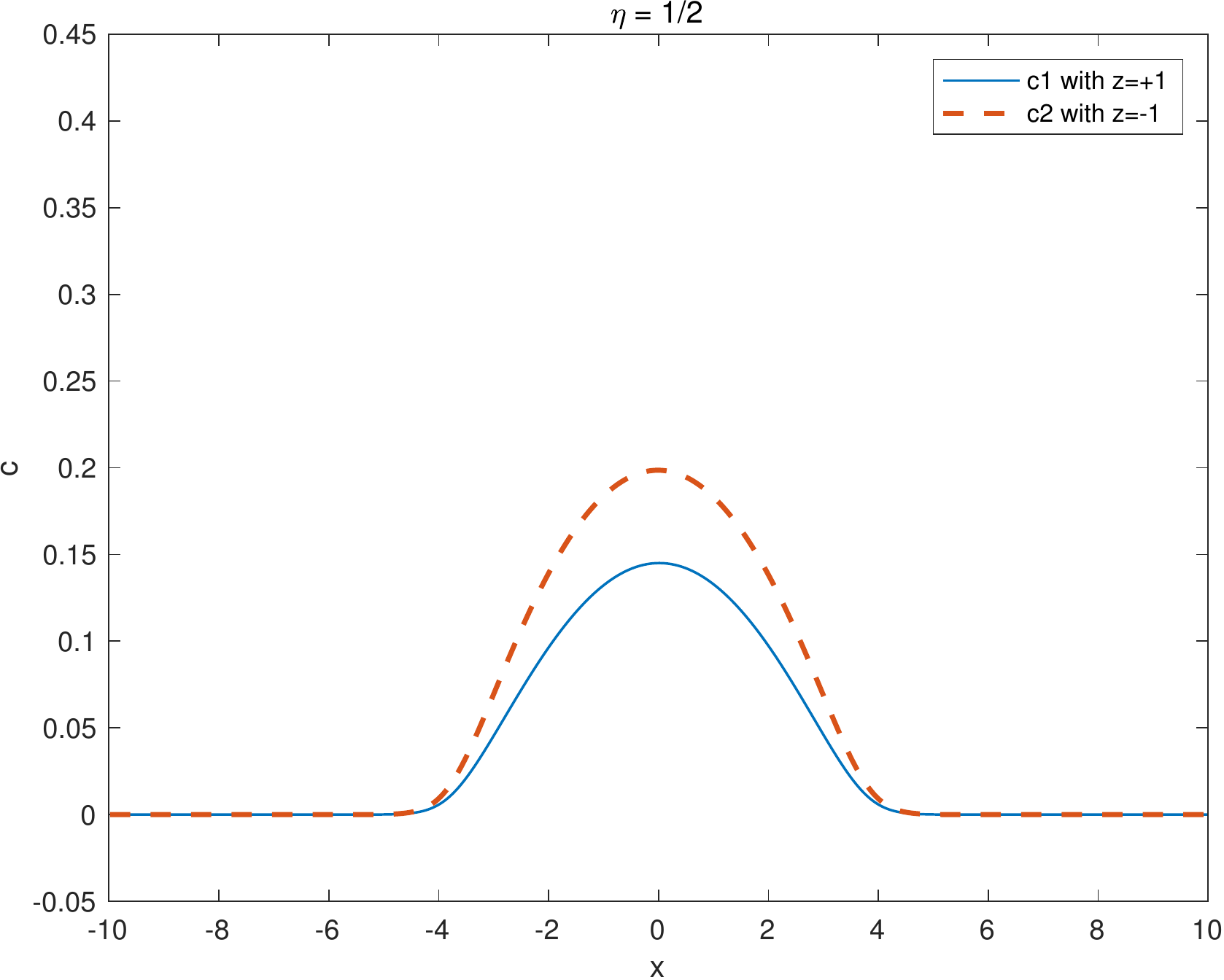}
		\end{minipage}
		\begin{minipage}[t]{0.33\linewidth}
			\centering
			\includegraphics[width=1.0\linewidth]{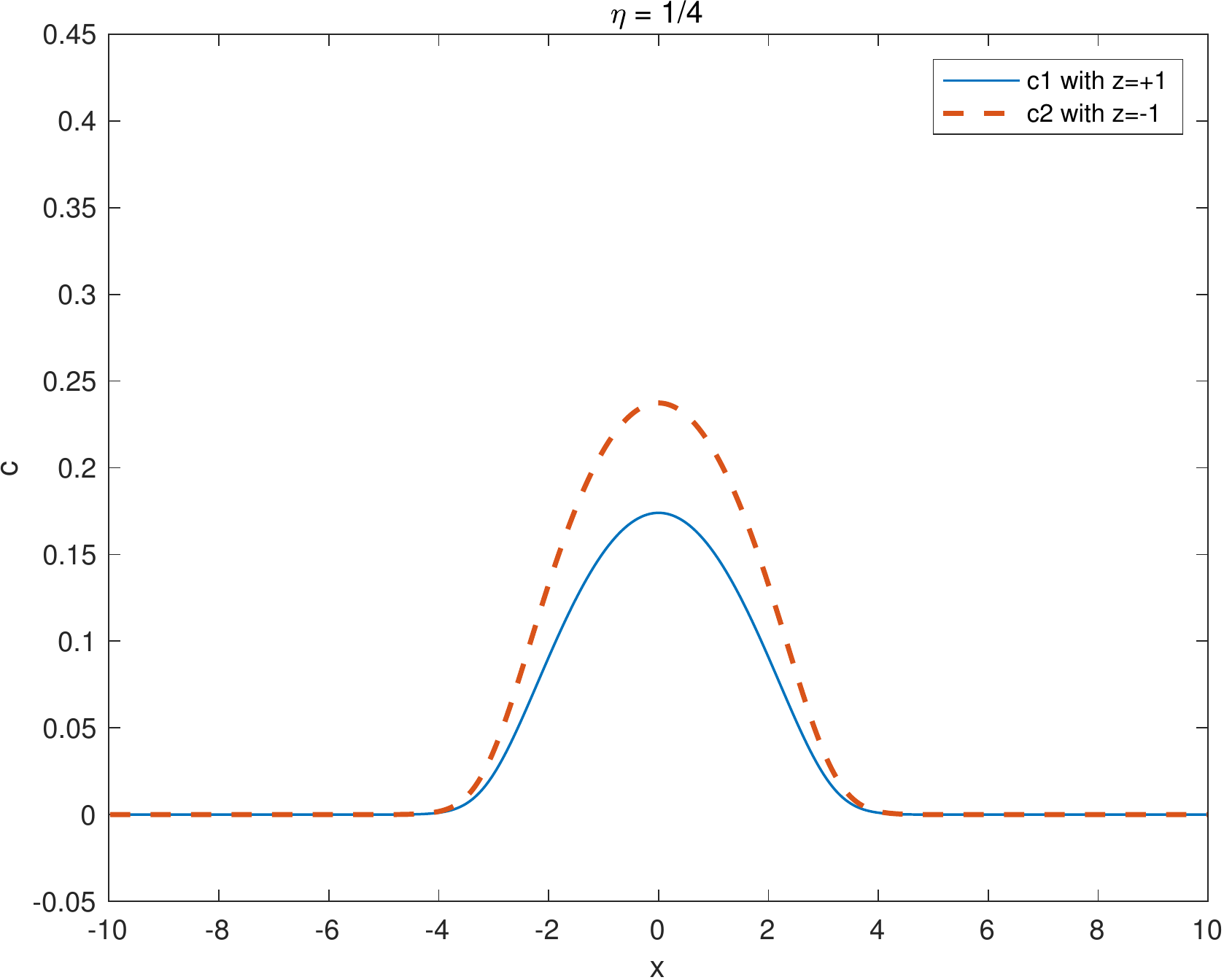}
		\end{minipage}
		\begin{minipage}[t]{0.33\linewidth}
			\centering
			\includegraphics[width=1.0\linewidth]{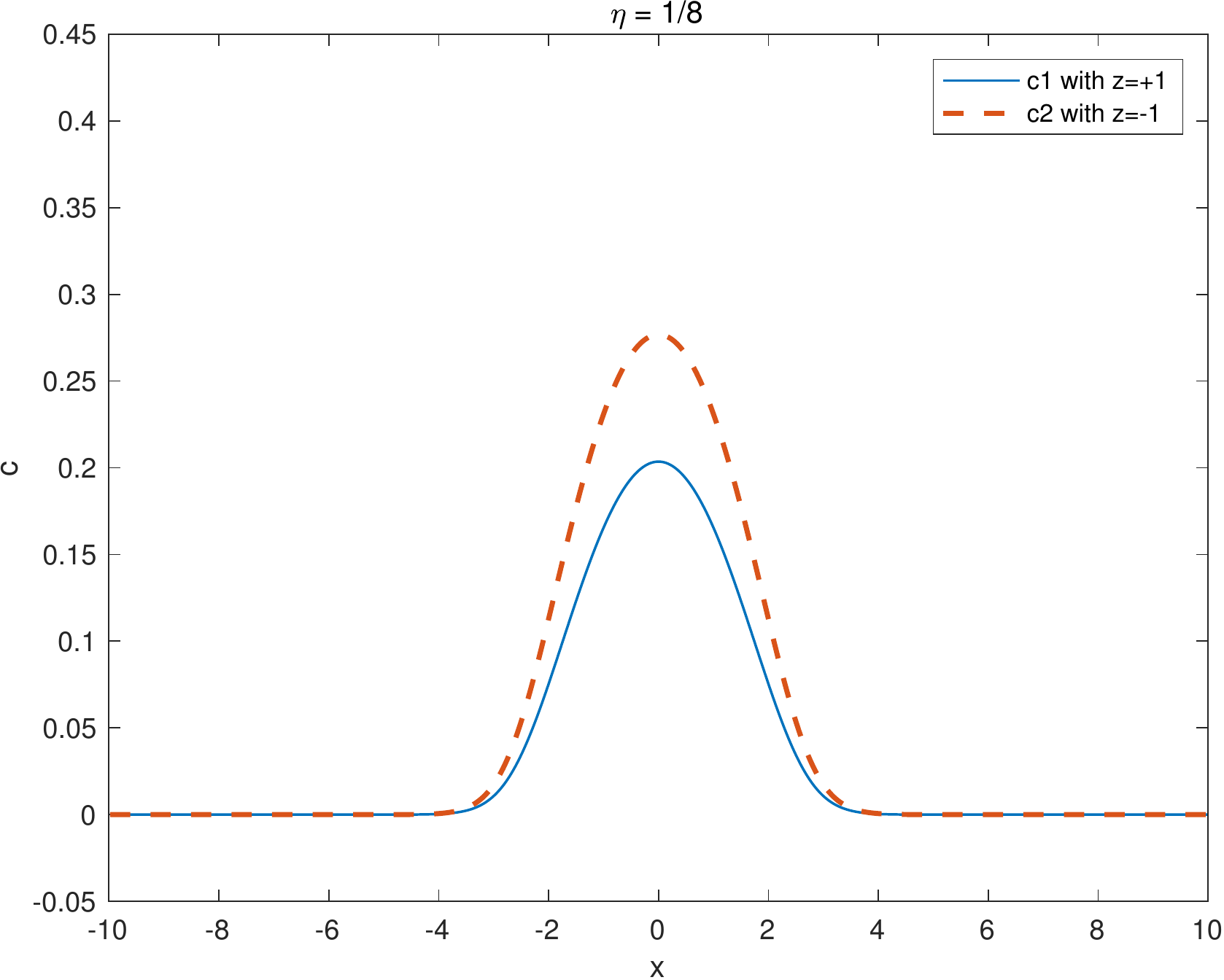}
		\end{minipage}
	}
	\subfigure{
		\begin{minipage}[t]{0.33\linewidth}
			\centering
			\includegraphics[width=1.0\linewidth]{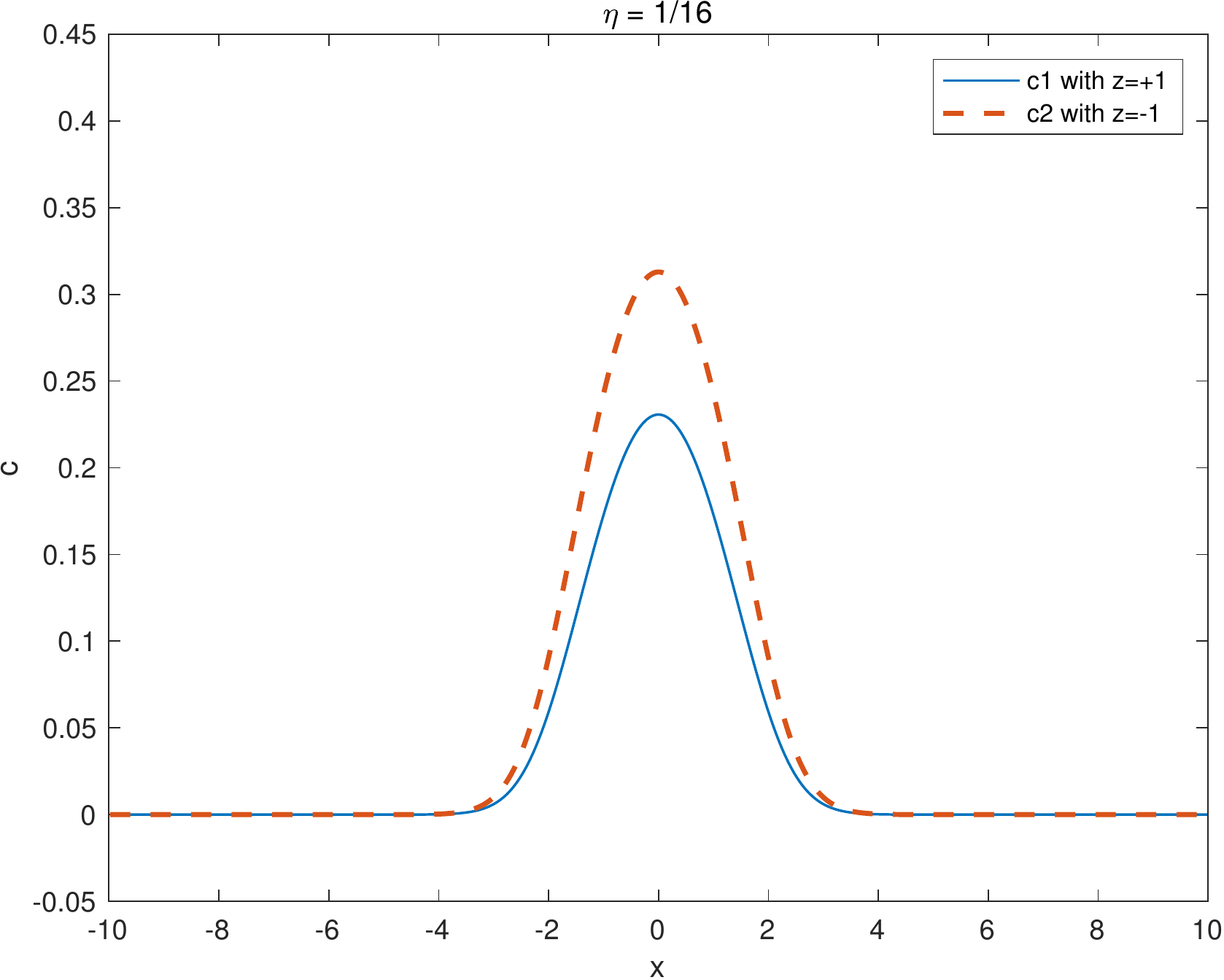}
		\end{minipage}
		\begin{minipage}[t]{0.33\linewidth}
			\centering
			\includegraphics[width=1.0\linewidth]{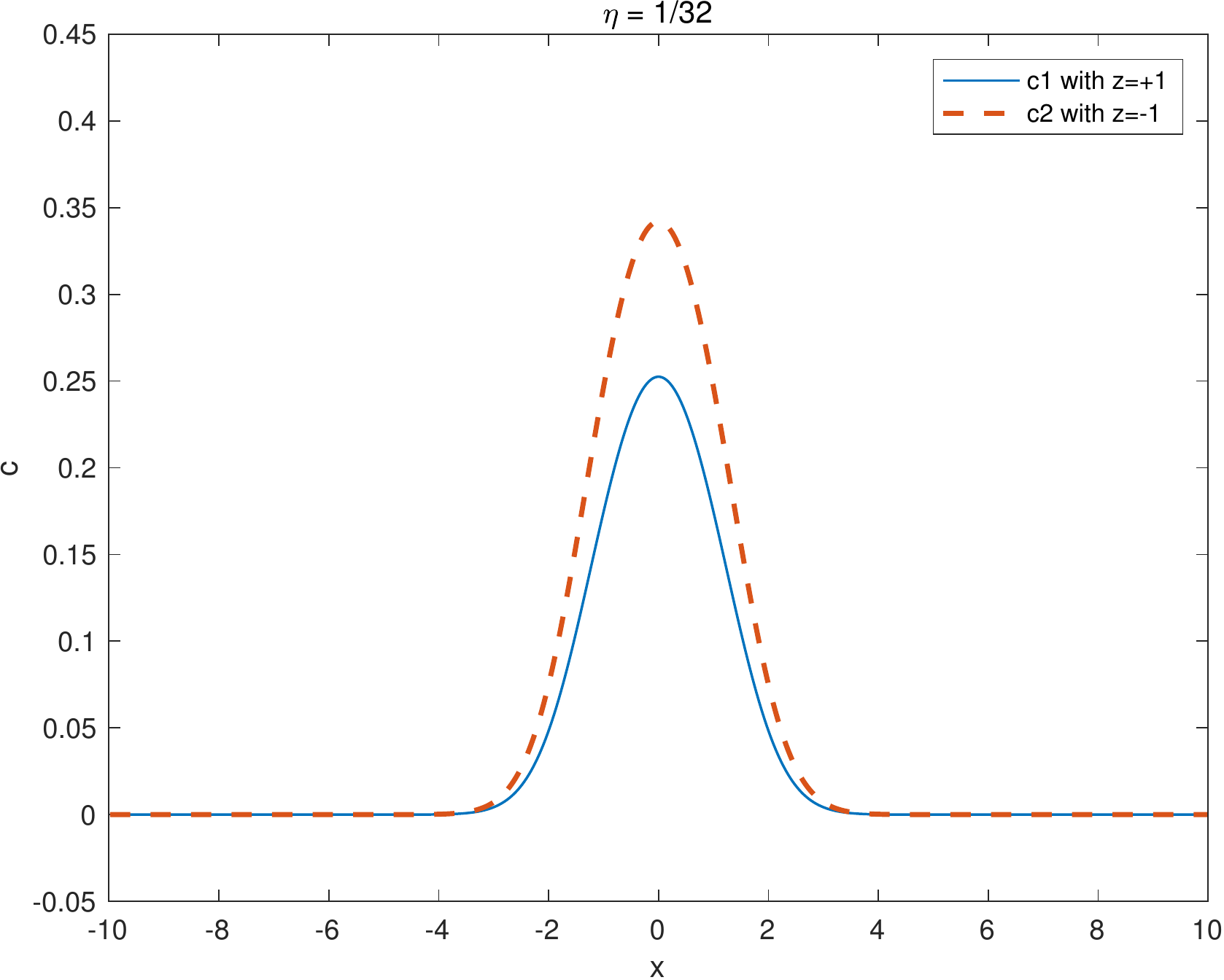}
		\end{minipage}
		\begin{minipage}[t]{0.33\linewidth}
			\centering
			\includegraphics[width=1.0\linewidth]{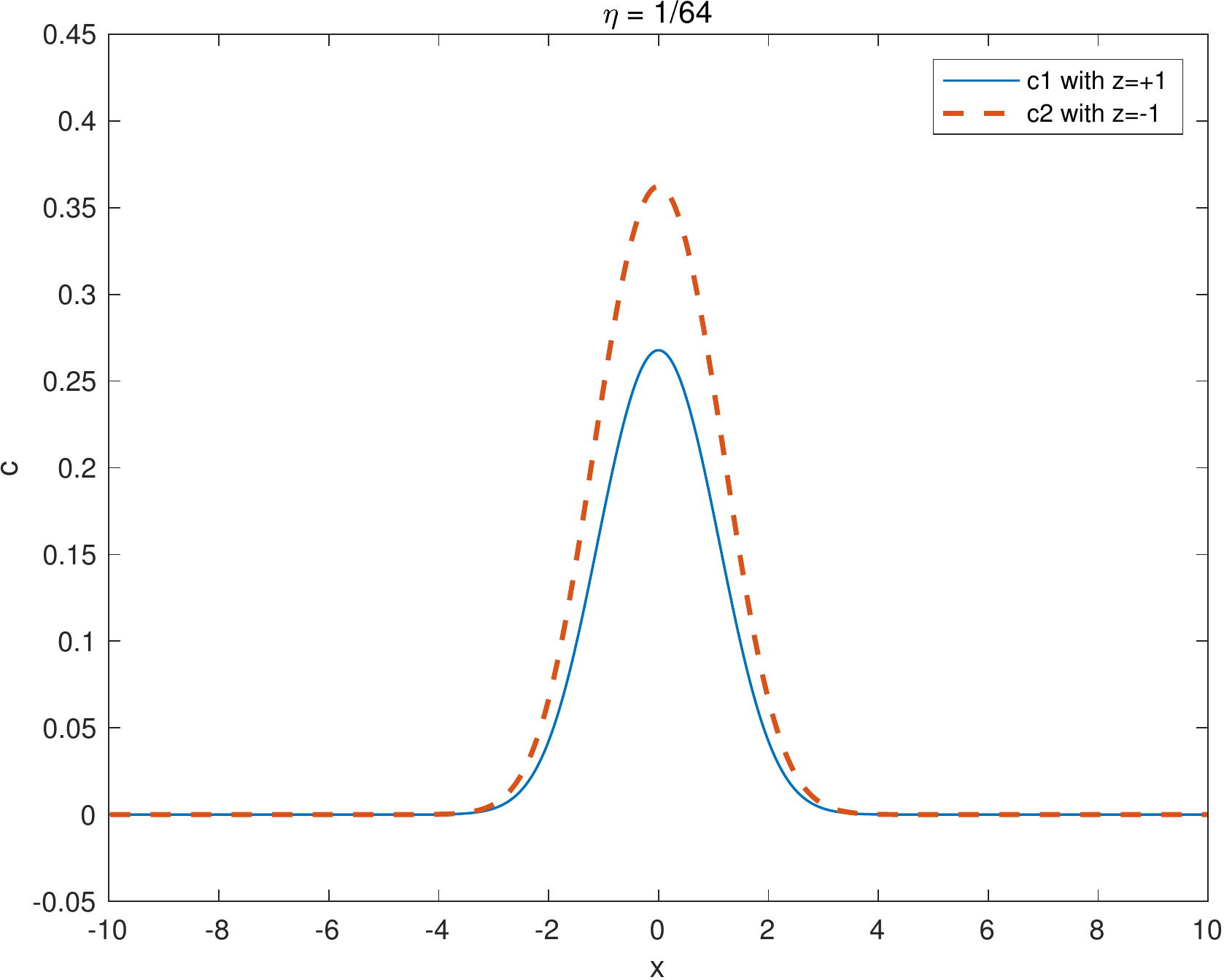}
		\end{minipage}
	}
	\subfigure{
		\begin{minipage}[t]{0.33\linewidth}
			\centering
			\includegraphics[width=1.0\linewidth]{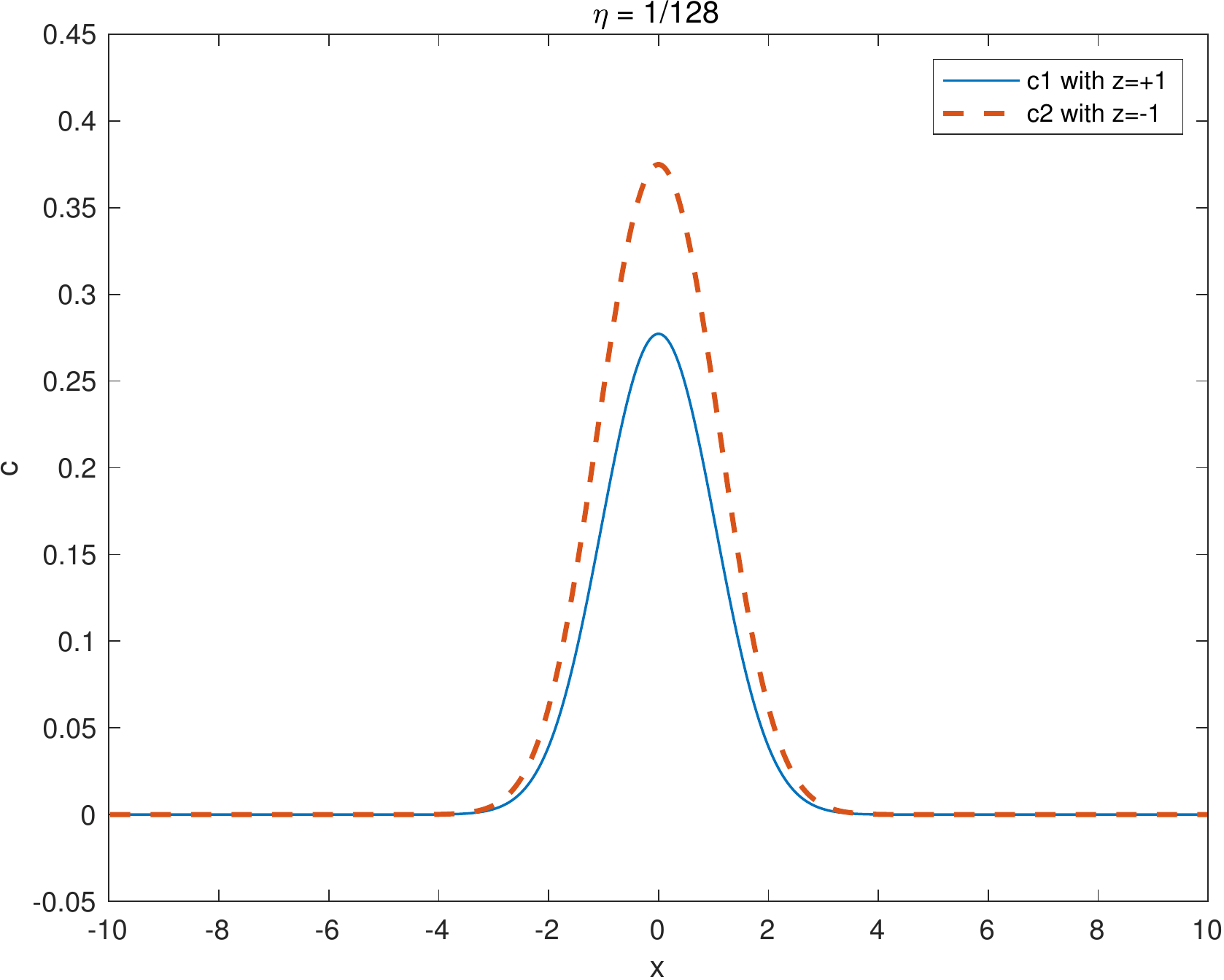}
		\end{minipage}
		\begin{minipage}[t]{0.33\linewidth}
			\centering
			\includegraphics[width=1.0\linewidth]{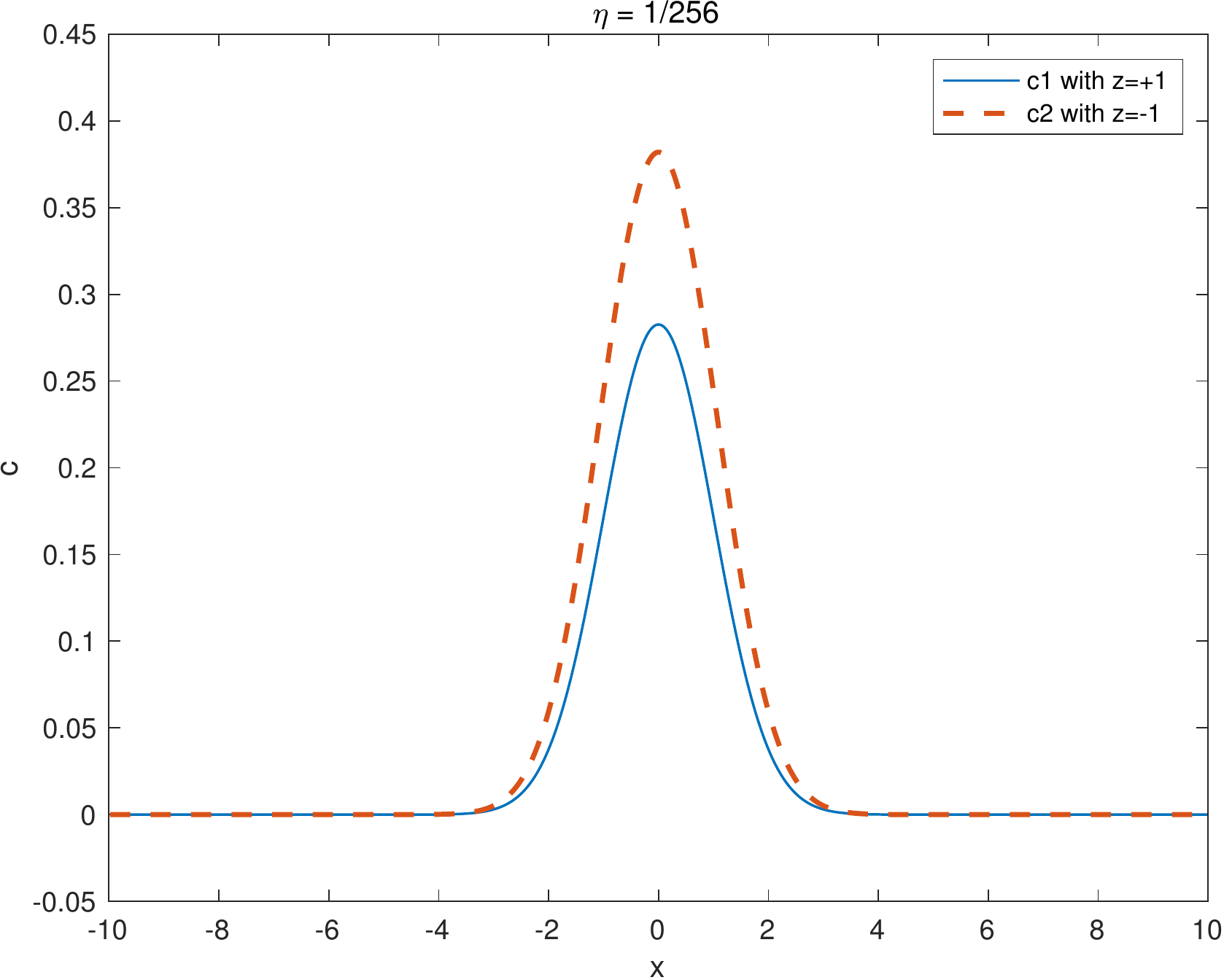}
		\end{minipage}
		\begin{minipage}[t]{0.33\linewidth}
			\centering
			\includegraphics[width=1.0\linewidth]{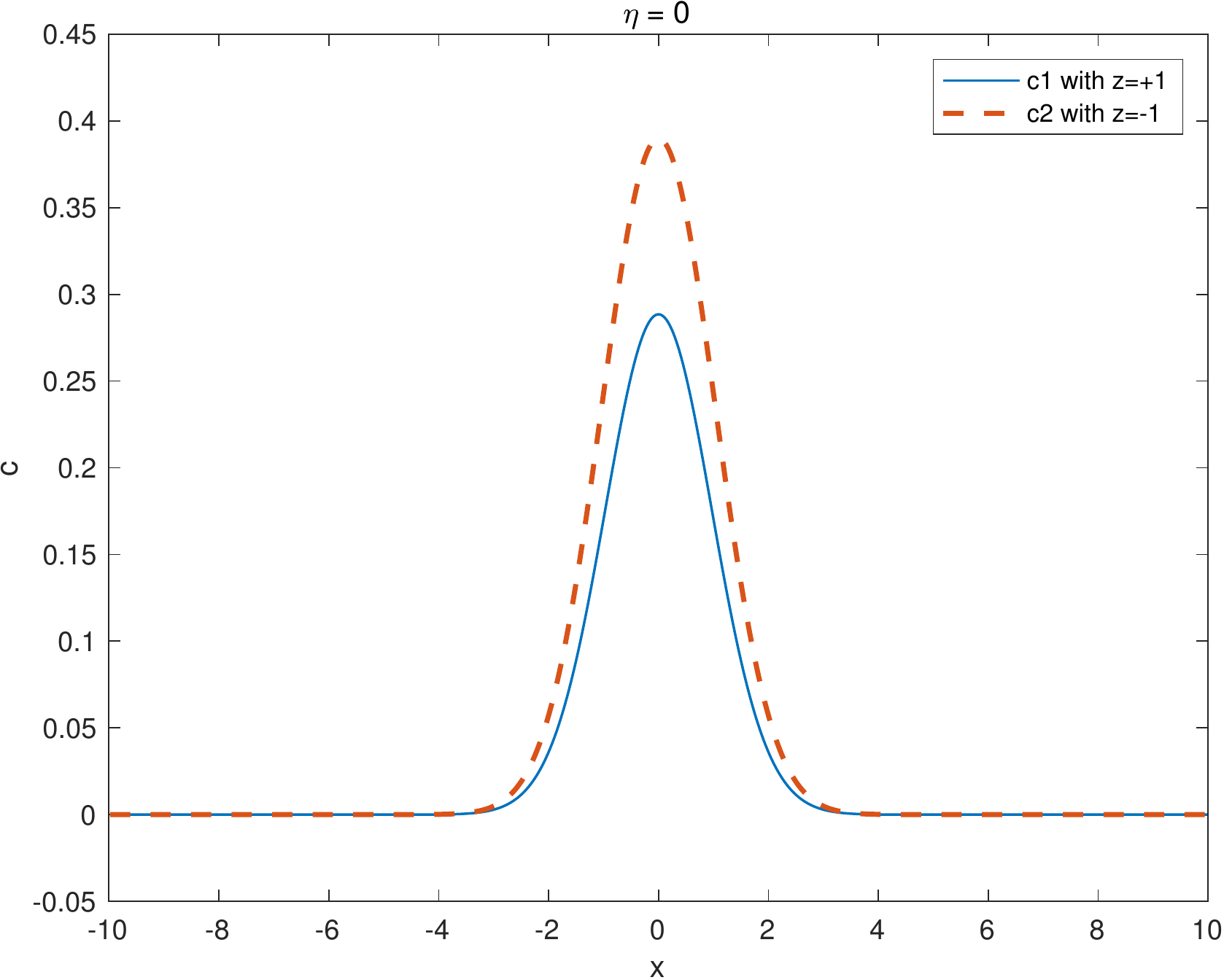}
		\end{minipage}
	}
	\caption{Multiple Species in One-dimension: The steady state density solutions $c_m$ with different $\eta$}
	\label{eta}
\end{figure}

\subsubsection{Boundary Value Problem}
If we retake the initial conditions (\ref{model4}) as  
\begin{equation}
\label{iniex2}
\left\{
\begin{array}{lll}
c_1(x, 0) = 10^{-6} &\text{with} &z_1 = 1, \\
c_2(x, 0) = \dfrac{1}{\sqrt{2 \pi}} \exp \left(-\dfrac{(x + 2)^2}{2} \right) &\text{with} &z_2 = -1.
\end{array}
\right.
\end{equation}
and the left boundary flux of $c_1$ as 
\begin{equation}
\label{bex2}
f_{-L}(t) = \dfrac{1}{\sqrt{2 \pi}} \exp \left(-\dfrac{(t - 5)^2}{2} \right),
\end{equation}
then Figure \ref{5c} shows how the density solutions $c_m, m = 1, 2$, develop with time $t$ and converge to the equilibrium.
\begin{figure}[htp]
	\subfigure{
		\begin{minipage}[t]{0.33\linewidth}
			\centering
			\includegraphics[width=1.0\linewidth]{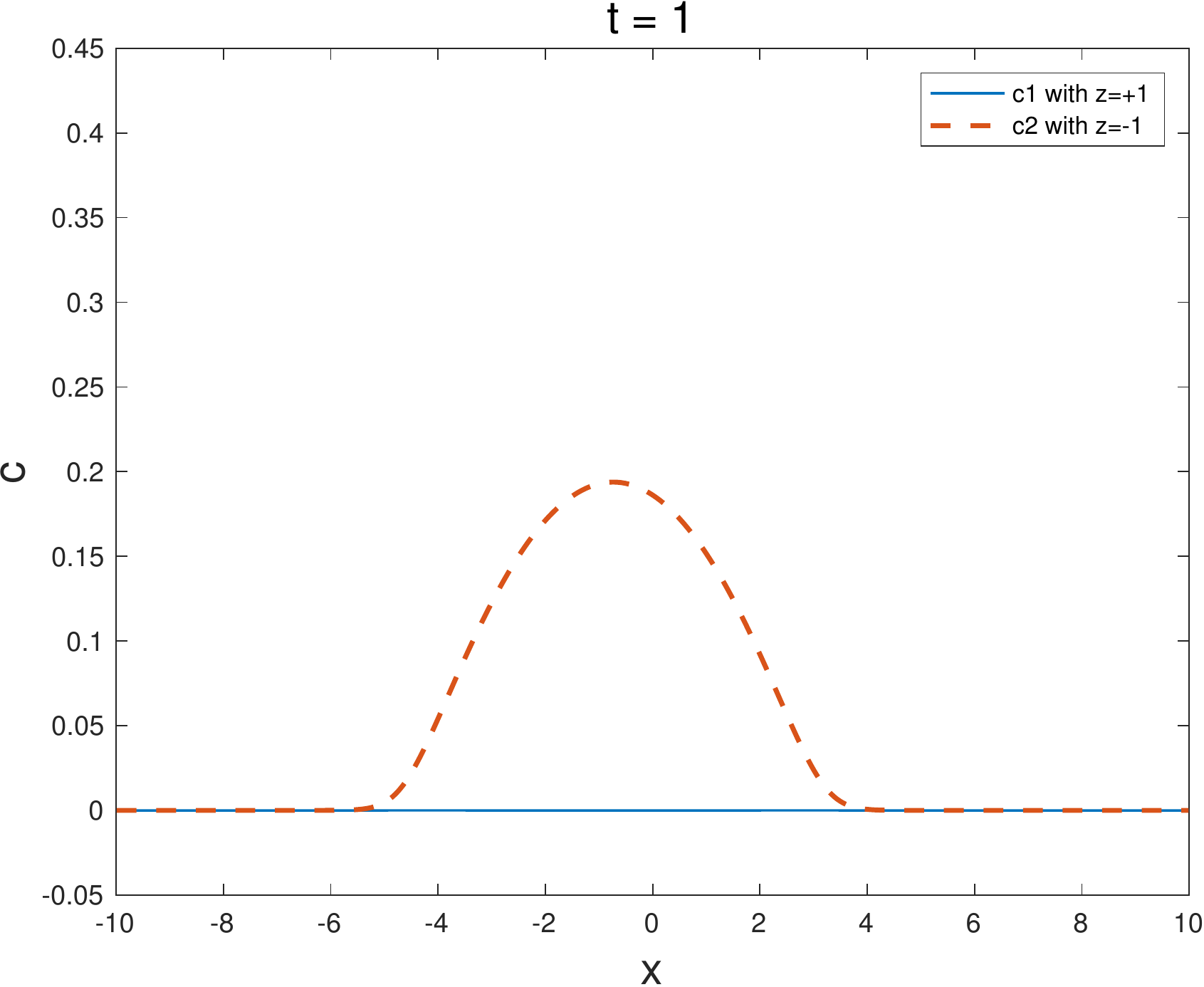}
		\end{minipage}
		\begin{minipage}[t]{0.33\linewidth}
			\centering
			\includegraphics[width=1.0\linewidth]{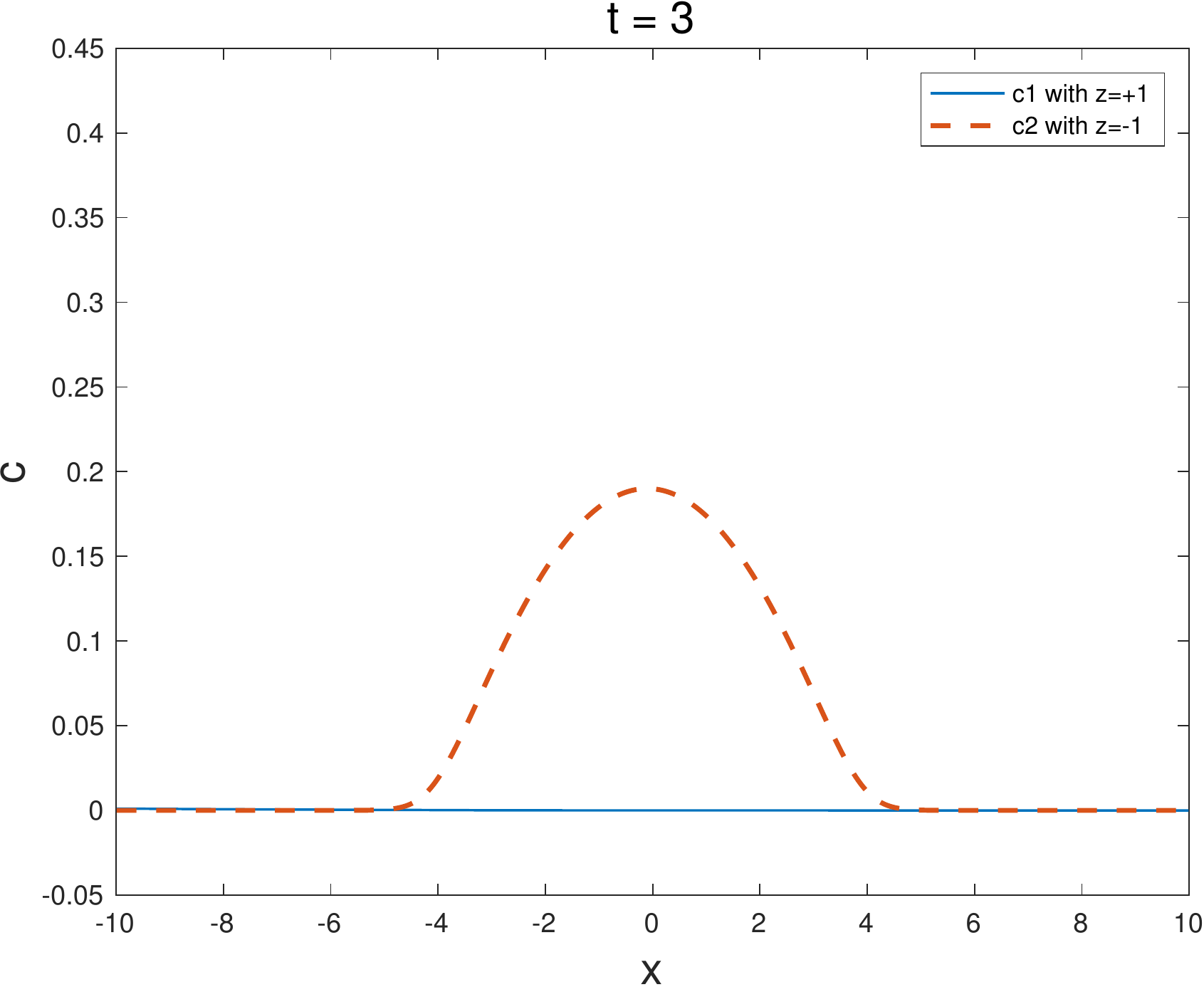}
		\end{minipage}
		\begin{minipage}[t]{0.33\linewidth}
			\centering
			\includegraphics[width=1.0\linewidth]{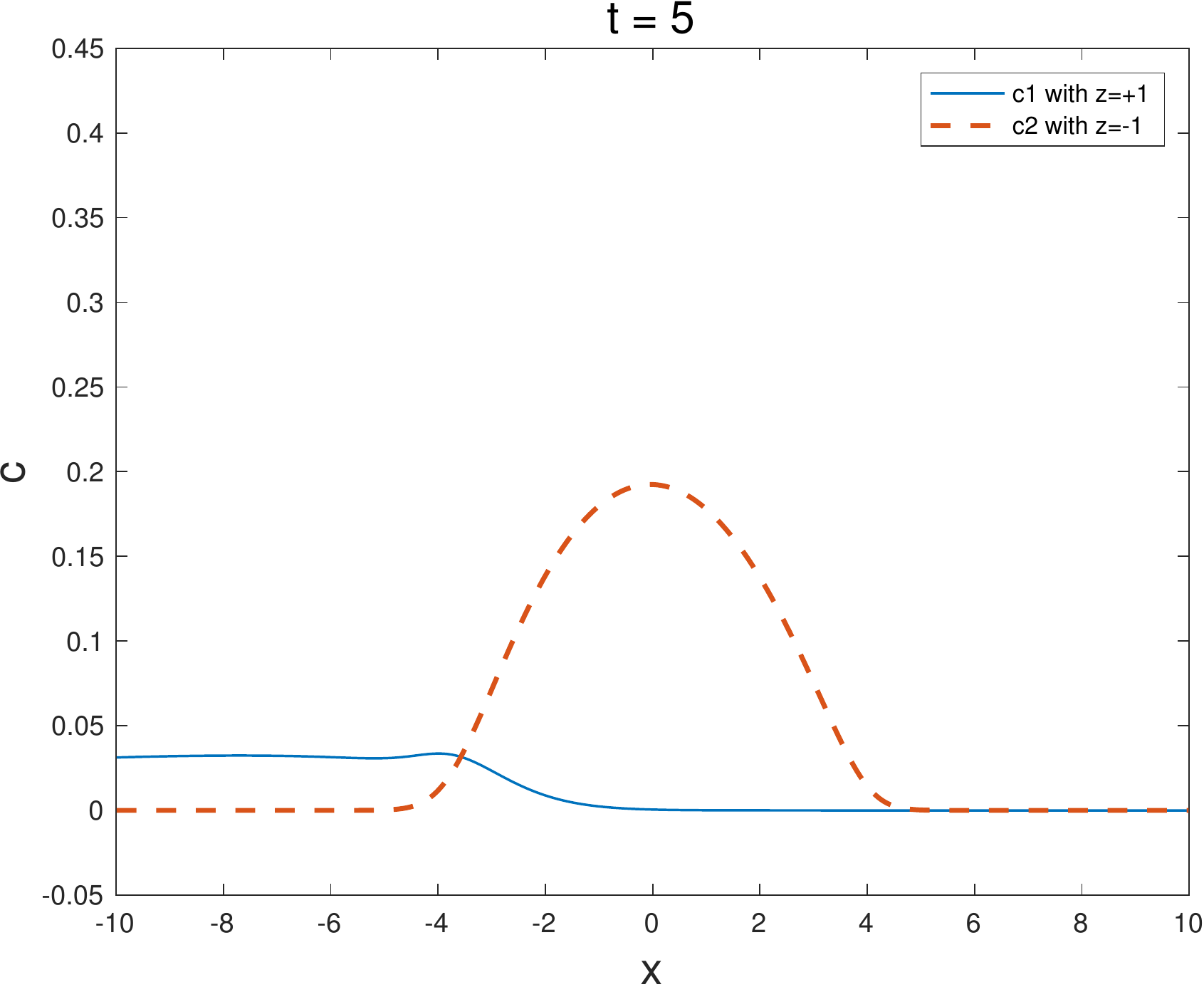}
		\end{minipage}
	}
	\subfigure{
		\begin{minipage}[t]{0.33\linewidth}
			\centering
			\includegraphics[width=1.0\linewidth]{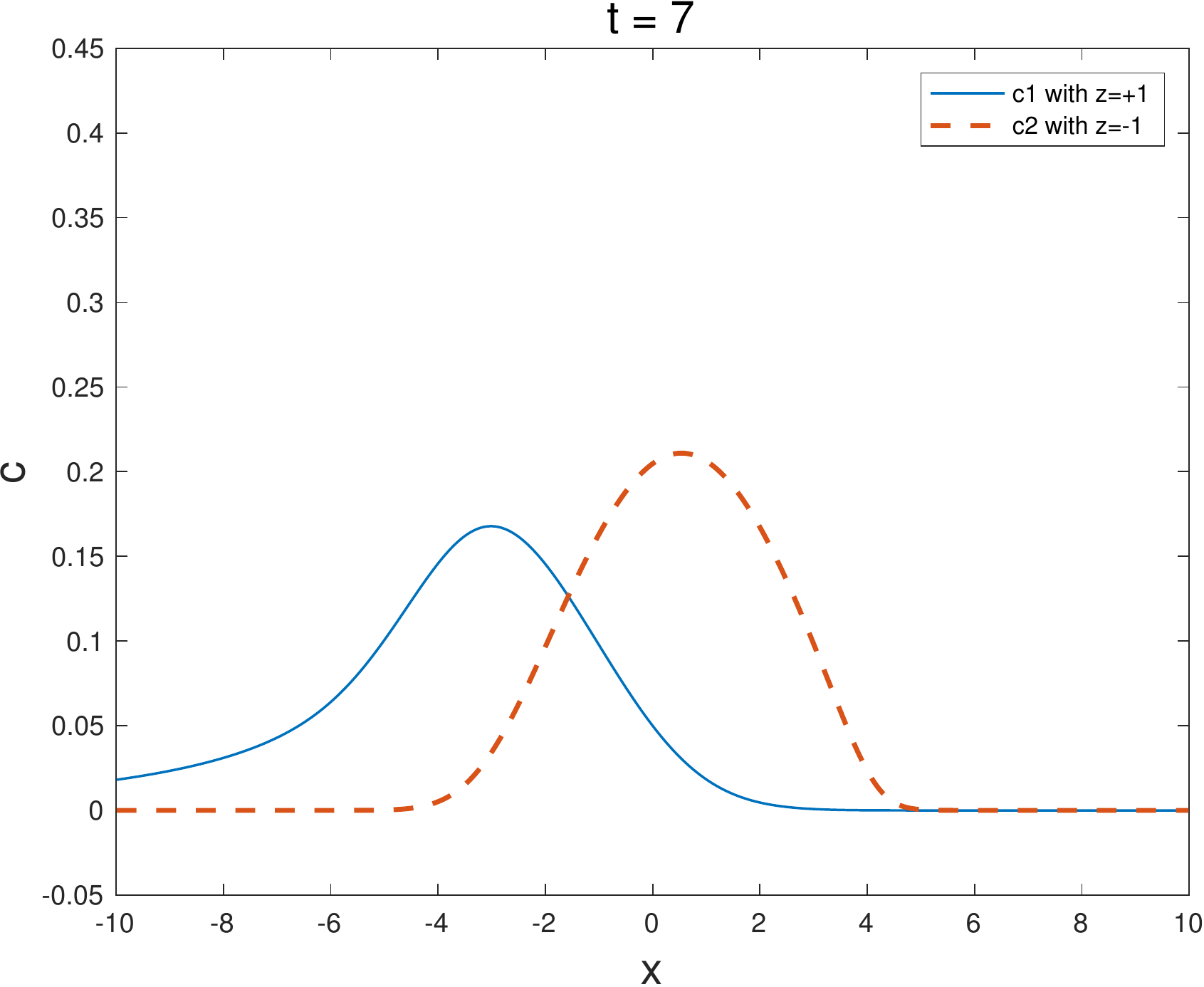}
		\end{minipage}
		\begin{minipage}[t]{0.33\linewidth}
			\centering
			\includegraphics[width=1.0\linewidth]{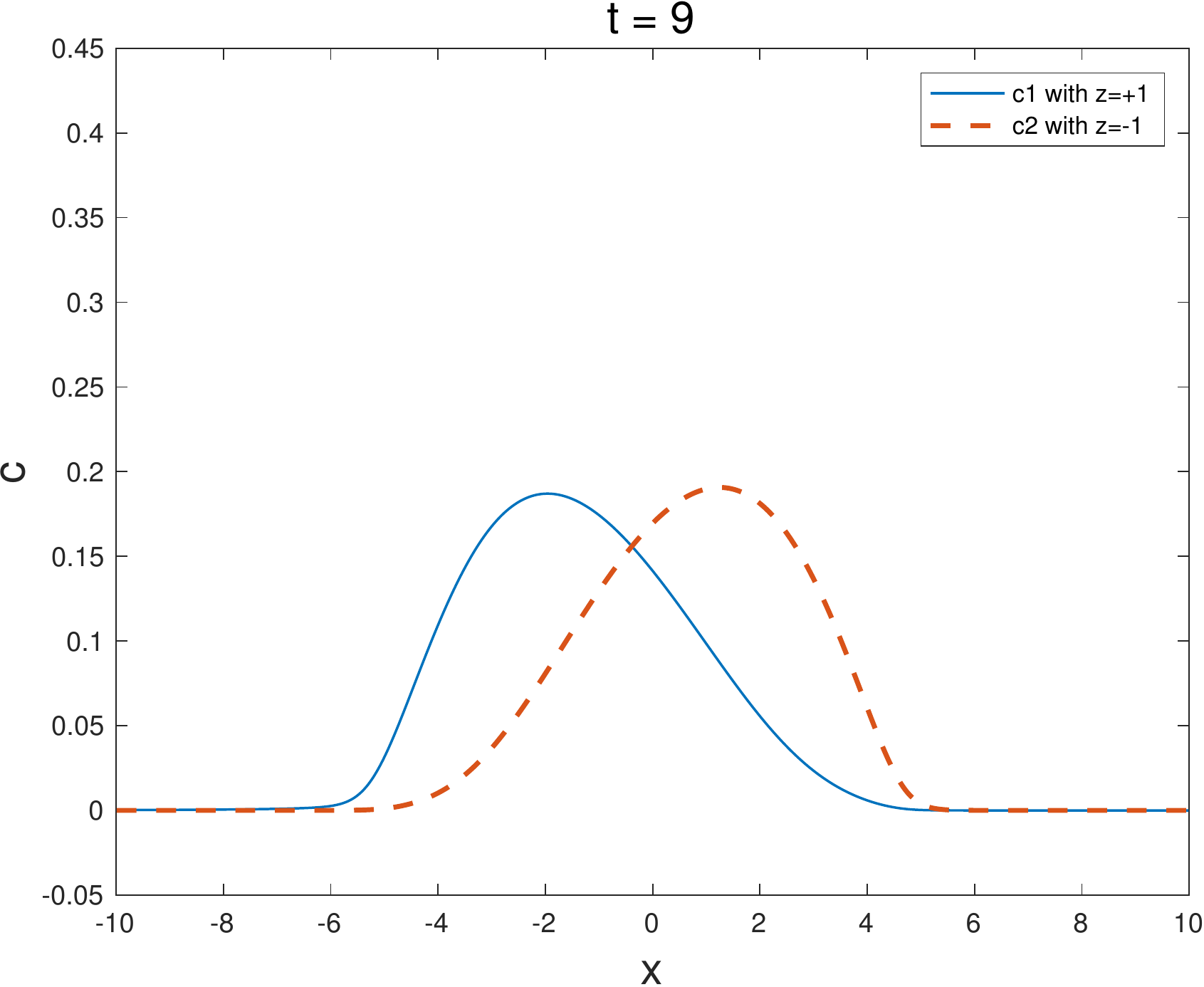}
		\end{minipage}
		\begin{minipage}[t]{0.33\linewidth}
			\centering
			\includegraphics[width=1.0\linewidth]{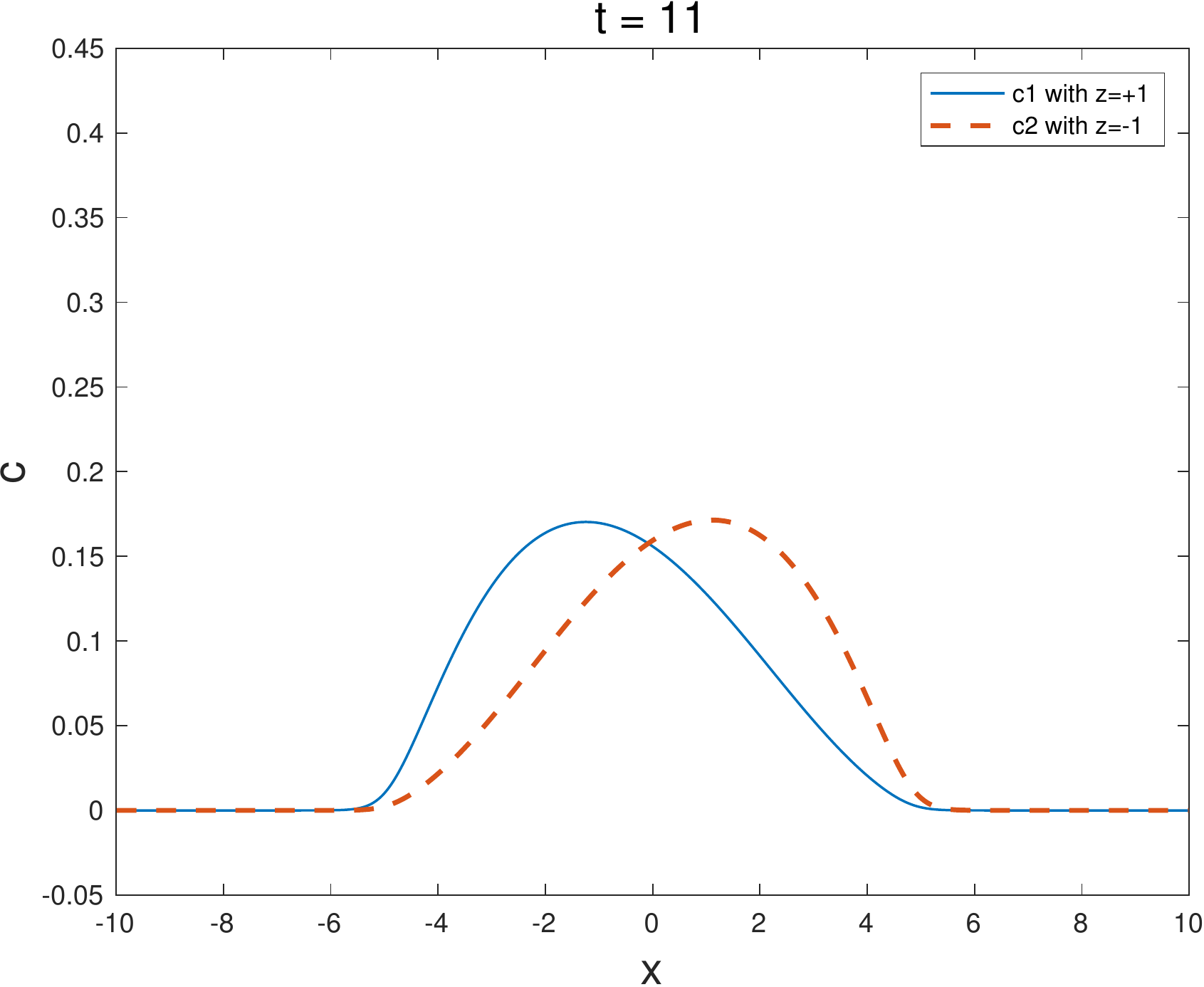}
		\end{minipage}
	}
	\subfigure{
		\begin{minipage}[t]{0.33\linewidth}
			\centering
			\includegraphics[width=1.0\linewidth]{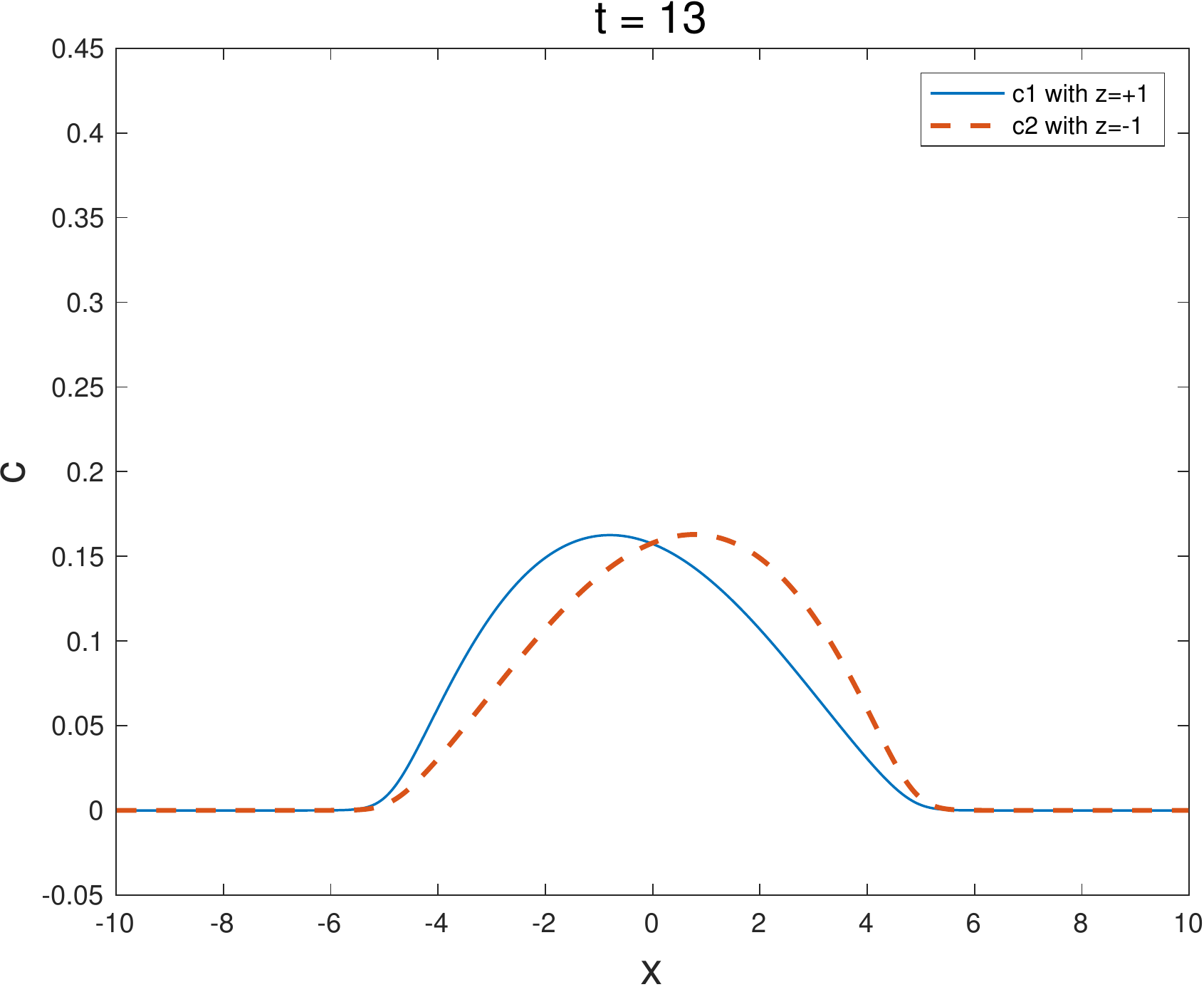}
		\end{minipage}
		\begin{minipage}[t]{0.33\linewidth}
			\centering
			\includegraphics[width=1.0\linewidth]{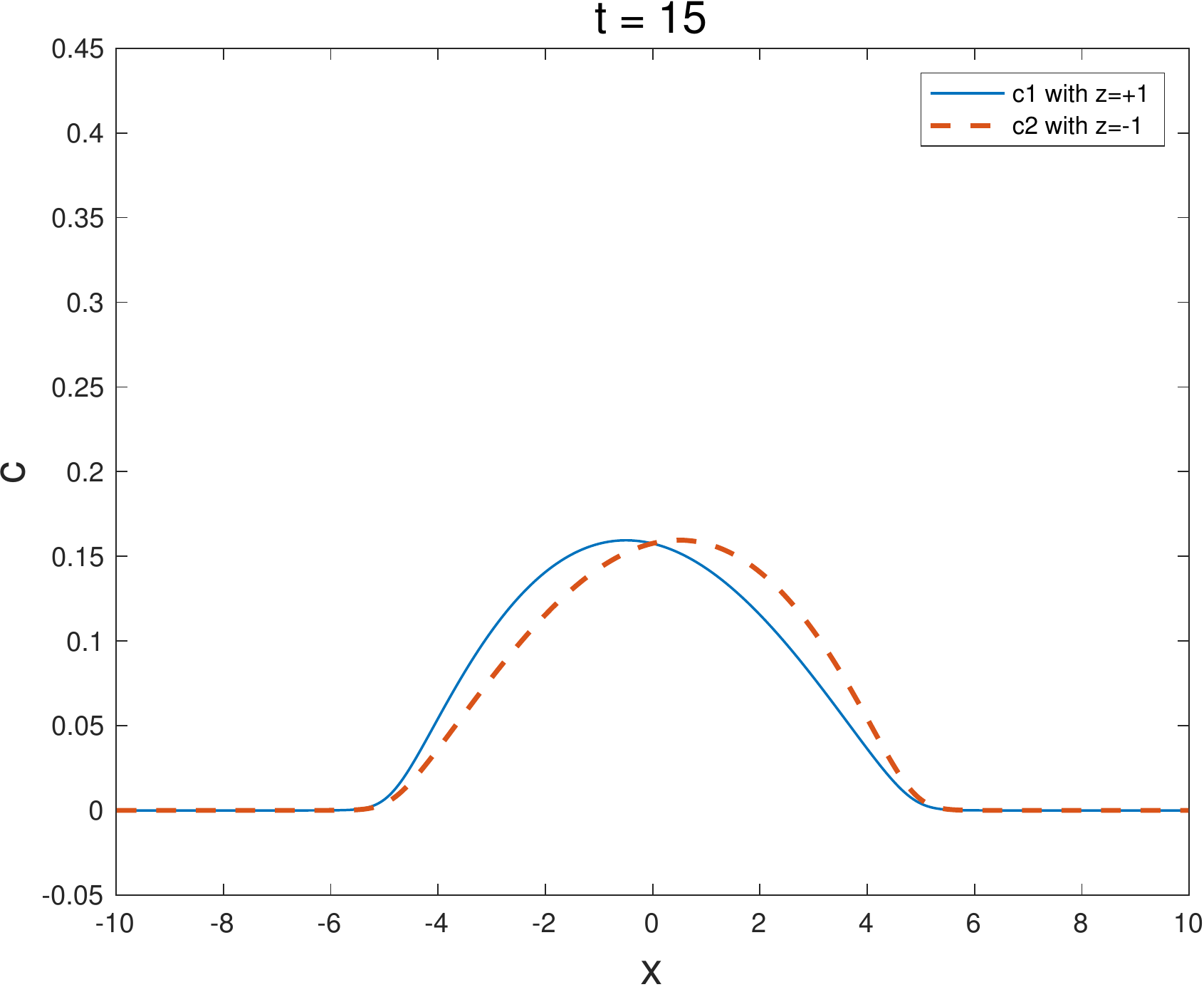}
		\end{minipage}
		\begin{minipage}[t]{0.33\linewidth}
			\centering
			\includegraphics[width=1.0\linewidth]{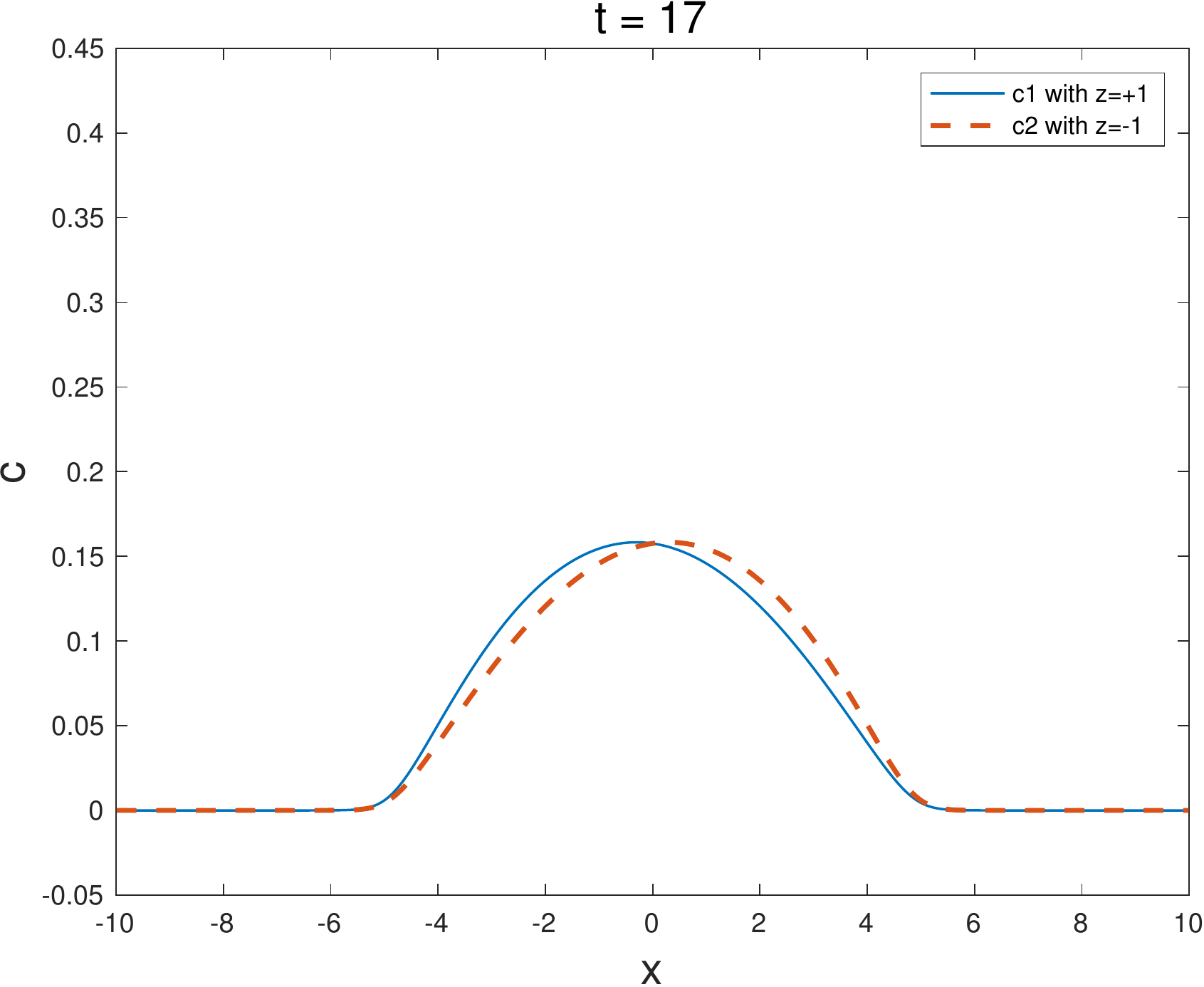}
		\end{minipage}
	}
	\subfigure{
		\begin{minipage}[t]{0.33\linewidth}
			\centering
			\includegraphics[width=1.0\linewidth]{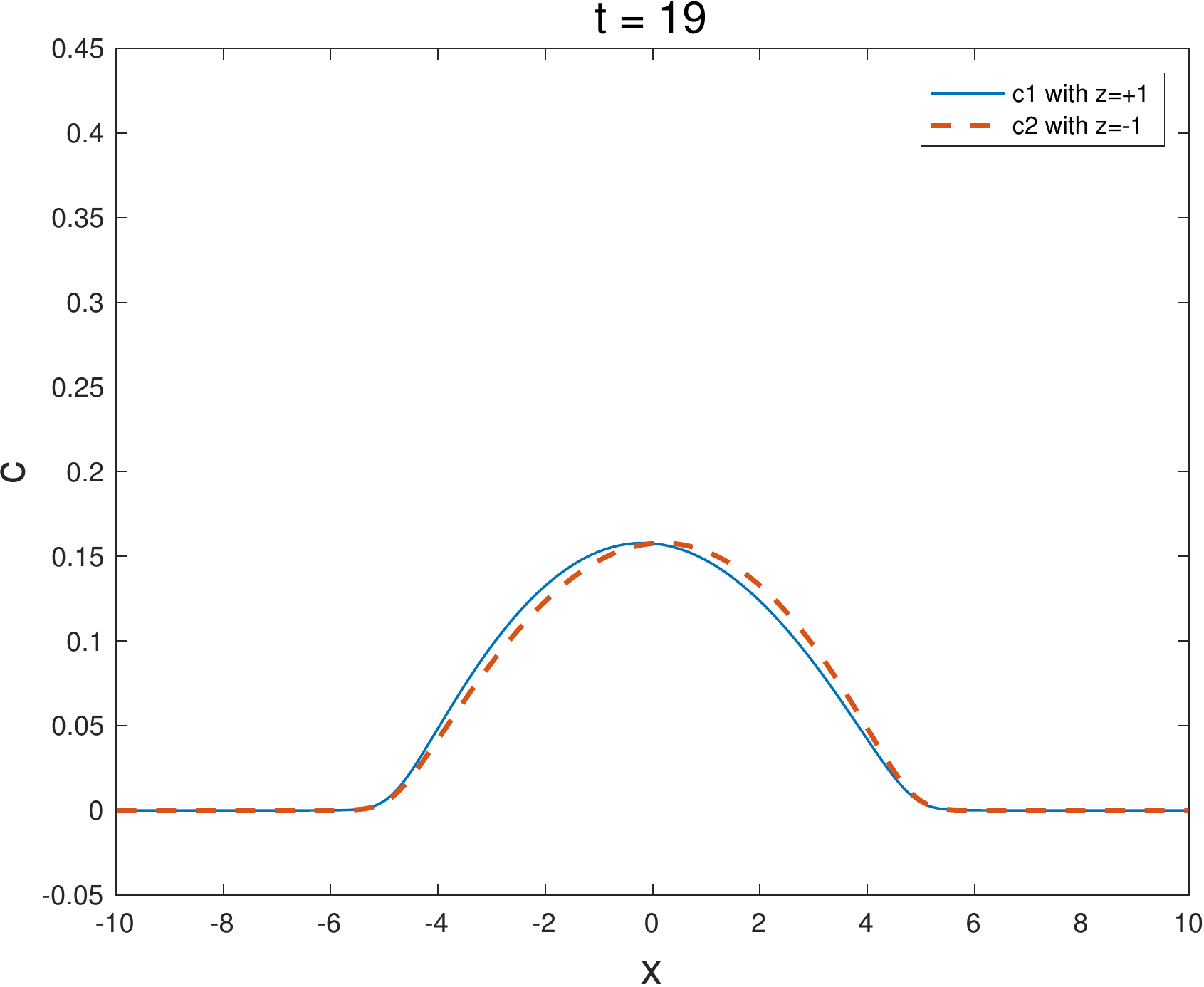}
		\end{minipage}
		\begin{minipage}[t]{0.33\linewidth}
			\centering
			\includegraphics[width=1.0\linewidth]{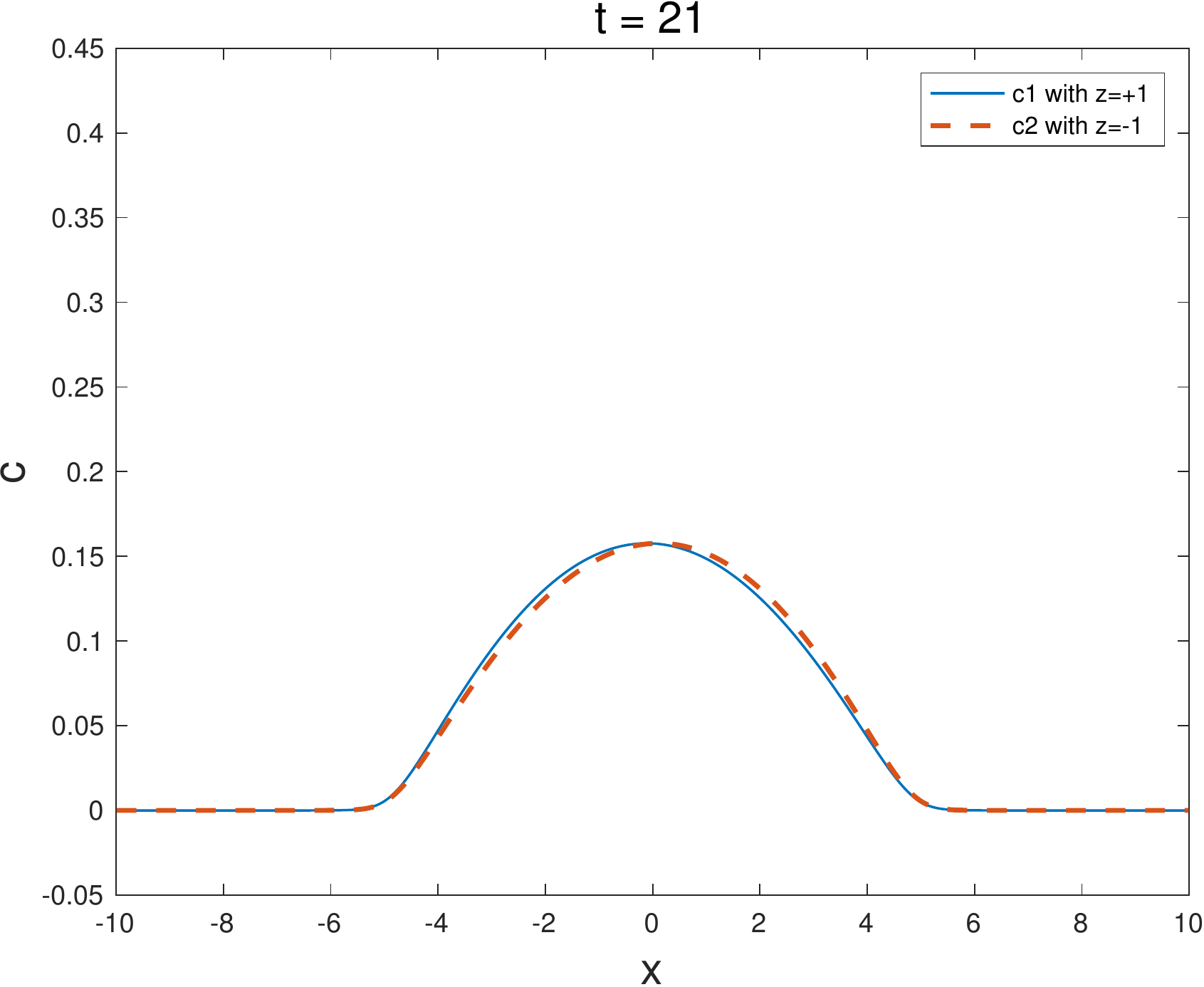}
		\end{minipage}
		\begin{minipage}[t]{0.33\linewidth}
			\centering
			\includegraphics[width=1.0\linewidth]{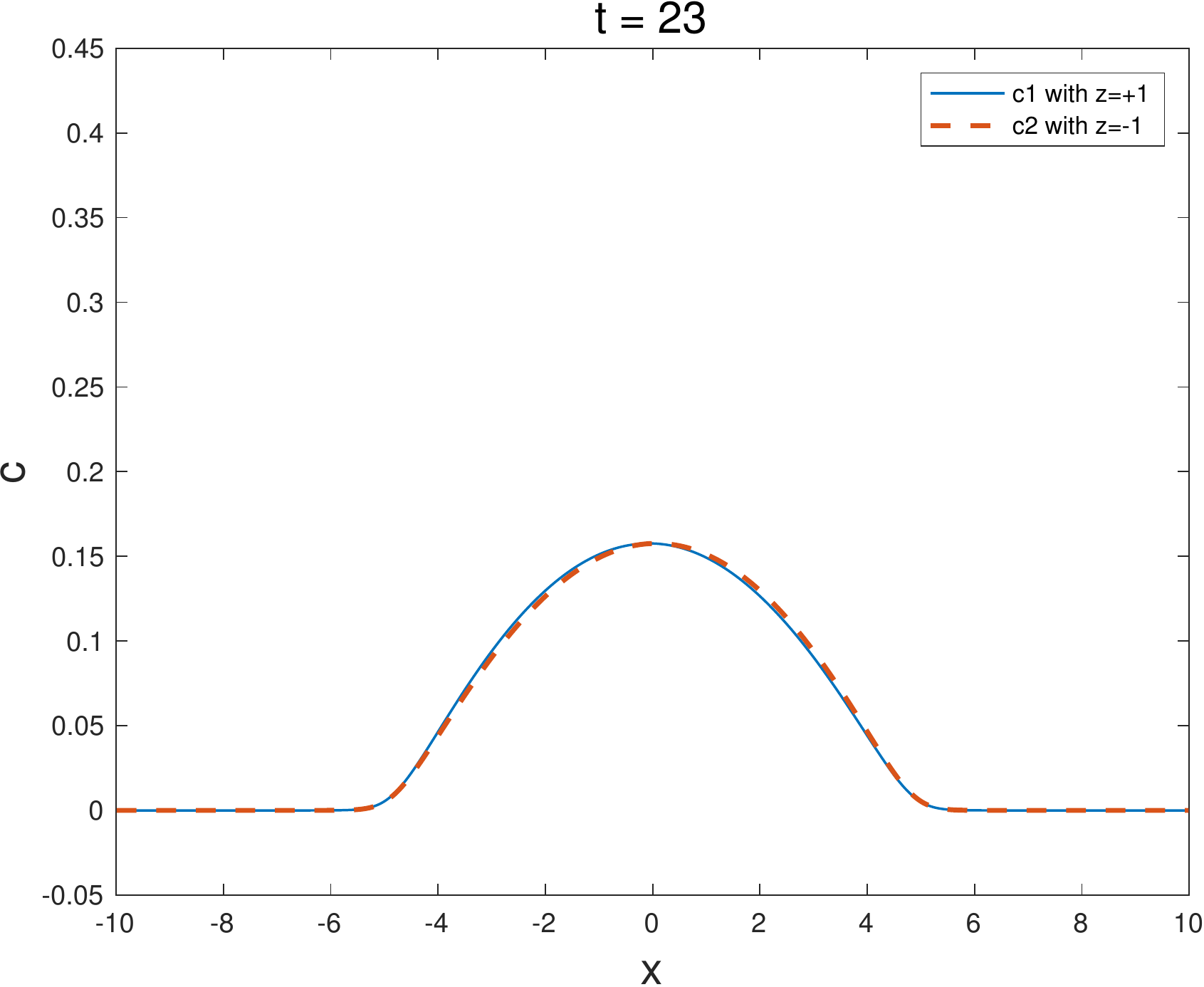}
		\end{minipage}
	}
	\caption{Multiple Species in One-dimension: The space-concentration curves with the mesh size $\Delta x$ being 0.01953125 and the time $t$ changing from 1 to 23}
	\label{5c}
\end{figure}

Similarly, Figure \ref{im333}(a) shows how  the discrete forms of the energy $\mathcal{F}_1, \mathcal{F}_2, \mathcal{F}_3, \mathcal{F}_4, \mathcal{F}$ changes with time $t$ and Figure \ref{im333}(b) shows the discrete chemical potential $\psi_m$ at time $t = 23$, which goes to a constant while the model goes to the equilibrium for all $m$.
\begin{figure}[htp] 
	\hspace*{0.05\textwidth}
	\subfigure[]{ 
		\includegraphics[width=0.45\textwidth]{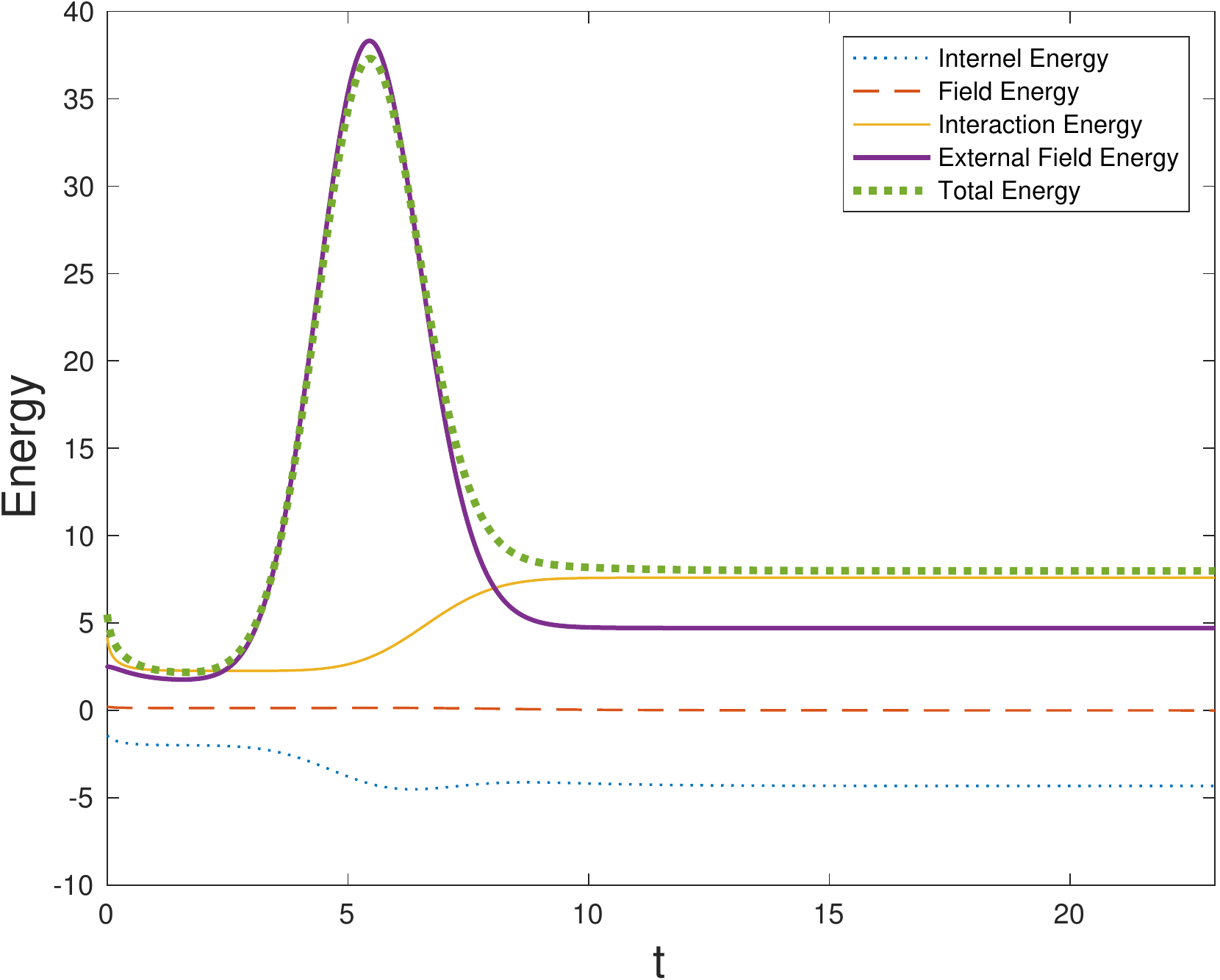}} 
	\subfigure[]{ 
		\includegraphics[width=0.45\textwidth]{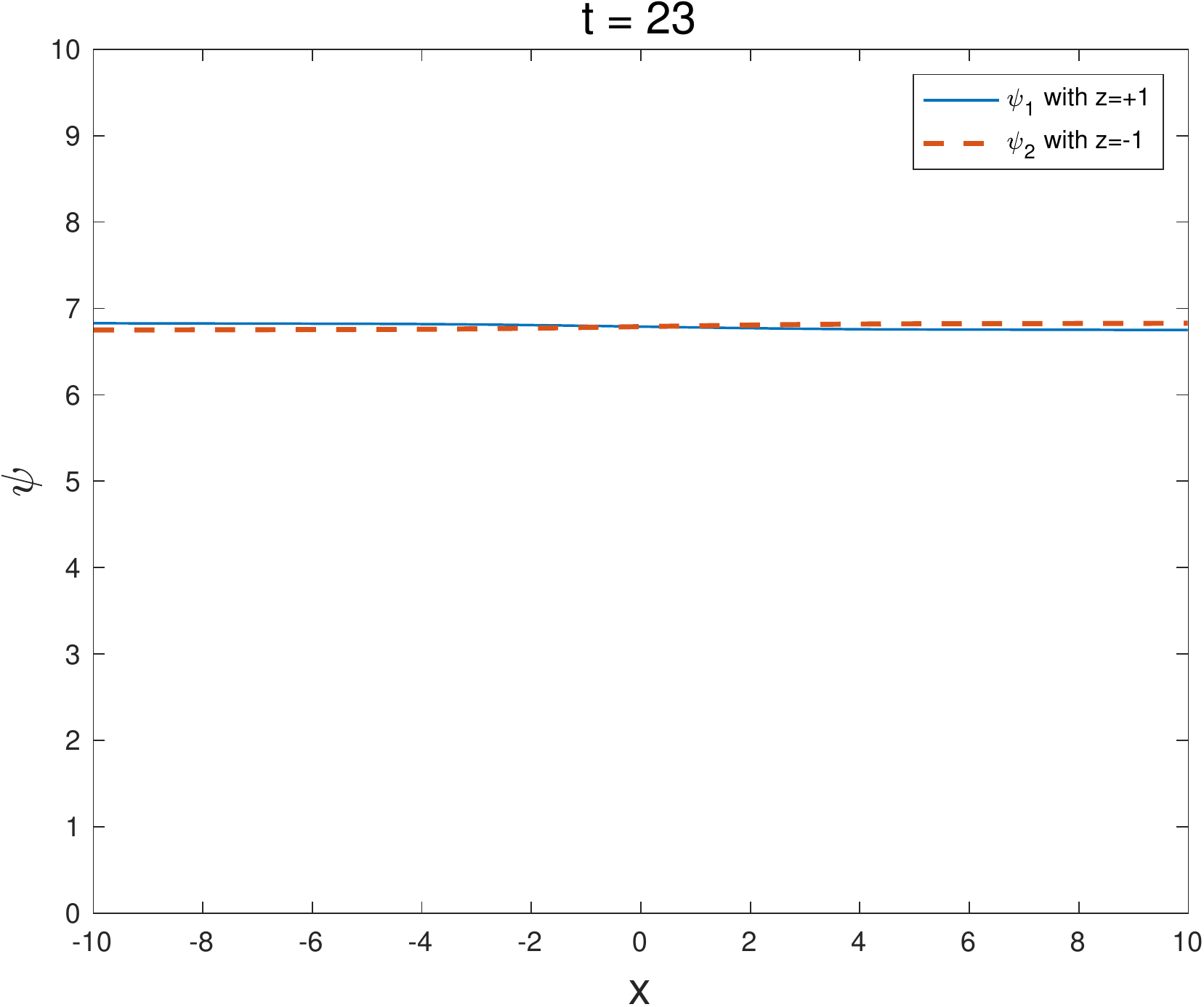}} 
	\caption{Multiple Species in One-dimension: (a): The time-energy plot of the model (\ref{model3}) equipped with the initial conditions (\ref{iniex2}) and the boundary condition (\ref{bex2}) with the mesh size $\Delta x$ being 0.01953125. (b): Discrete chemical potential at time $t = 23$ with the mesh size $\Delta x$ being 0.01953125.}
	\label{im333}
\end{figure}

\subsection{Multiple Species in Two-dimension}
\subsubsection{Steady State}
Here we consider the two-dimensional kernel $\mathcal{W}(x, y) = \frac{1}{r^2 + \epsilon^2}, r = \sqrt{x^2 + y^2}, \epsilon = \frac{1}{10}$, $\mathcal{K}(x, y) = - \frac{1}{2 \pi} \log(\sqrt{r^2 + \epsilon^2})$, $V_{\text{ext}}(x, y) = \frac{1}{2} r^2$ and the initial conditions (\ref{model4}) are given by the following form
\begin{equation}
\label{iniex3}
\left\{
\begin{array}{lll}
c_1^0 = \dfrac{1}{\sqrt{2 \pi}} \exp \left(-\dfrac{(x - 2)^2 + (y - 2)^2}{2} \right) &\text{with} &z_1 = 1, \\
c_2^0 = \dfrac{1}{\sqrt{2 \pi}} \exp \left(-\dfrac{(x + 2)^2 + (y + 2)^2}{2} \right) &\text{with} &z_2 = -1.
\end{array}
\right.
\end{equation}
Here, we retake $\eta = 1$, the computation domain as $[-L, L] \times [-L, L], \ L = 10$ and the mesh size $\Delta x = \Delta y = 0.0390625$, $\Delta t$ is determined by (\ref{cfl2}), here $\Delta t = \max \left\{\dfrac{\Delta x}{5 ~ U_{\text{max}}}, \dfrac{\Delta y}{5 ~ V_{\text{max}}} \right\}$. And we omit it in other two-dimensional examples. Figure \ref{c1oft} and Figure \ref{c2oft} show how the concentrations of the $m$-th ionic species $c_m, m = 1, 2$, change with time $t$ respectively. And Figure \ref{eoft2d} shows the relation between the time $t$ and the discrete forms of the energy $\mathcal{F}, \mathcal{F}_1, \mathcal{F}_2, \mathcal{F}_3, \mathcal{F}_4$. 

\begin{figure}[htp] 
	\centering
	\subfigure{
		\begin{minipage}[t]{0.35\linewidth}
			\centering
			\includegraphics[width=1.0\linewidth]{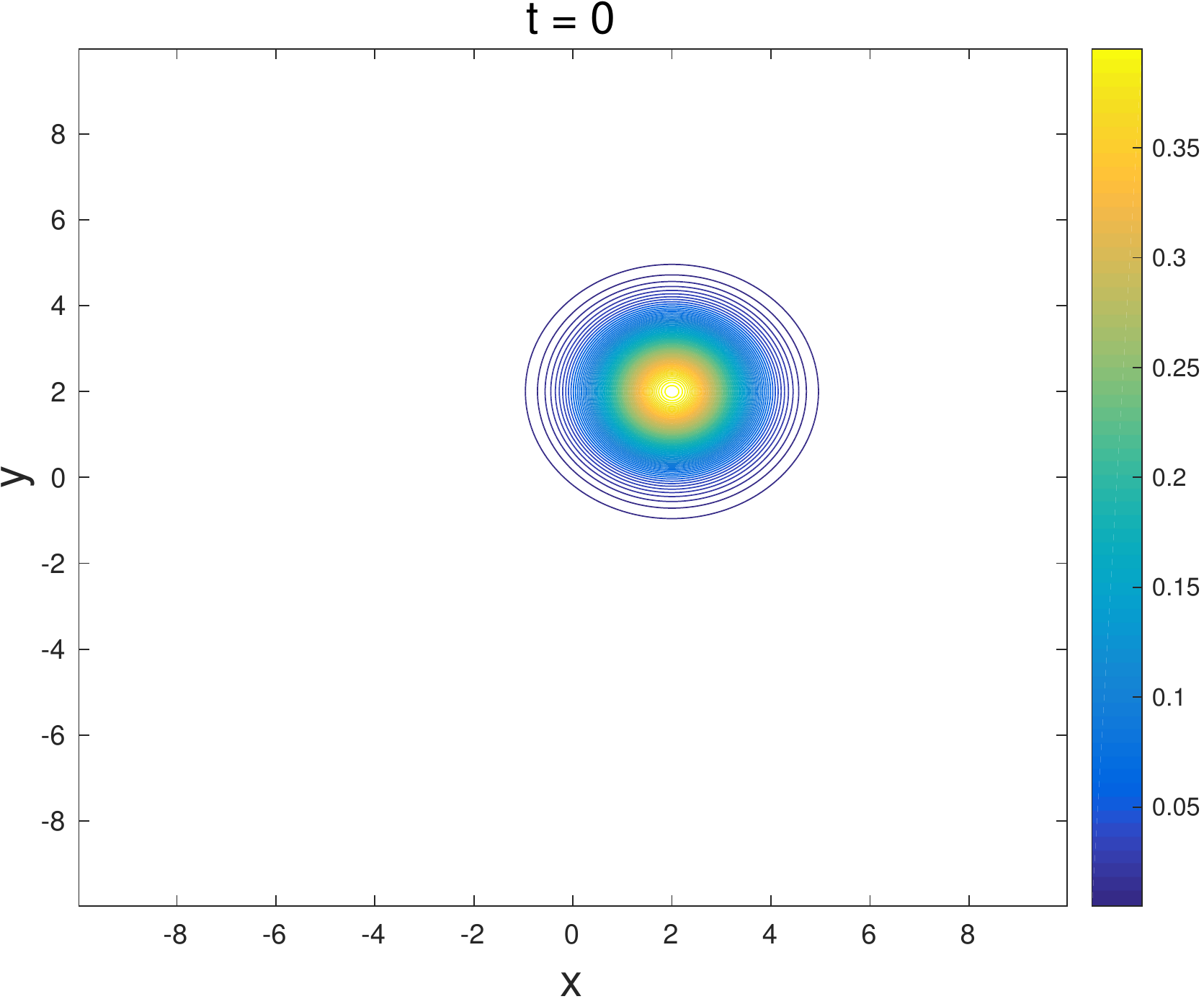}
		\end{minipage}
		\begin{minipage}[t]{0.35\linewidth}
			\centering
			\includegraphics[width=1.0\linewidth]{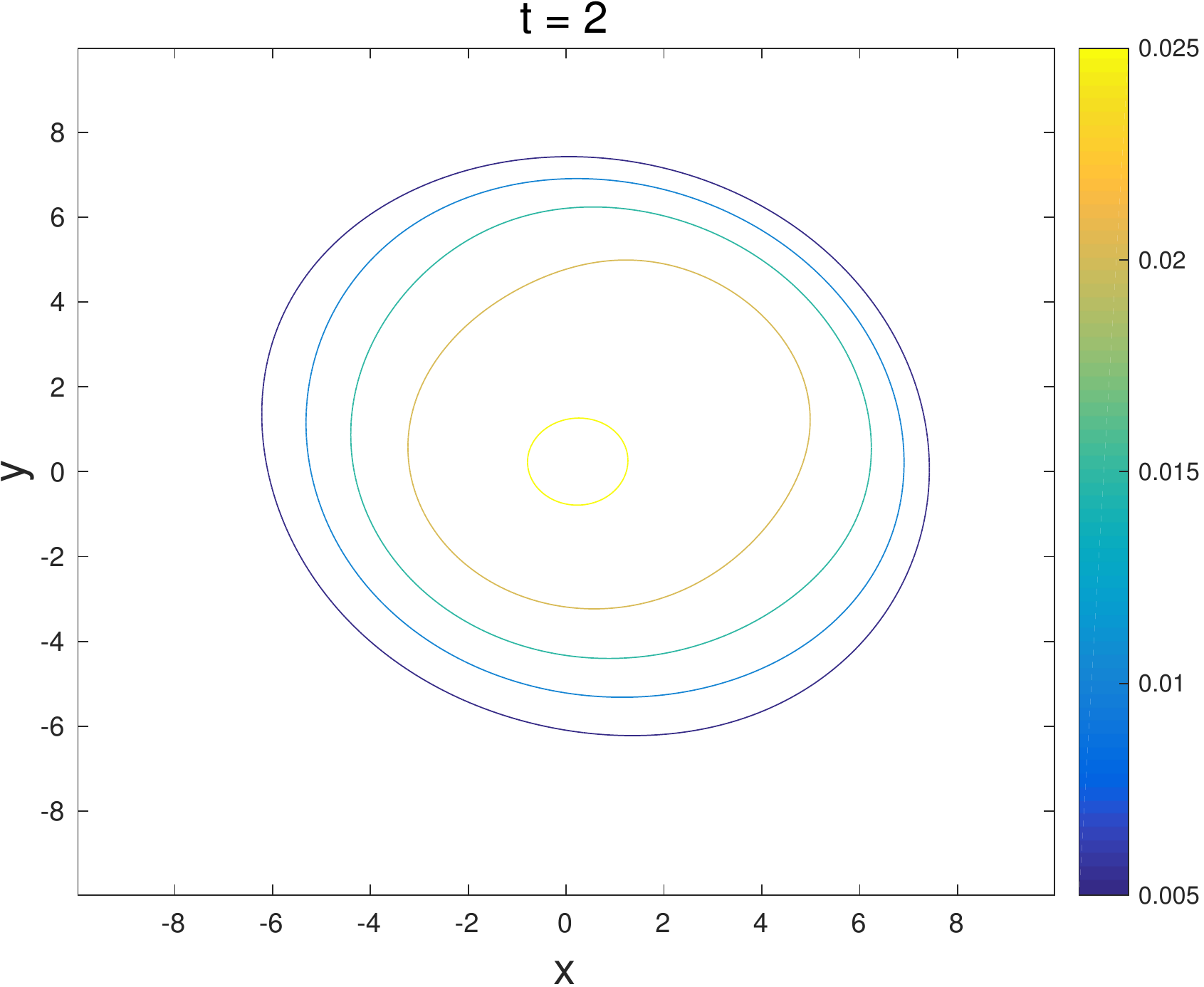}
		\end{minipage}
	}
	\subfigure{
		\begin{minipage}[t]{0.35\linewidth}
			\centering
			\includegraphics[width=1.0\linewidth]{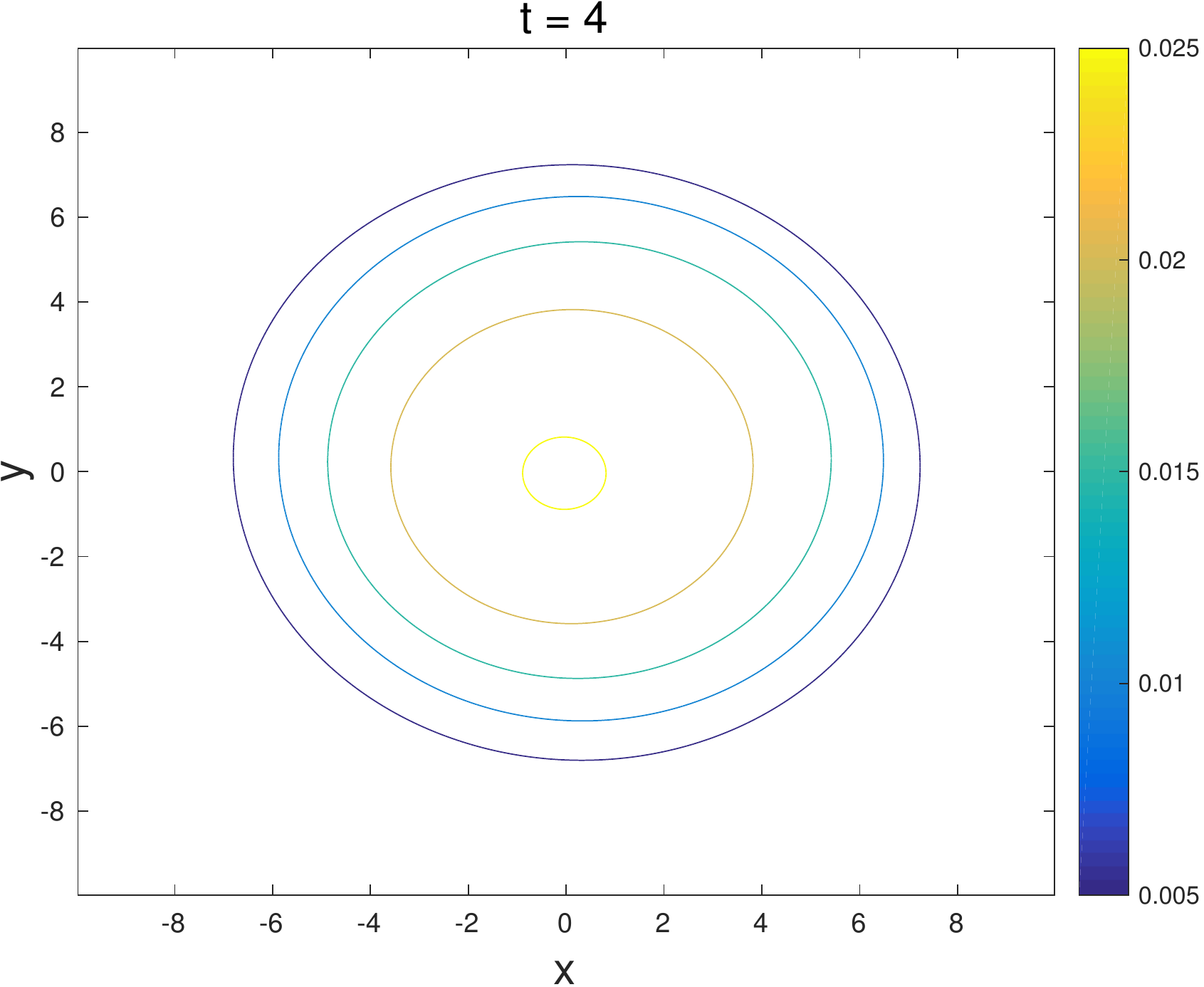}
		\end{minipage}
		\begin{minipage}[t]{0.35\linewidth}
			\centering
			\includegraphics[width=1.0\linewidth]{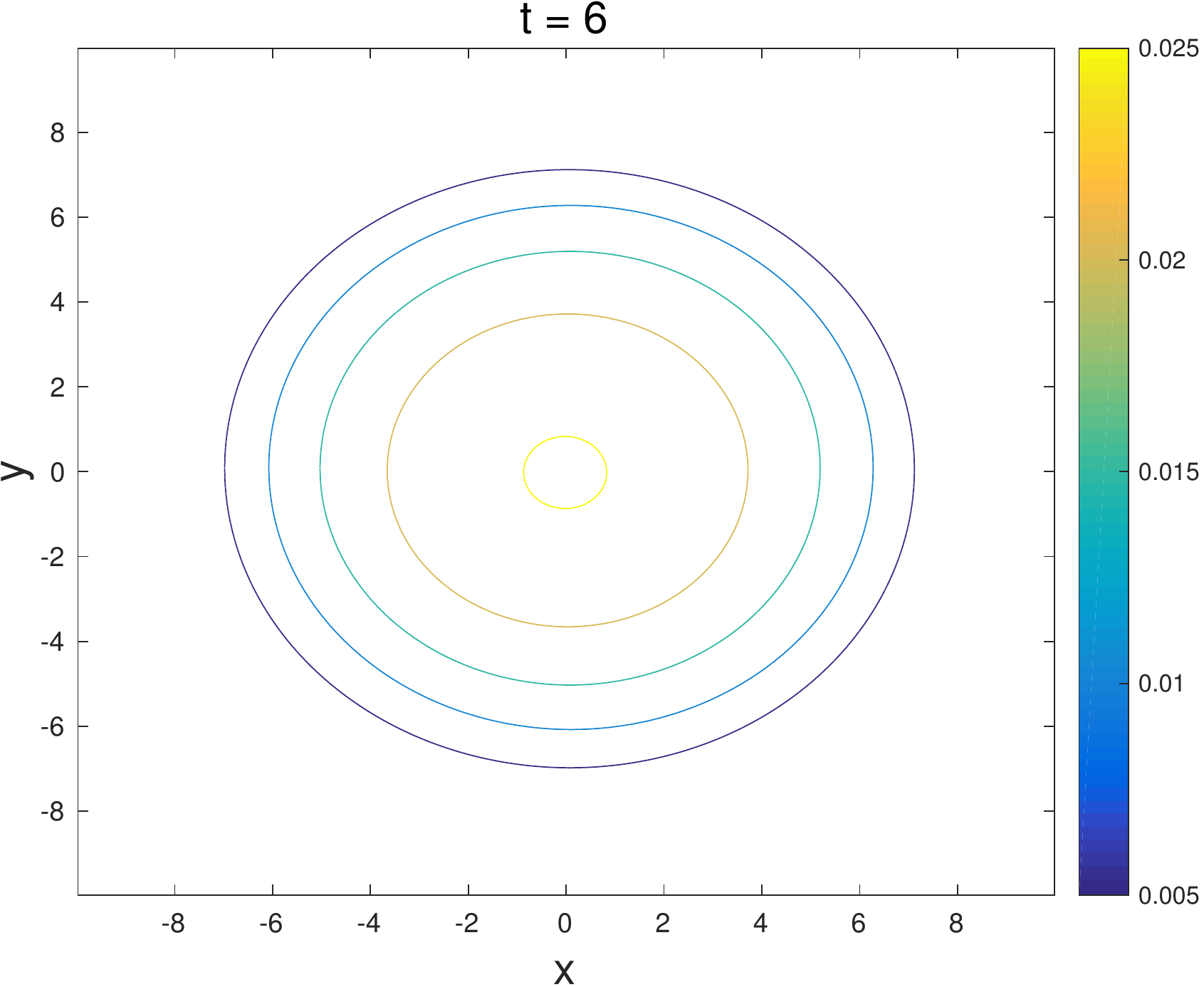}
		\end{minipage}
	}
	\caption{Multiple Species in Two-dimension: The space-concentration $c_1$ curves with both the mesh size $\Delta x$ and $\Delta y$ being 0.0390625 and the time $t = 0, 2, 4, 6$}
	\label{c1oft}
\end{figure}

\begin{figure}[htp] 
	\centering
	\subfigure{
		\begin{minipage}[t]{0.35\linewidth}
			\centering
			\includegraphics[width=1.0\linewidth]{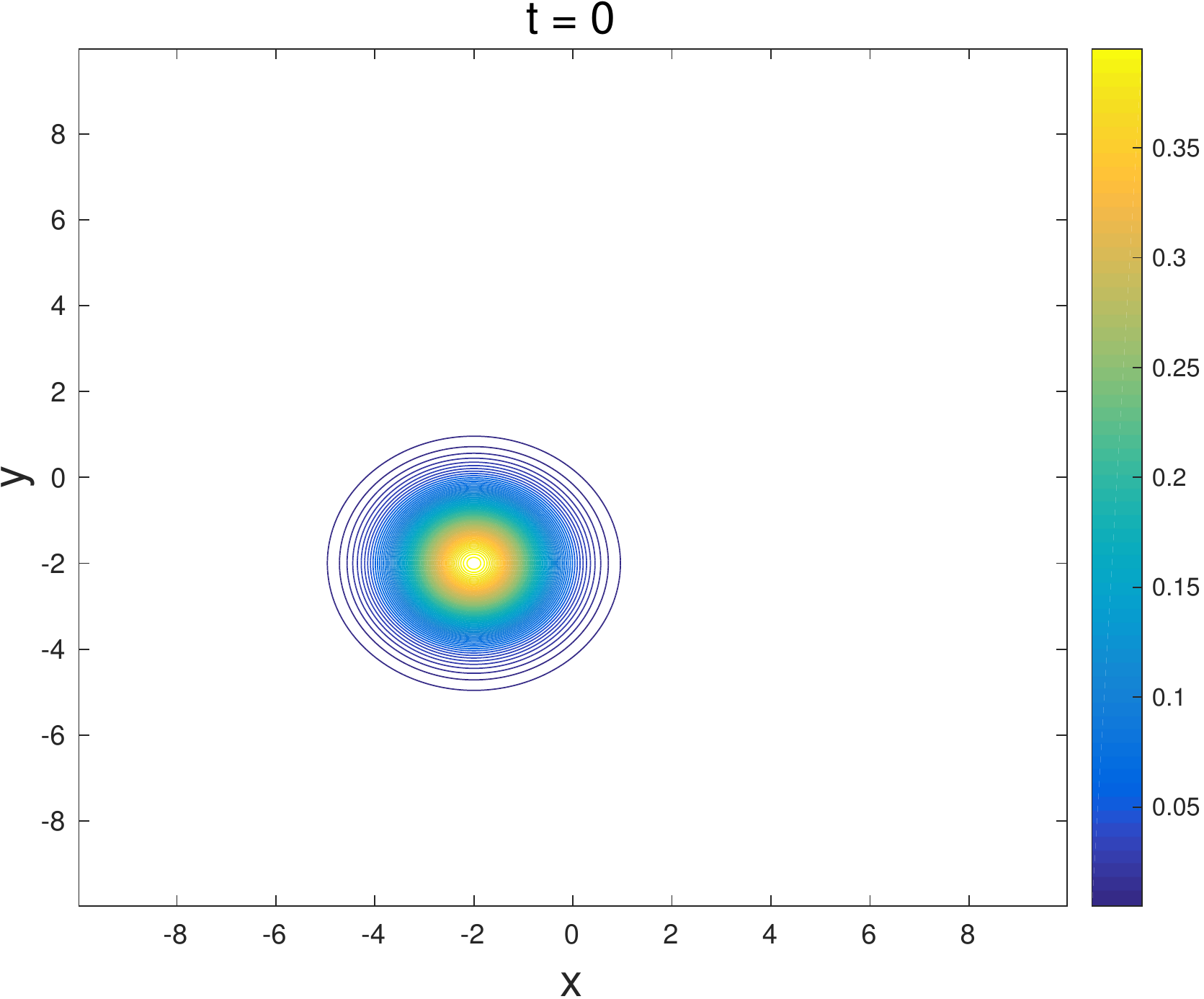}
		\end{minipage}
		\begin{minipage}[t]{0.35\linewidth}
			\centering
			\includegraphics[width=1.0\linewidth]{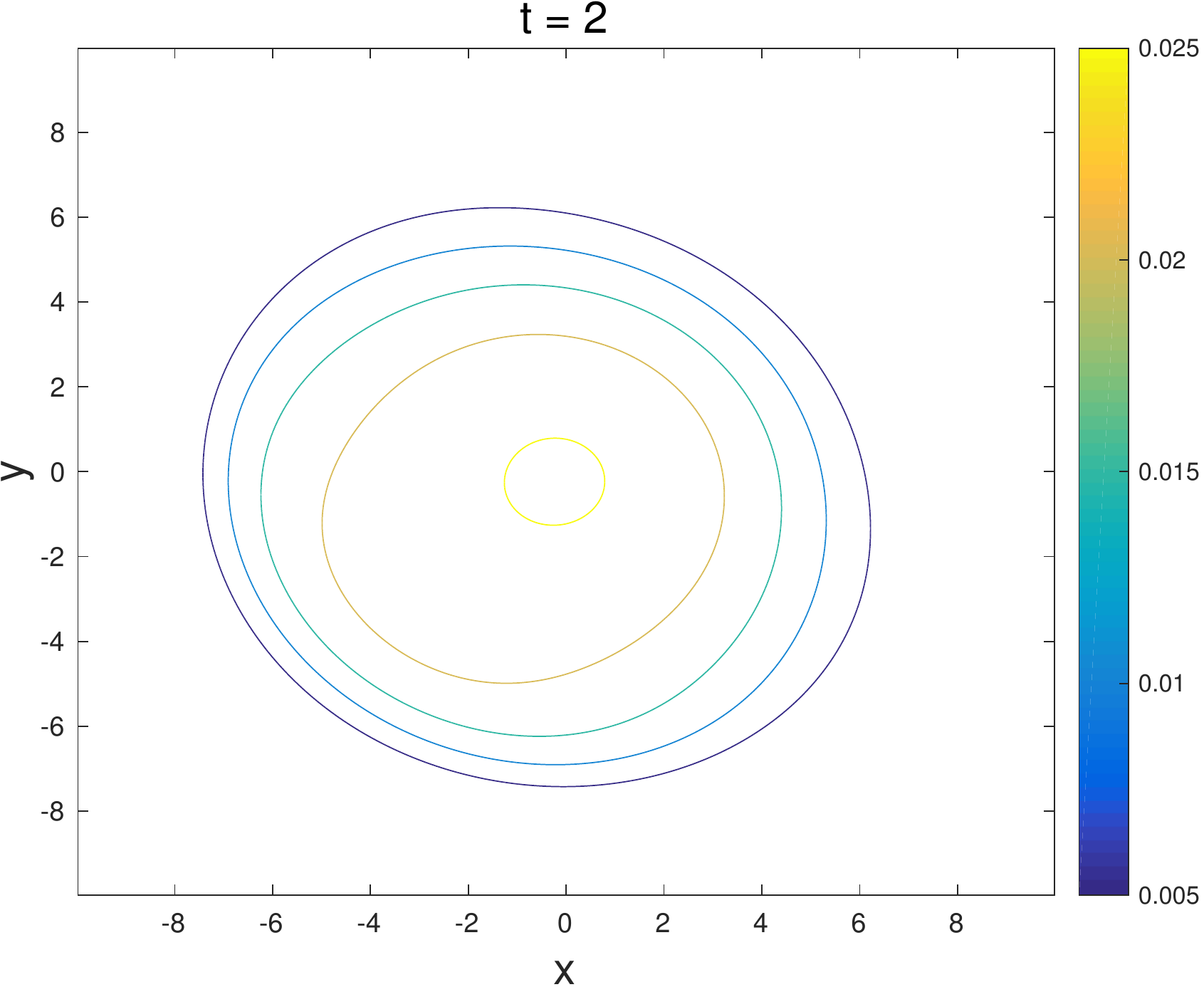}
		\end{minipage}
	}
	\subfigure{
		\begin{minipage}[t]{0.35\linewidth}
			\centering
			\includegraphics[width=1.0\linewidth]{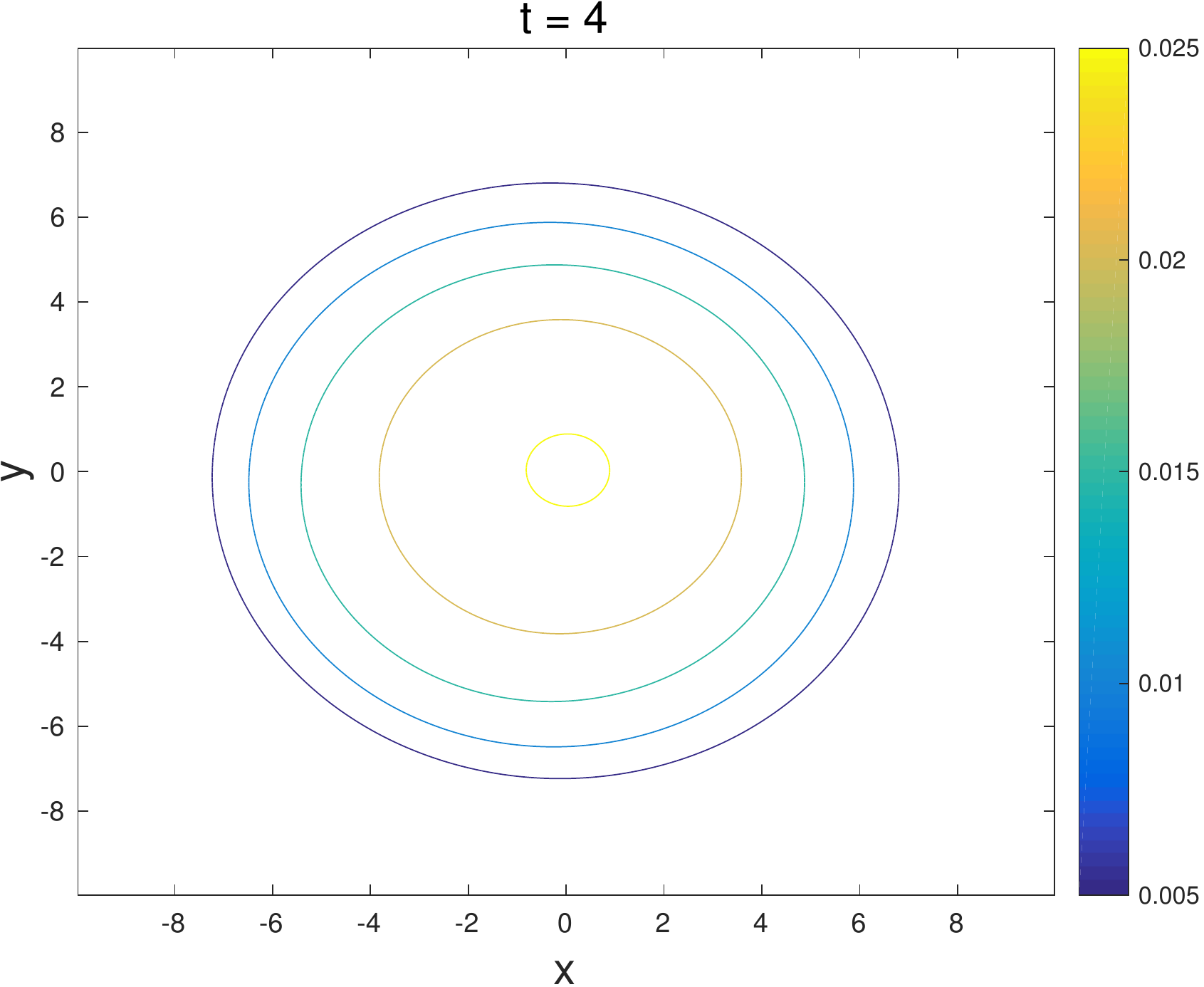}
		\end{minipage}
		\begin{minipage}[t]{0.35\linewidth}
			\centering
			\includegraphics[width=1.0\linewidth]{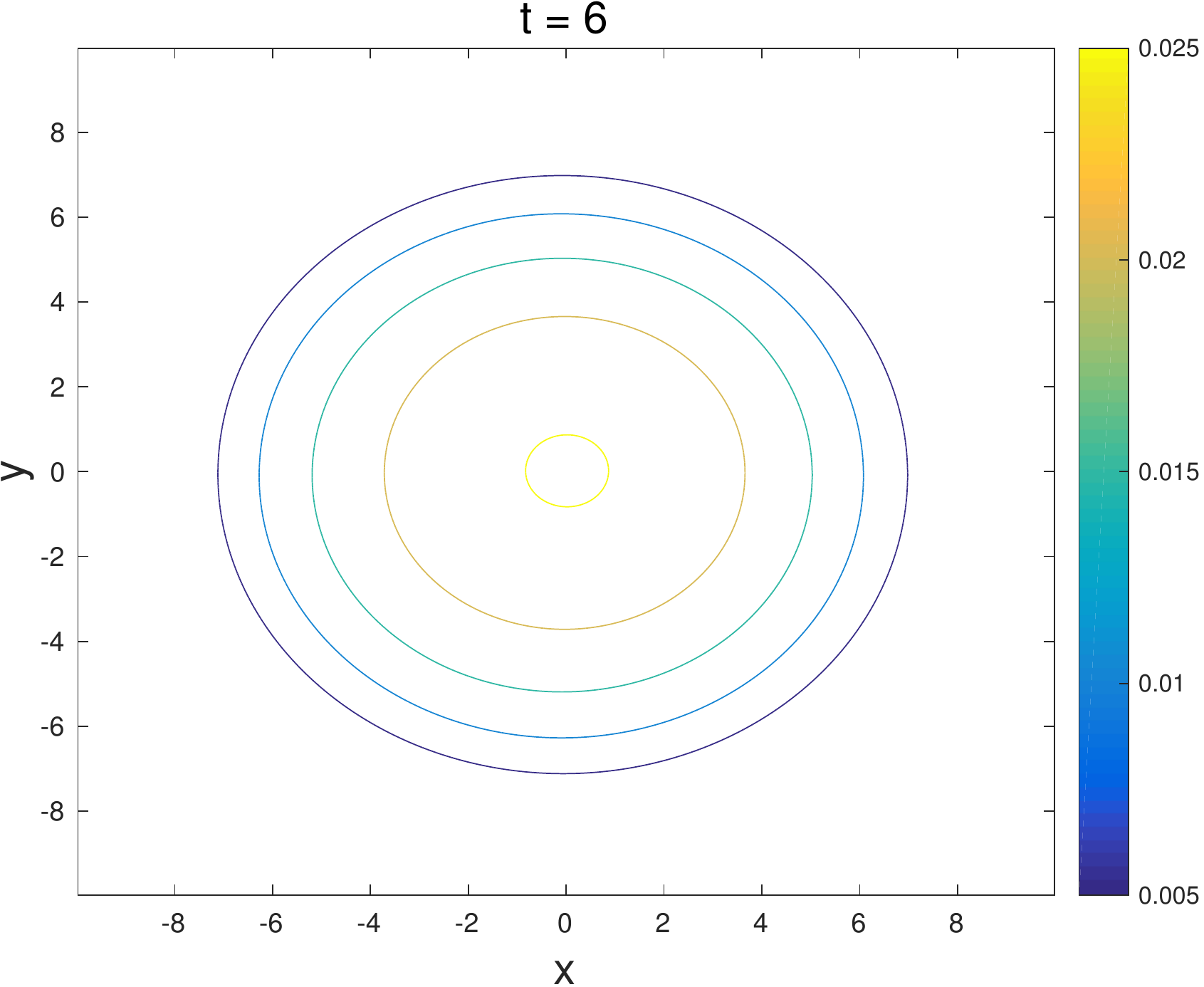}
		\end{minipage}
	}
	\caption{Multiple Species in Two-dimension: The space-concentration $c_2$ curves with both the mesh size $\Delta x$ and $\Delta y$ being 0.0390625 and the time $t = 0, 2, 4, 6$}
	\label{c2oft}
\end{figure}

\begin{figure}[htp] 
	\centering
	\subfigure{
		\includegraphics[width=0.6\linewidth]{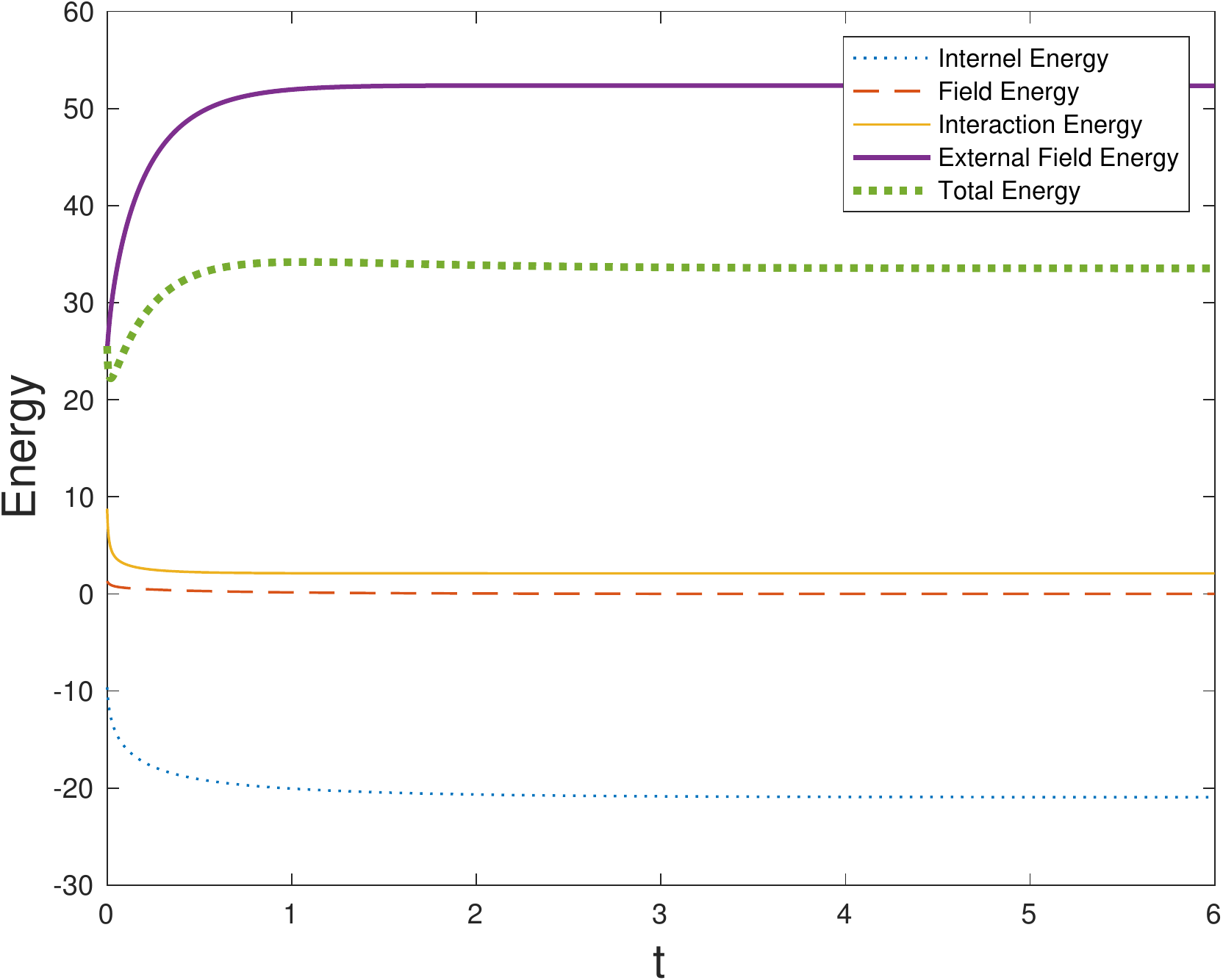}  
	}
	\caption{Multiple Species in Two-dimension: The time-energy plot of the  model (\ref{model3}) equipped with the initial conditions (\ref{iniex3}) with both the mesh size $\Delta x$ and $\Delta y$ being 0.0390625.}  
	\label{eoft2d} 
\end{figure}

\subsubsection{Finite Size Effect} 
The strength of the steric repulsion arising from the finite size is indicated by the parameter $\eta$ in the kernel $\mathcal{W}(x)$.
Let $\eta = 1, \dfrac{1}{4}, \cdots, \dfrac{1}{128}, 0$ and the mesh size $\Delta x = \Delta y = 0.1562$, 
Figure \ref{c1ofeta} shows different steady state solutions with different values of $\eta$, where we can find that the finite size effect makes the concentrations $c_m, \ m = 1, 2,$ not overly peaked. 

\begin{figure}[htp] 
	\subfigure{
		\begin{minipage}[t]{0.33\linewidth}
			\centering
			\includegraphics[width=1.0\linewidth]{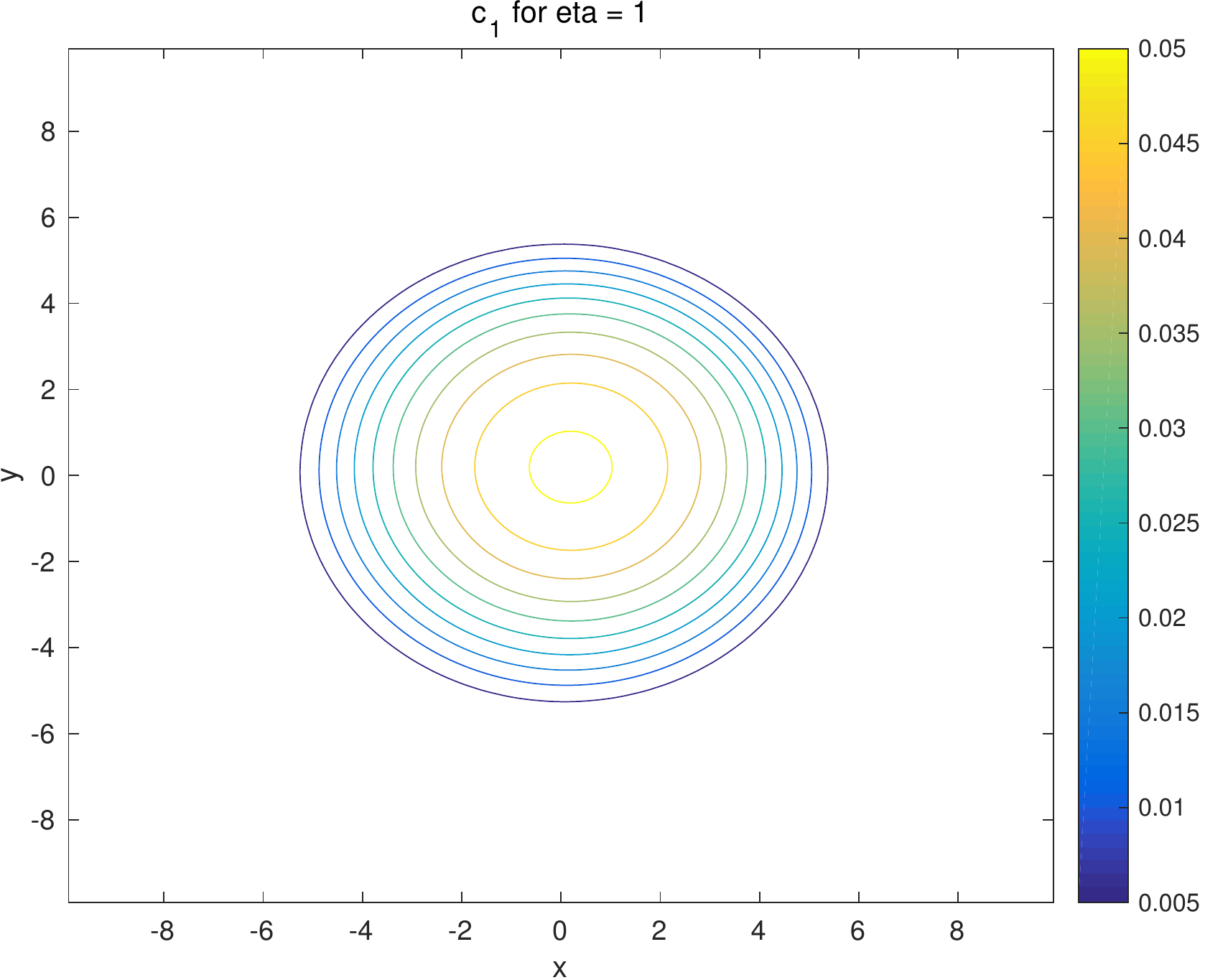}
		\end{minipage}
		\begin{minipage}[t]{0.33\linewidth}
			\centering
			\includegraphics[width=1.0\linewidth]{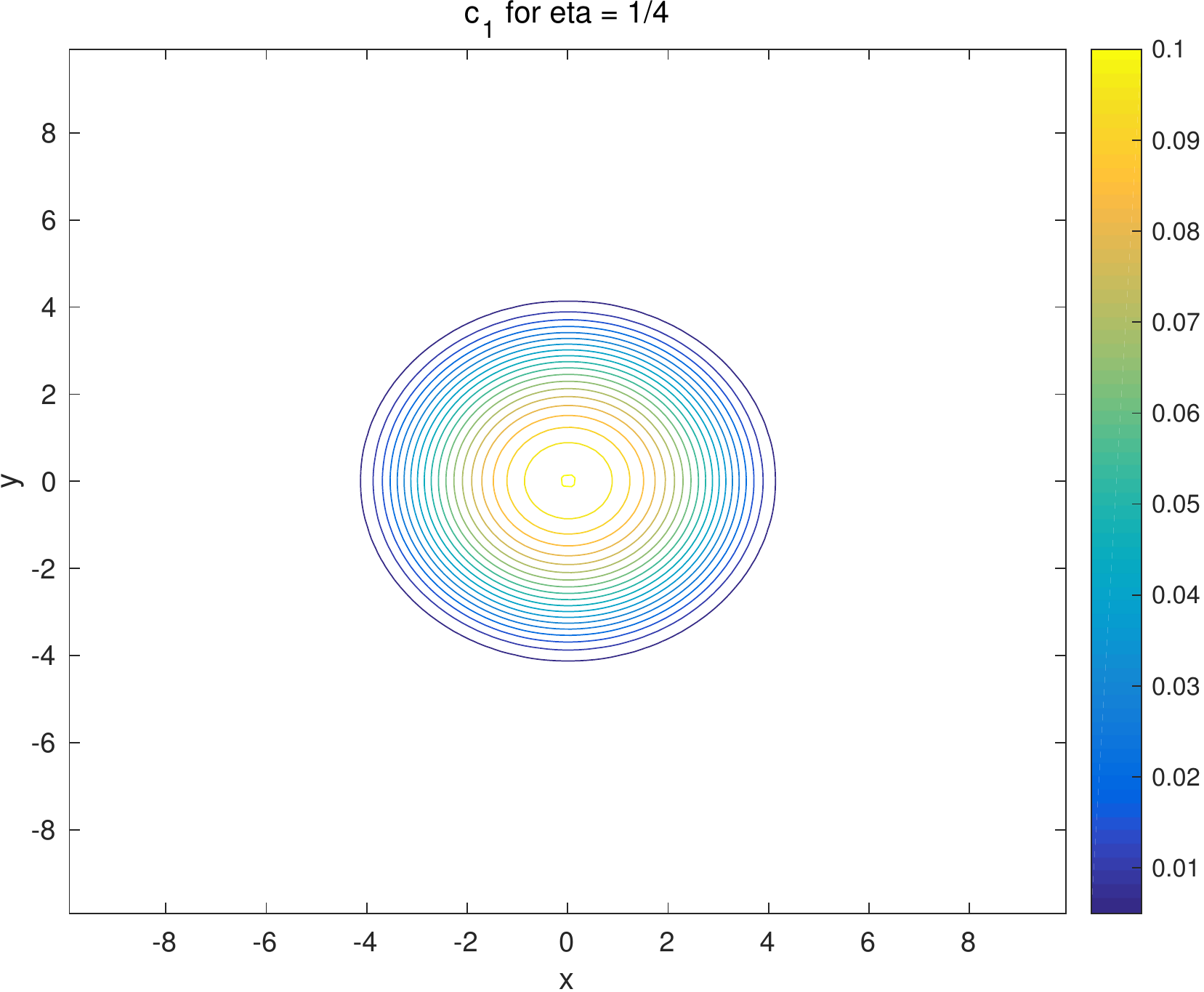}
		\end{minipage}
		\begin{minipage}[t]{0.33\linewidth}
			\centering
			\includegraphics[width=1.0\linewidth]{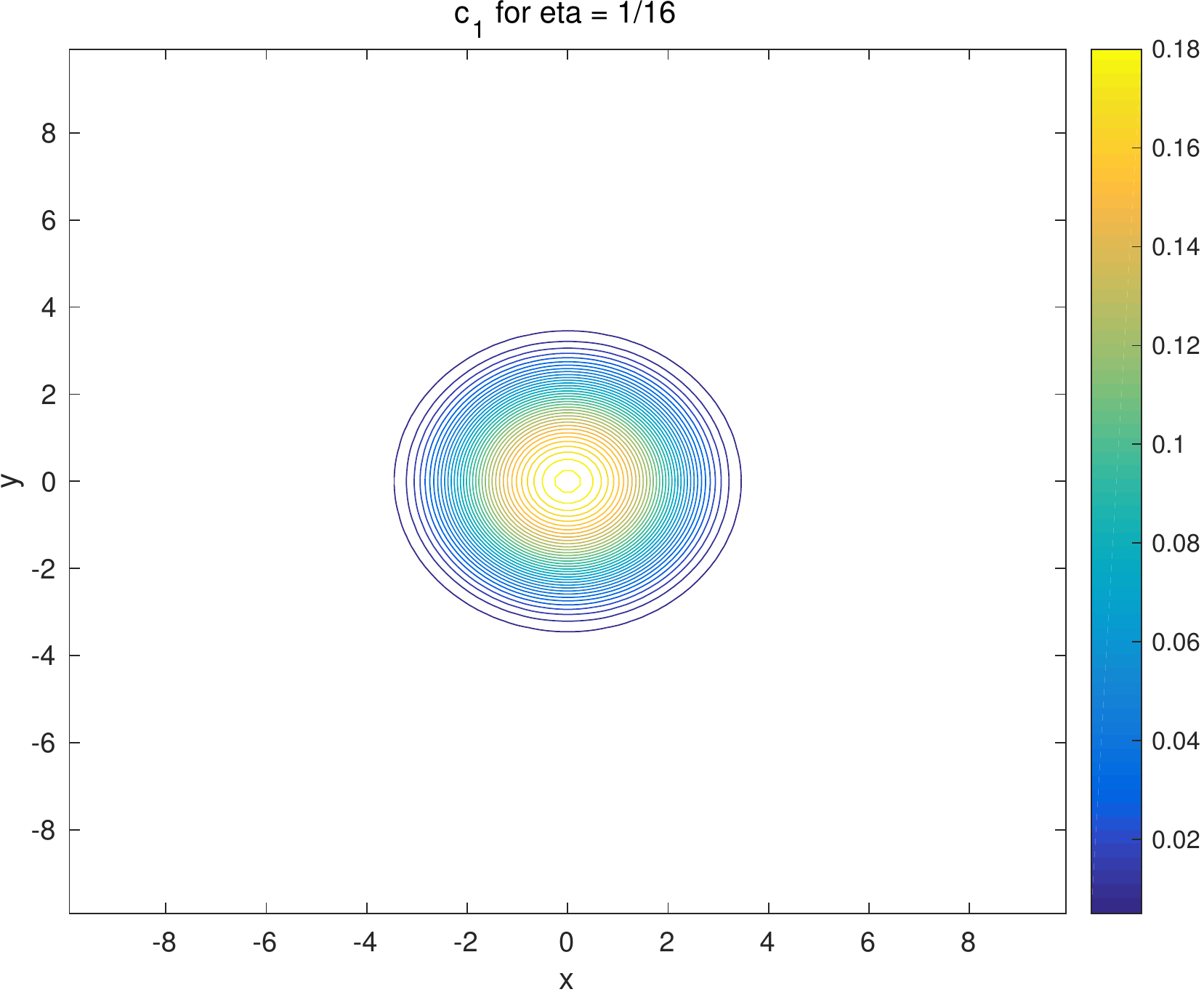}
		\end{minipage}
	}
	\subfigure{
		\begin{minipage}[t]{0.33\linewidth}
			\centering
			\includegraphics[width=1.0\linewidth]{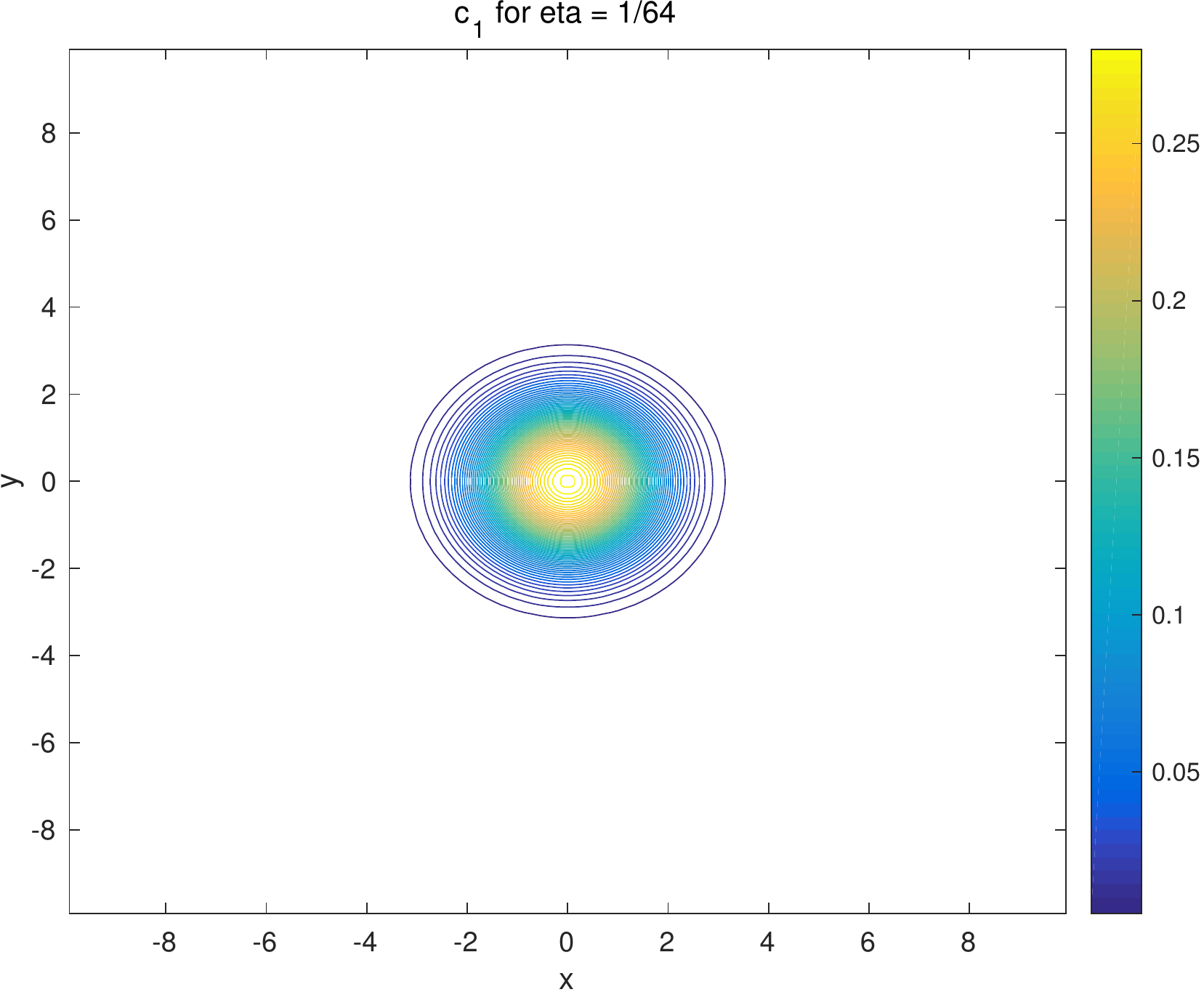}
		\end{minipage}
		\begin{minipage}[t]{0.33\linewidth}
			\centering
			\includegraphics[width=1.0\linewidth]{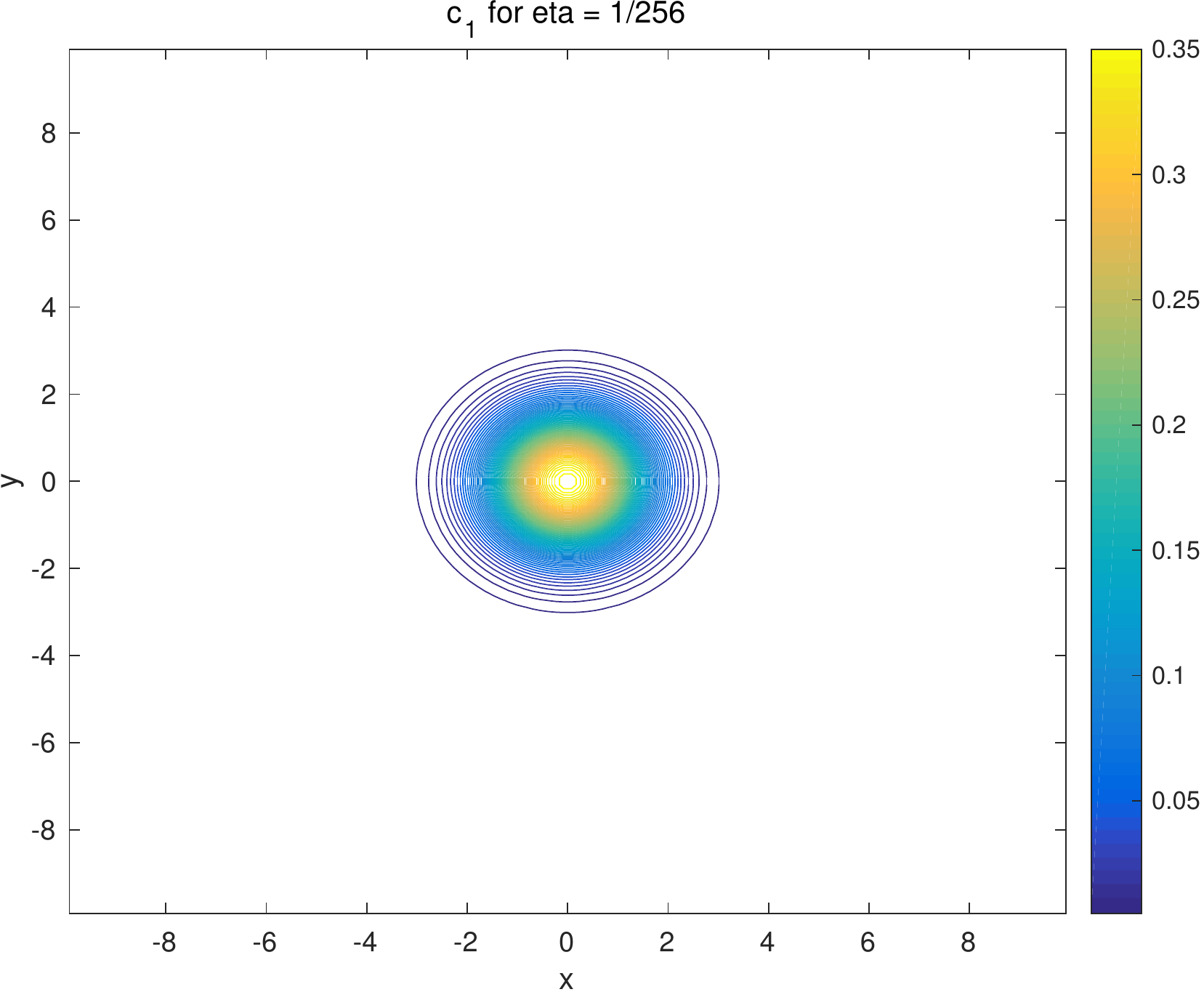}
		\end{minipage}
		\begin{minipage}[t]{0.33\linewidth}
			\centering
			\includegraphics[width=1.0\linewidth]{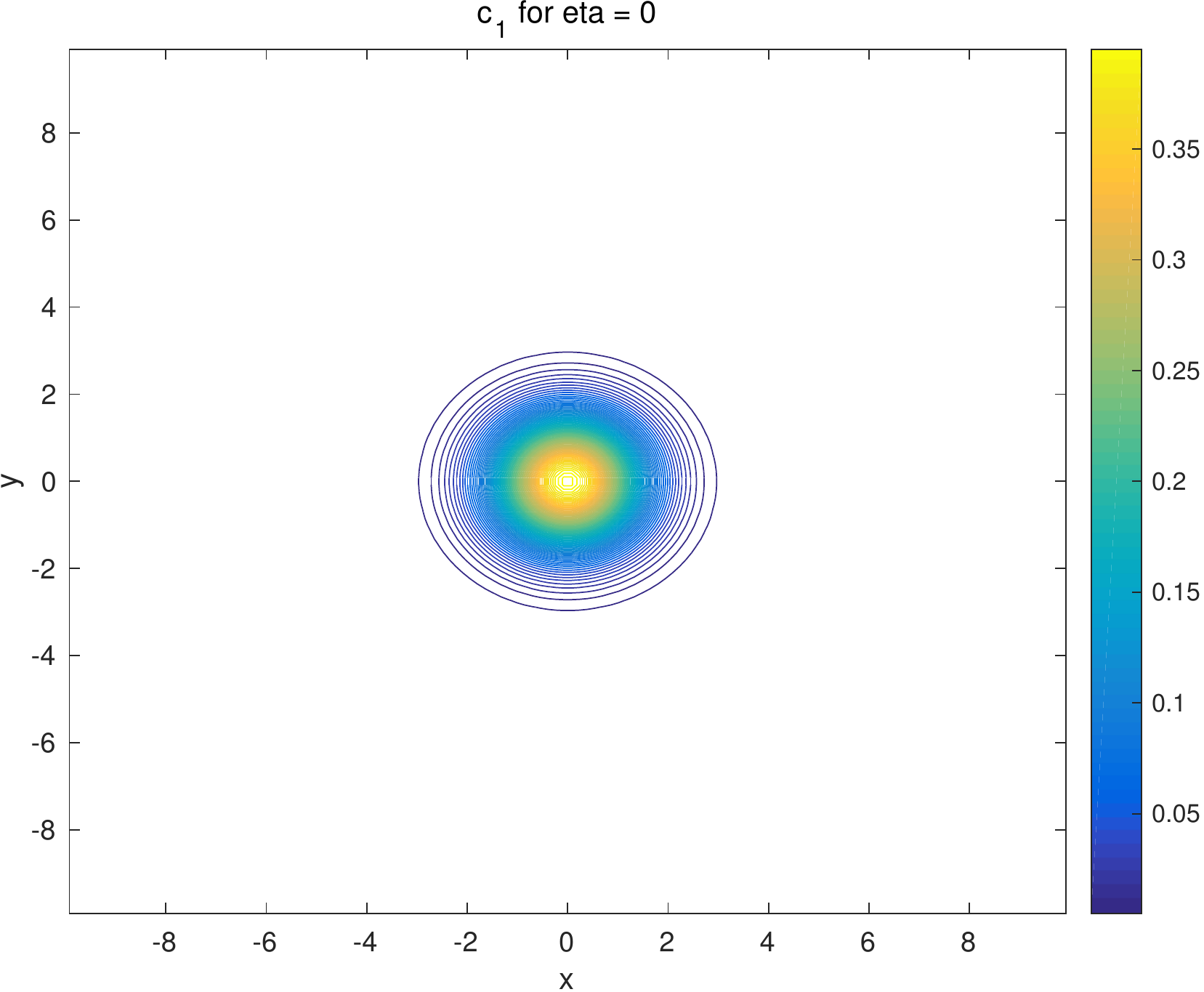}
		\end{minipage}
	}
	\caption{Multiple Species in Two-dimension: The steady state density solutions $c_1$ with different $\eta$}
	\label{c1ofeta}
\end{figure}

\subsection{Example 3} 
\subsubsection{Steady State}
Consider the equations (\ref{model3}) in one-dimension with the kernel $\mathcal{W}(x) = \frac{\eta}{x^2 + \epsilon^2}, \epsilon = \frac{1}{10}$, $\mathcal{K}(x) = \exp(-|x|)$, $V_{\text{ext}}(x) = \frac{1}{2} x^2$ and the initial conditions are given by
\begin{equation*}
\left\{
\begin{array}{lll}
c_1(x, 0) = \dfrac{1}{\sqrt{2 \pi}} \exp \left(-\dfrac{(x - 2)^2}{2} \right) &\text{with} &z_1 = 1, \\
c_2(x, 0) = \dfrac{1}{\sqrt{2 \pi}} \exp \left(-\dfrac{(x + 2)^2}{2} \right) &\text{with} &z_2 = -1.
\end{array}
\right.
\end{equation*}
In addition, we add a constant electric field whose field intensity is 2 to the solutions to observe the behavior of the ionic species, i.e. the field system (\ref{model3})-(\ref{model4}) becomes
\begin{equation}
\begin{aligned}
\partial_t c_m (x, t) &= \nabla \cdot \left[ c_m \nabla \left(1 + \log c_m + z_m \mathcal{K}*\rho + \mathcal{W}*\theta + V_{\text{ext}} + z_m V_{\text{electric field}} \right) \right], 
\\
c_m(x, 0) &= c^0_{m}(x), ~~m = 1, \cdots, M,
\end{aligned}
\end{equation}
where $V_{\text{electric field}} = 2 x$.

Here, take $\eta = 1$, the computation domain as $[-2L, 2L], \ L = 10$, where we are concerned with the results on domain $[-L, L]$ and the mesh size $\Delta x = 0.0390625\ (N = 2^{10})$, Figure \ref{ex3c} shows how the concentrations $c_m, m = 1, 2$, change with time $t$. For positive electric charges, the velocity concerned with the constant electric field $v_0 = -\partial_{x} (z_m V_{\text{electric field}}) = -2 z_m< 0$, which means the positively charged ions are driven towards the left boundary while the negative electric charges are driven towards the right boundary.

\begin{figure}[htp]
	\subfigure{
		\begin{minipage}[t]{0.33\linewidth}
			\centering
			\includegraphics[width=1.0\linewidth]{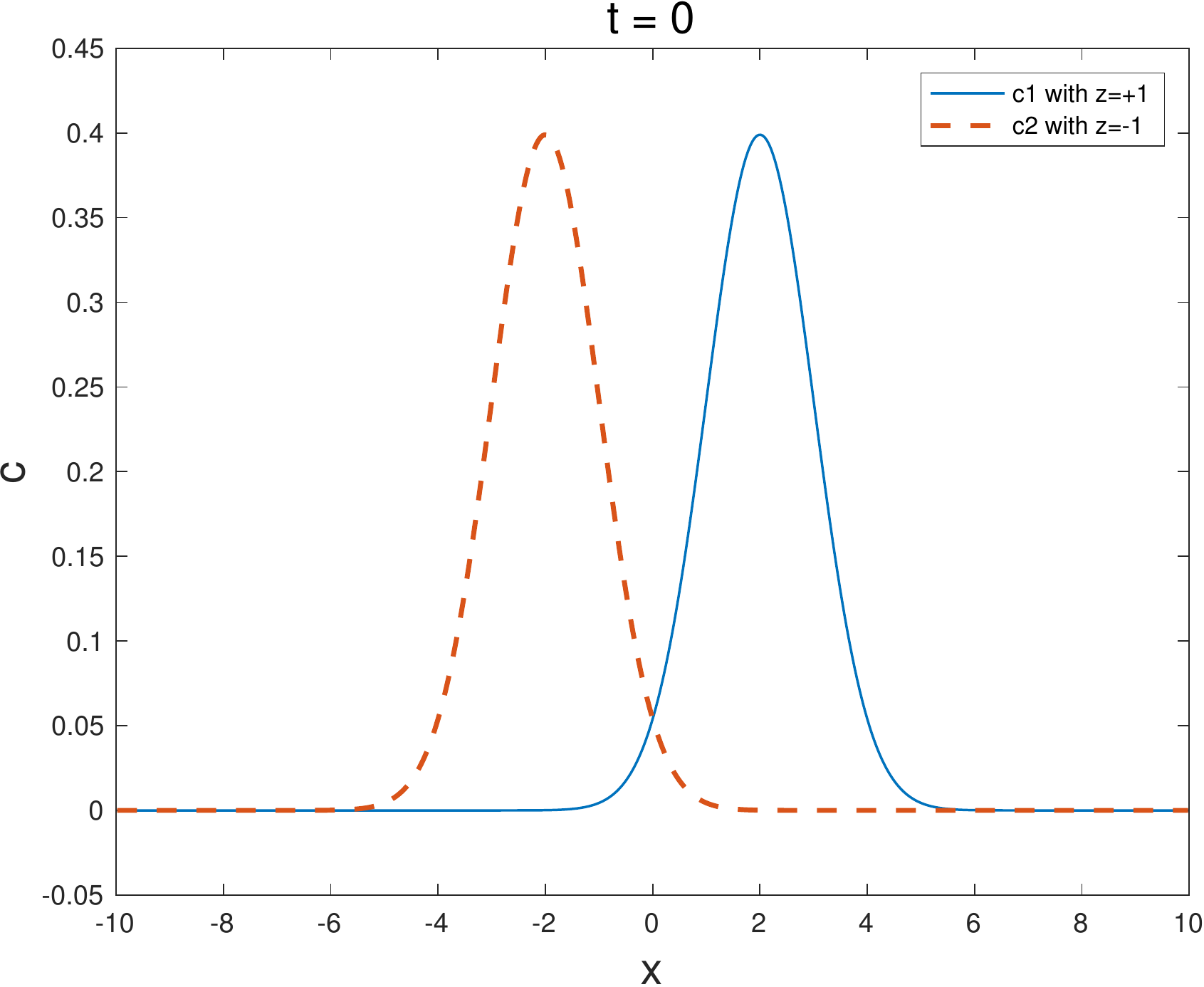}
		\end{minipage}
		\begin{minipage}[t]{0.33\linewidth}
			\centering
			\includegraphics[width=1.0\linewidth]{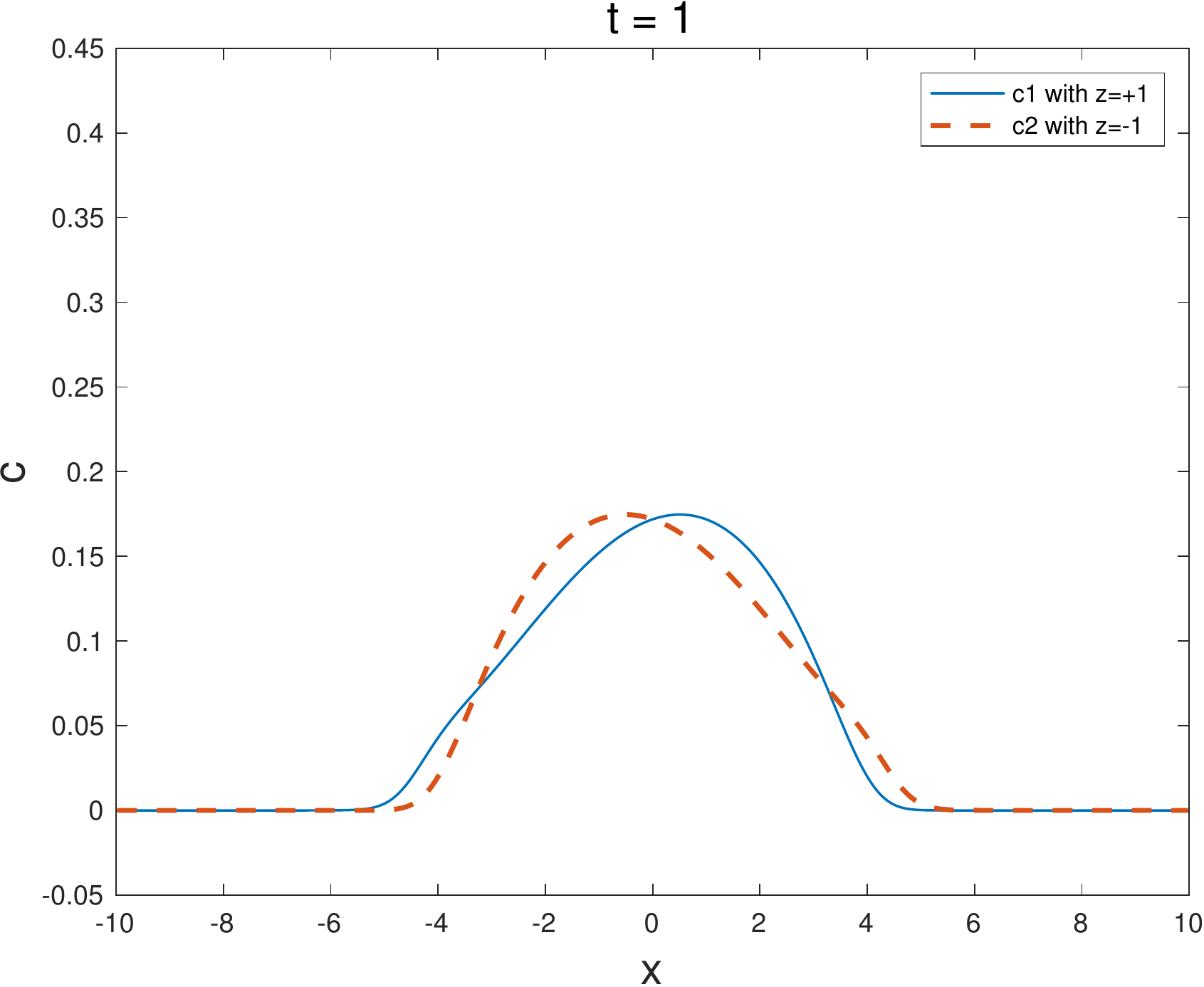}
		\end{minipage}
		\begin{minipage}[t]{0.33\linewidth}
			\centering
			\includegraphics[width=1.0\linewidth]{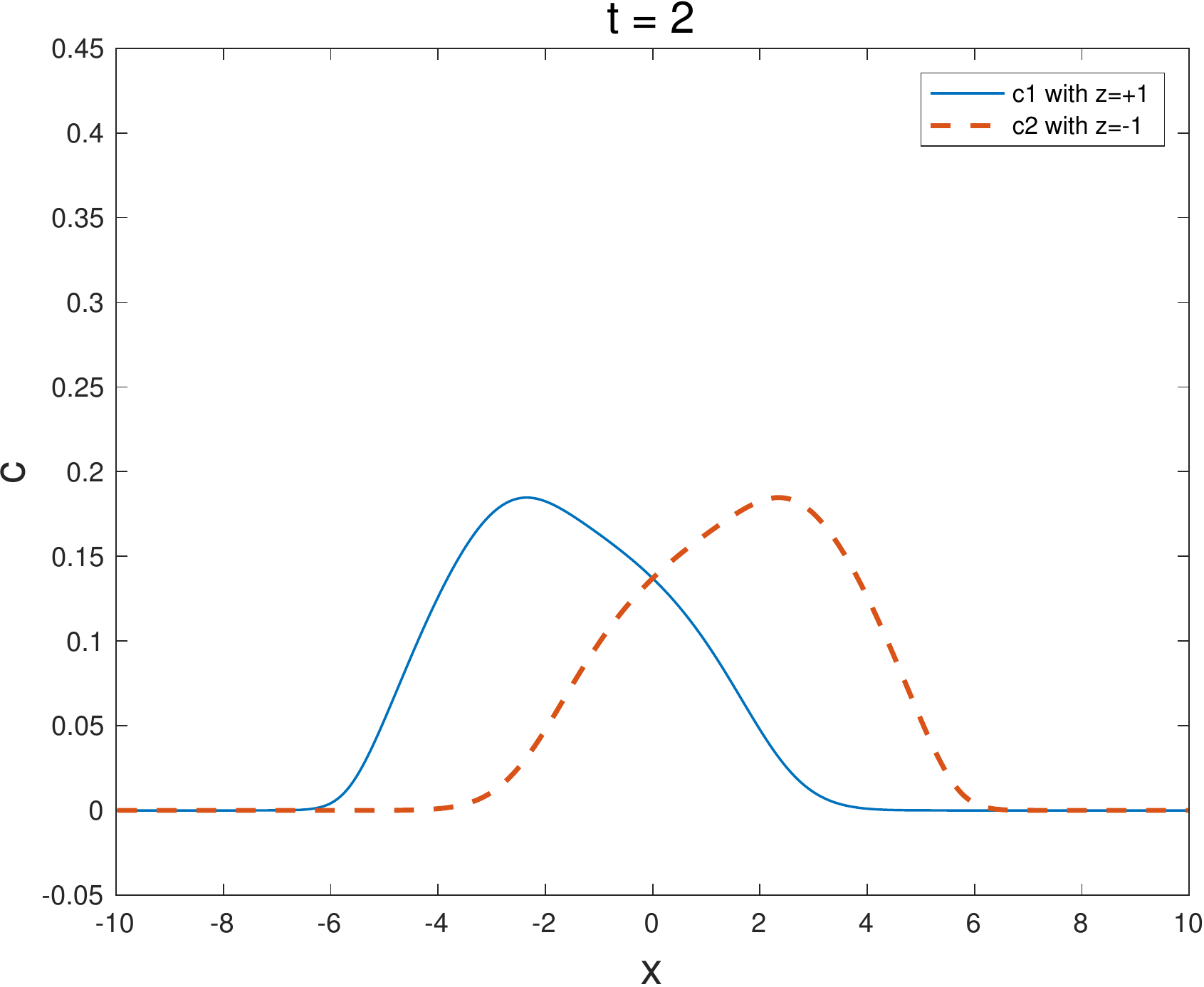}
		\end{minipage}
	}
	\subfigure{
		\begin{minipage}[t]{0.33\linewidth}
			\centering
			\includegraphics[width=1.0\linewidth]{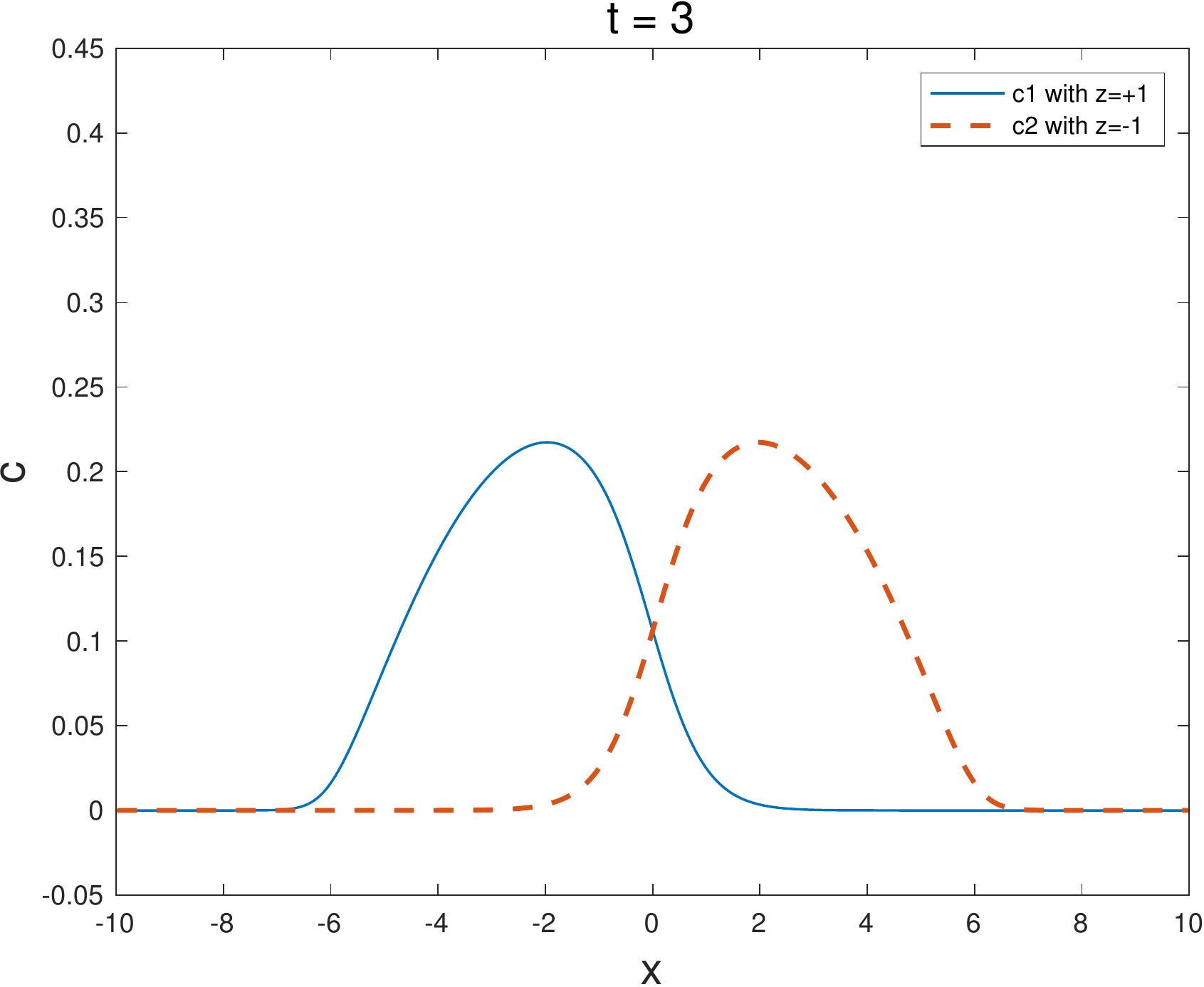}
		\end{minipage}
		\begin{minipage}[t]{0.33\linewidth}
			\centering
			\includegraphics[width=1.0\linewidth]{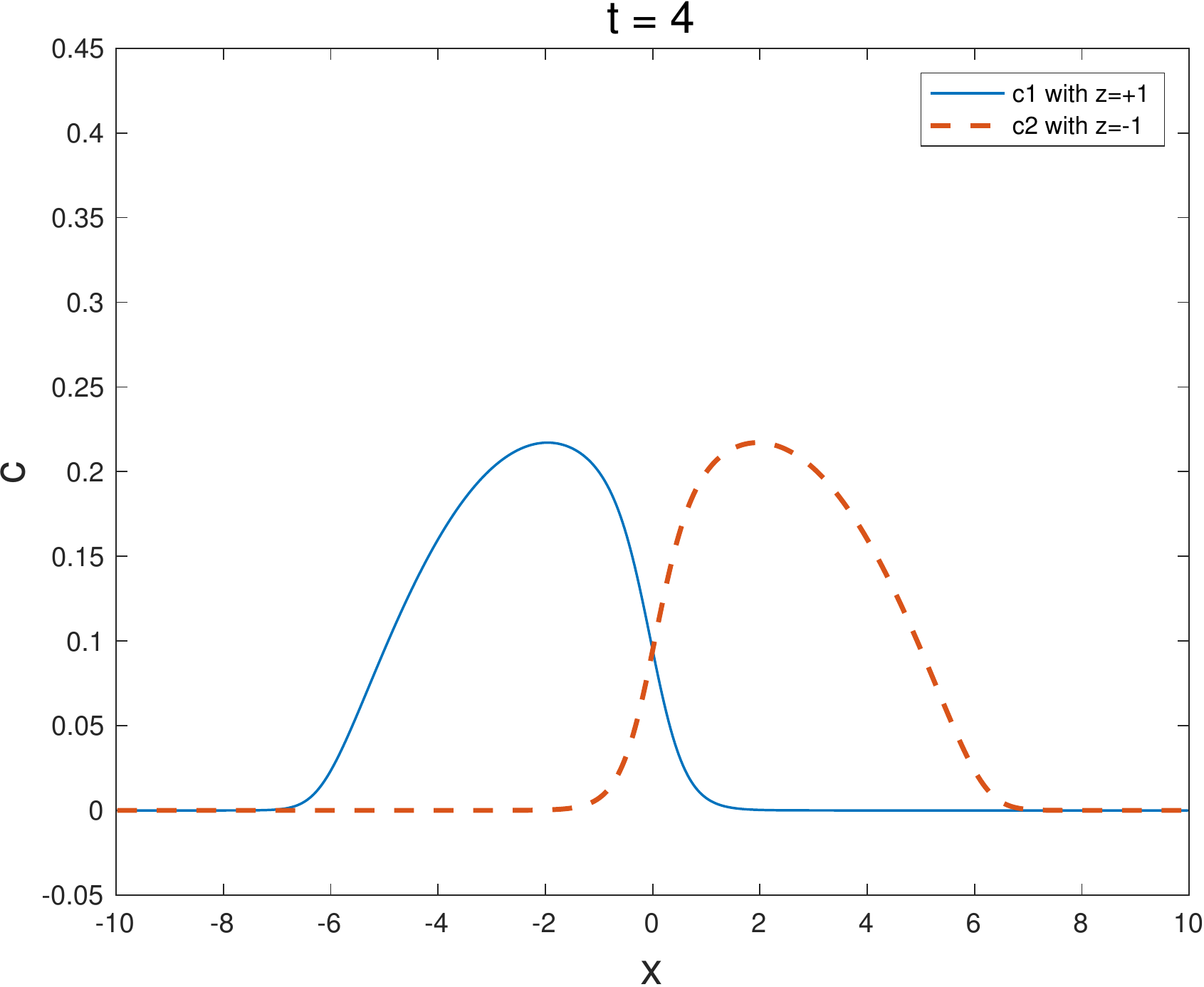}
		\end{minipage}
		\begin{minipage}[t]{0.33\linewidth}
			\centering
			\includegraphics[width=1.0\linewidth]{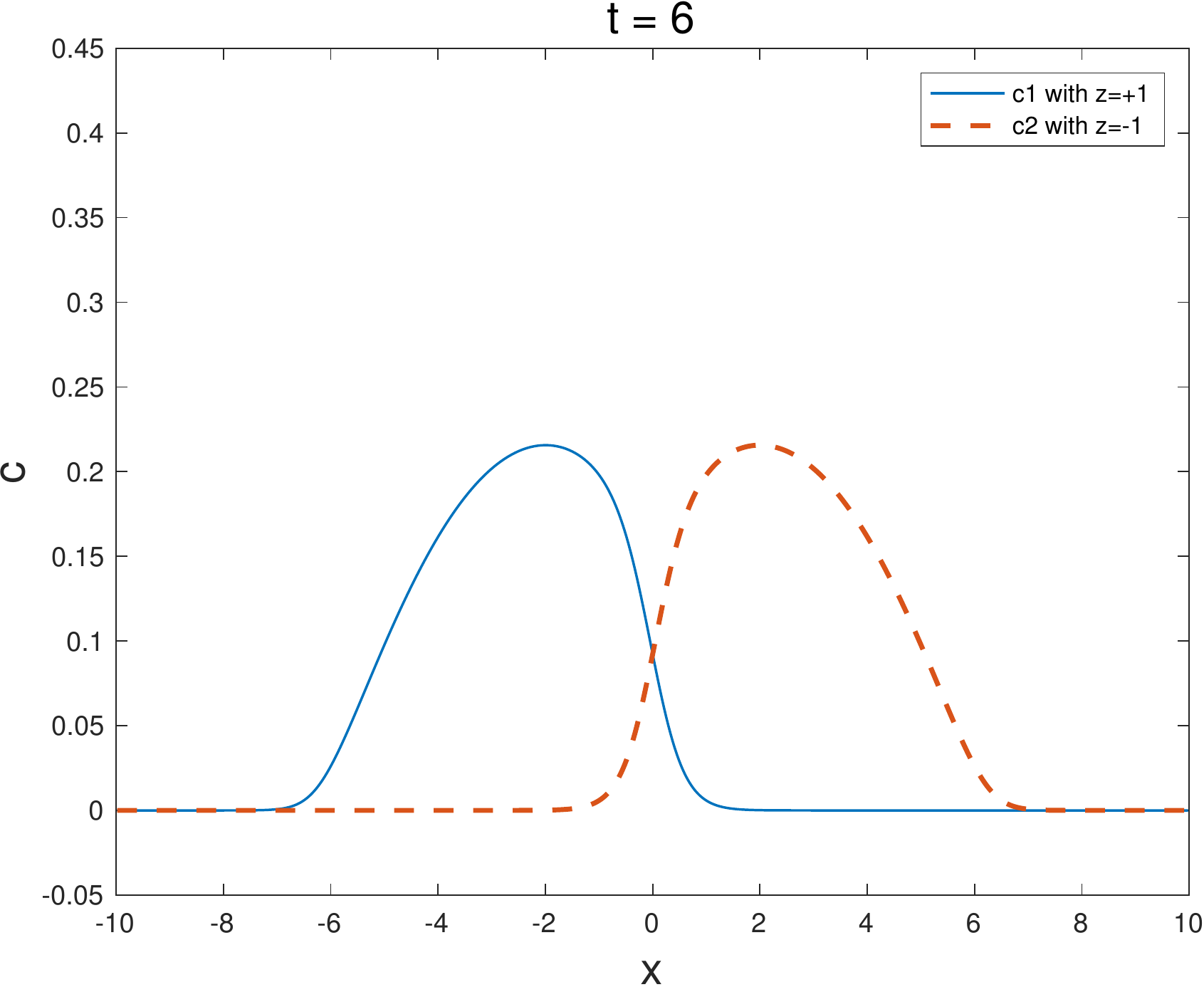}
		\end{minipage}
	}
	\subfigure{
		\begin{minipage}[t]{0.33\linewidth}
			\centering
			\includegraphics[width=1.0\linewidth]{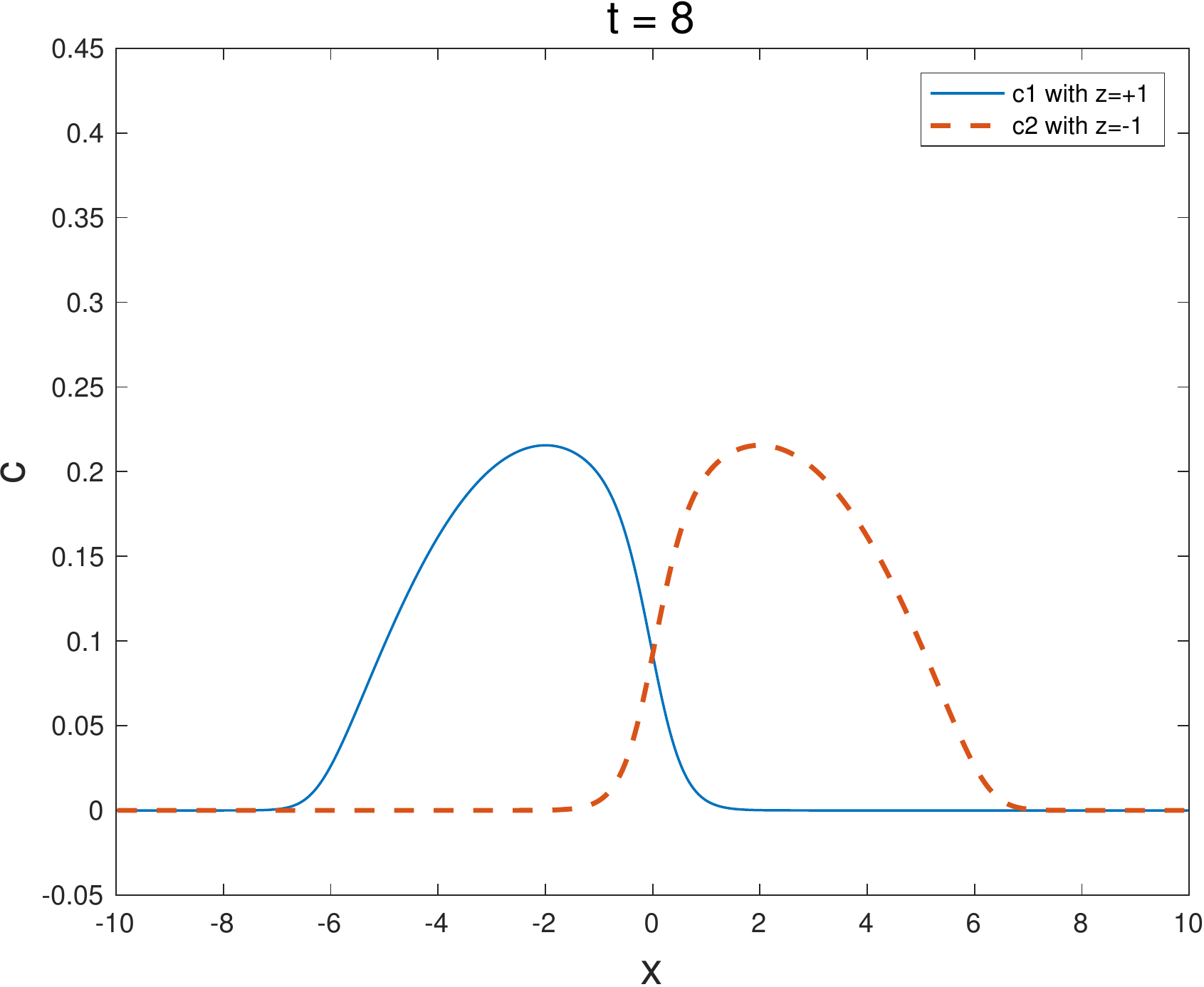}
		\end{minipage}
		\begin{minipage}[t]{0.33\linewidth}
			\centering
			\includegraphics[width=1.0\linewidth]{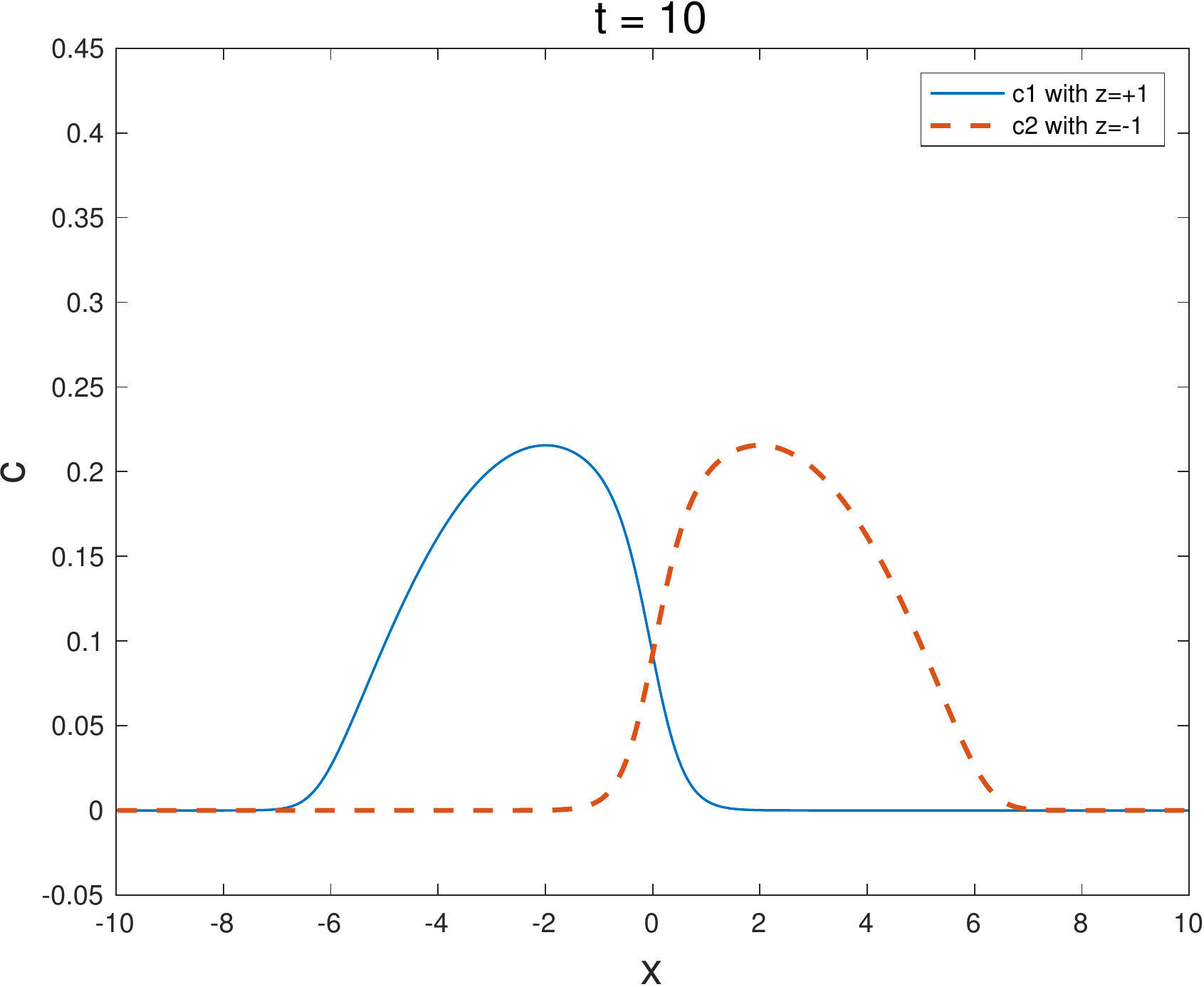}
		\end{minipage}
		\begin{minipage}[t]{0.33\linewidth}
			\centering
			\includegraphics[width=1.0\linewidth]{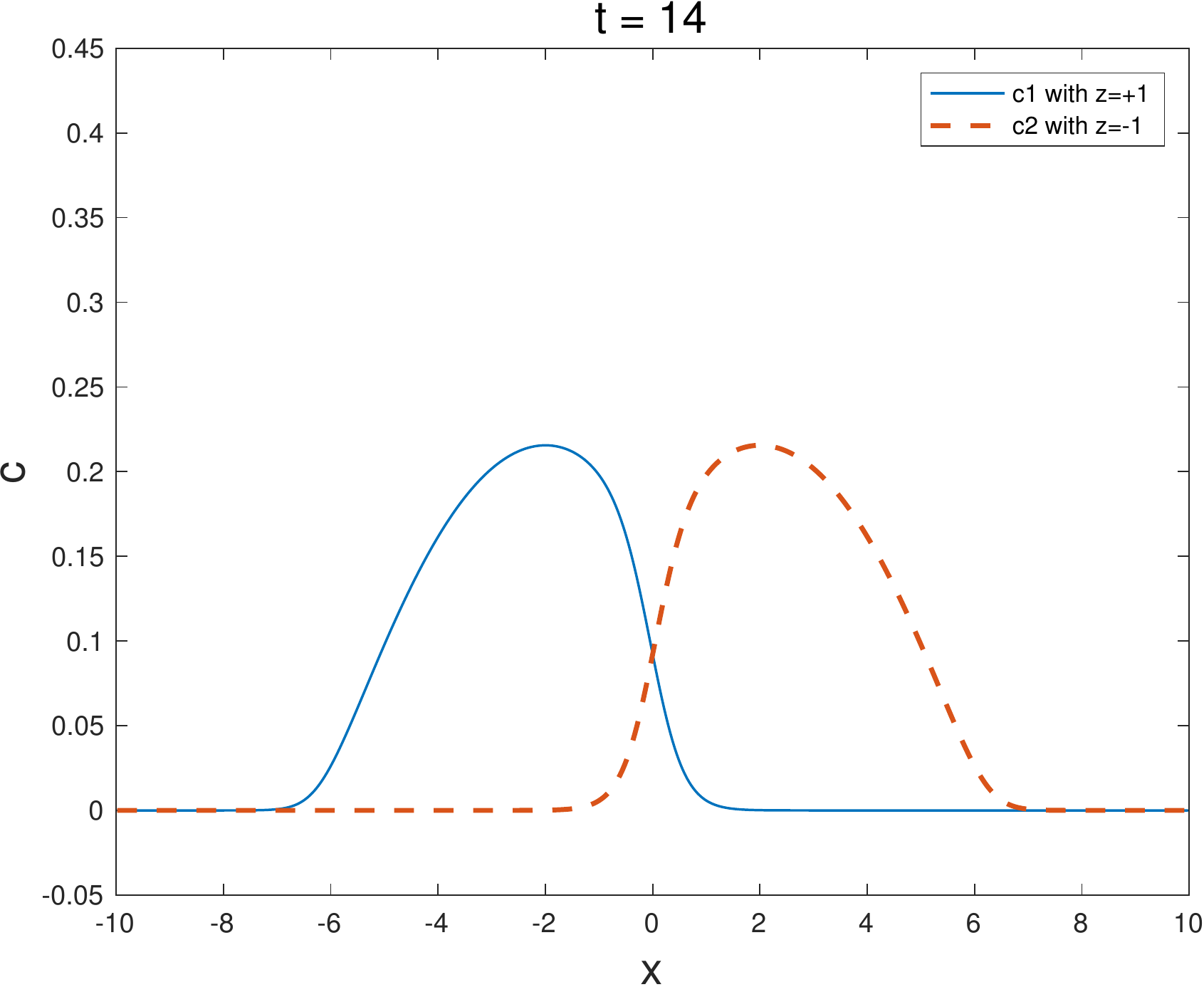}
		\end{minipage}
	}
	\caption{Example 3: The space-concentration $c_m$ curves with the mesh size $\Delta x$  being 0.0390625 and the time $t$ changing from 0 to 14}
	\label{ex3c}
\end{figure}

\subsubsection{Finite Size Effect} 
Next we aim to investigate this phenomenon numerically that such electric field can make positive and negative electric charges gather on different ends and how the nonlocal repulsion modifies the profile of the steady states. Let $\eta = \dfrac{1}{2}, \dfrac{1}{4}, \cdots, \dfrac{1}{128}, 0$ and the mesh size $\Delta x = 0.0390625$, Figure \ref{ex3eta} shows different steady state solutions with different $\eta$, here the density solution of time $t = 14$ approximates steady state solution. In conclusion, it's observed that the positive and negative particles move in different directions and the finite size effect makes the concentrations $c_m, m = 1, 2$, not overly peaked like before.

\begin{figure}[htp]
	\subfigure{
		\begin{minipage}[t]{0.33\linewidth}
			\centering
			\includegraphics[width=1.0\linewidth]{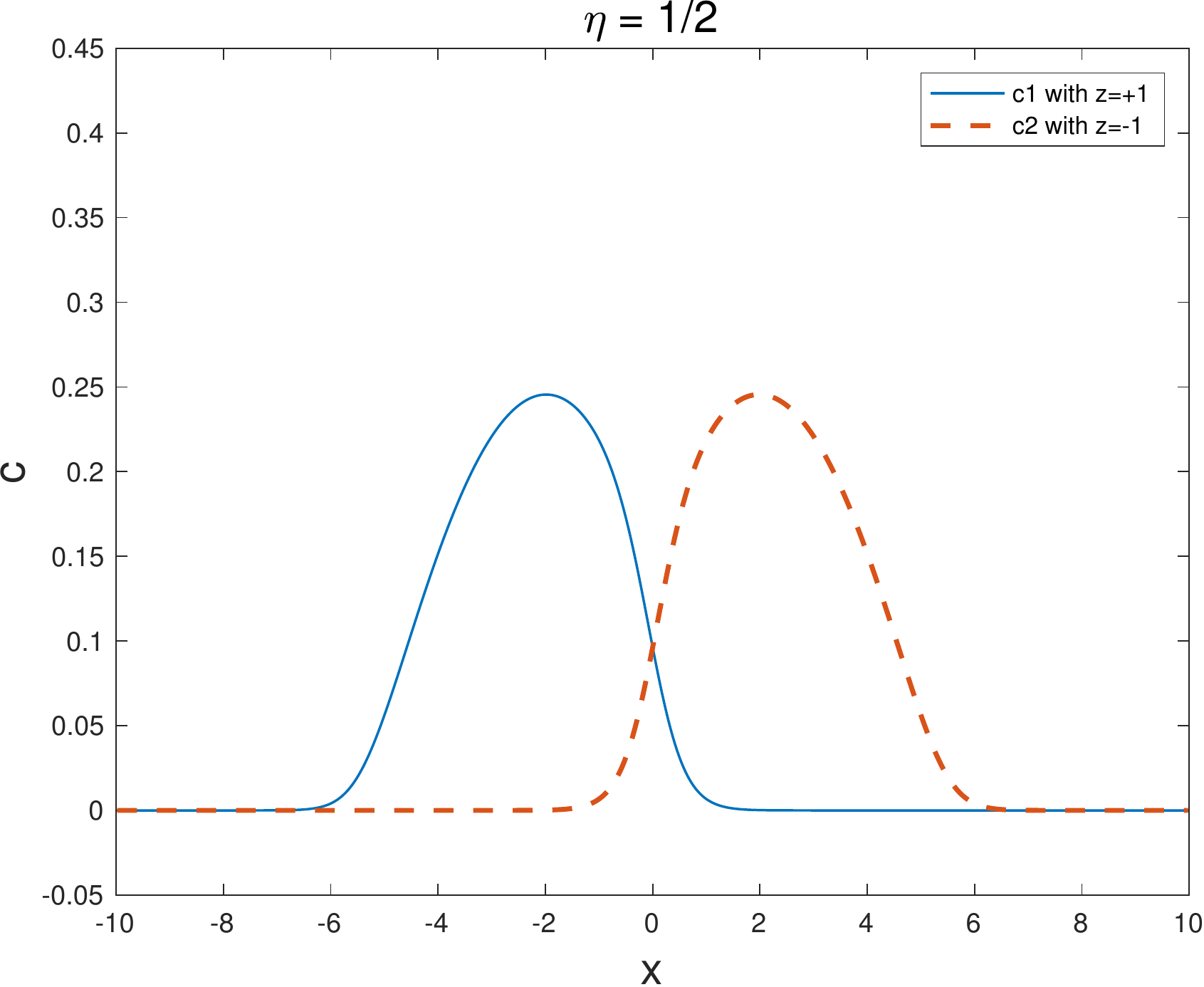}
		\end{minipage}
		\begin{minipage}[t]{0.33\linewidth}
			\centering
			\includegraphics[width=1.0\linewidth]{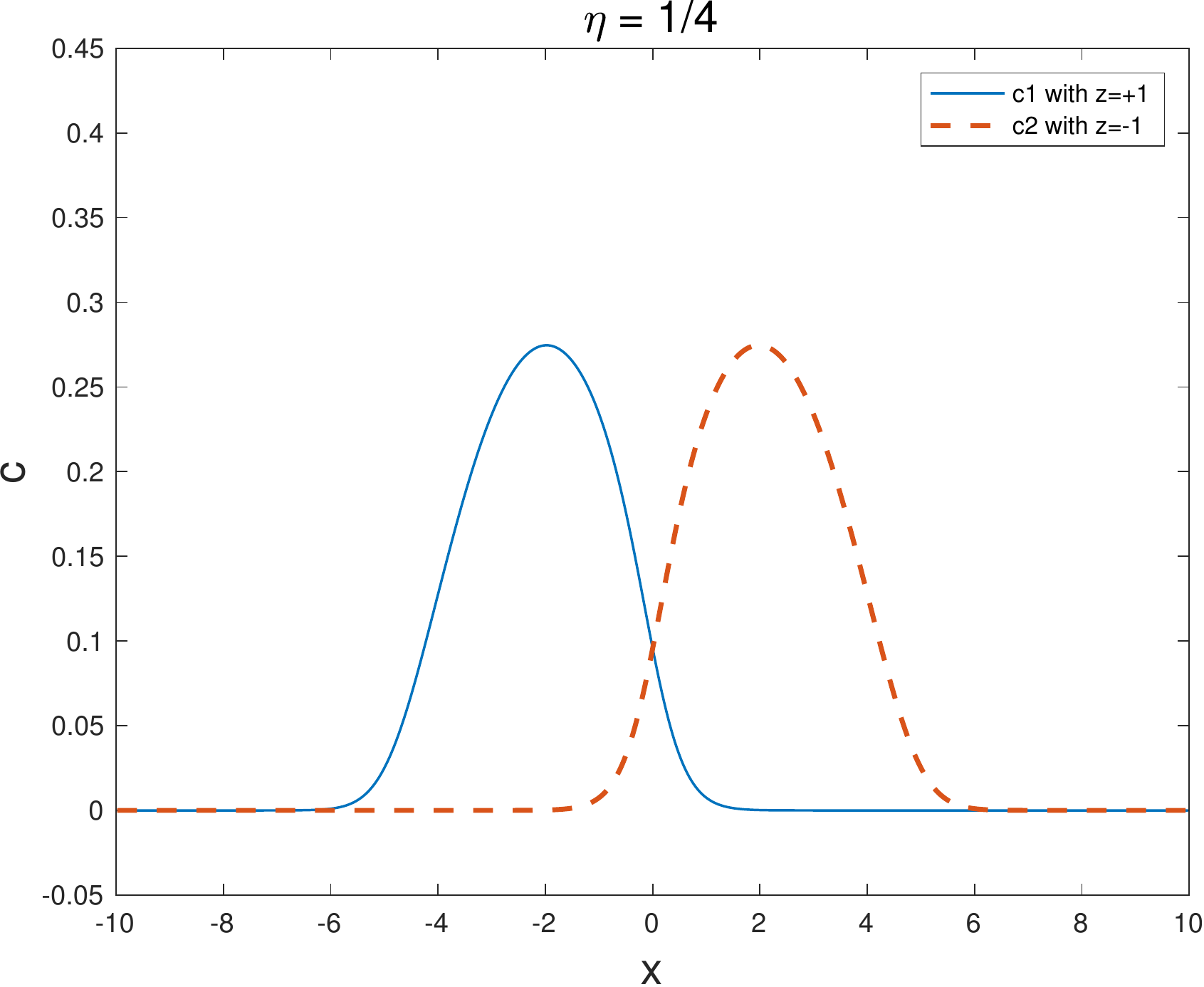}
		\end{minipage}
		\begin{minipage}[t]{0.33\linewidth}
			\centering
			\includegraphics[width=1.0\linewidth]{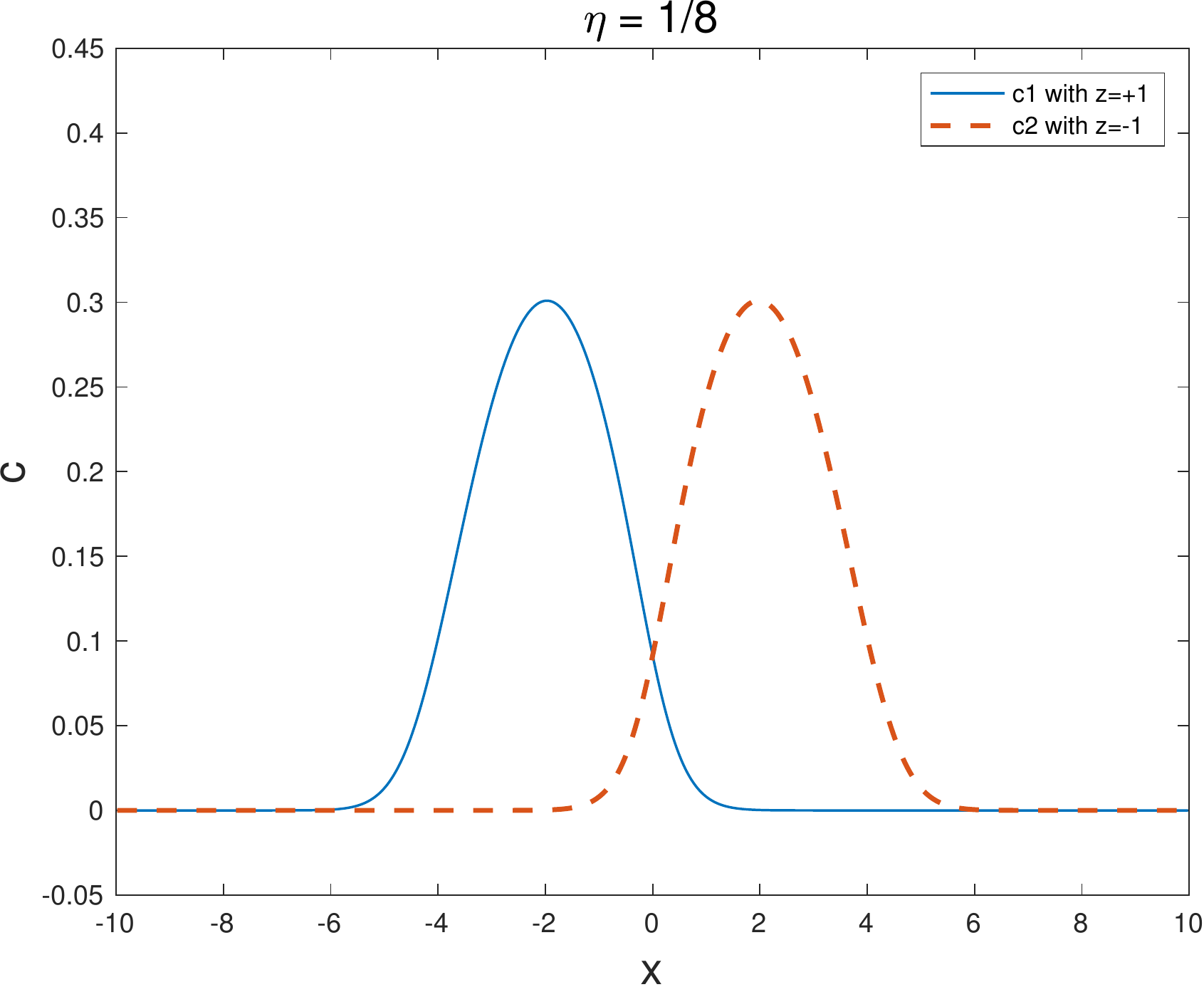}
		\end{minipage}
	}
	\subfigure{
		\begin{minipage}[t]{0.33\linewidth}
			\centering
			\includegraphics[width=1.0\linewidth]{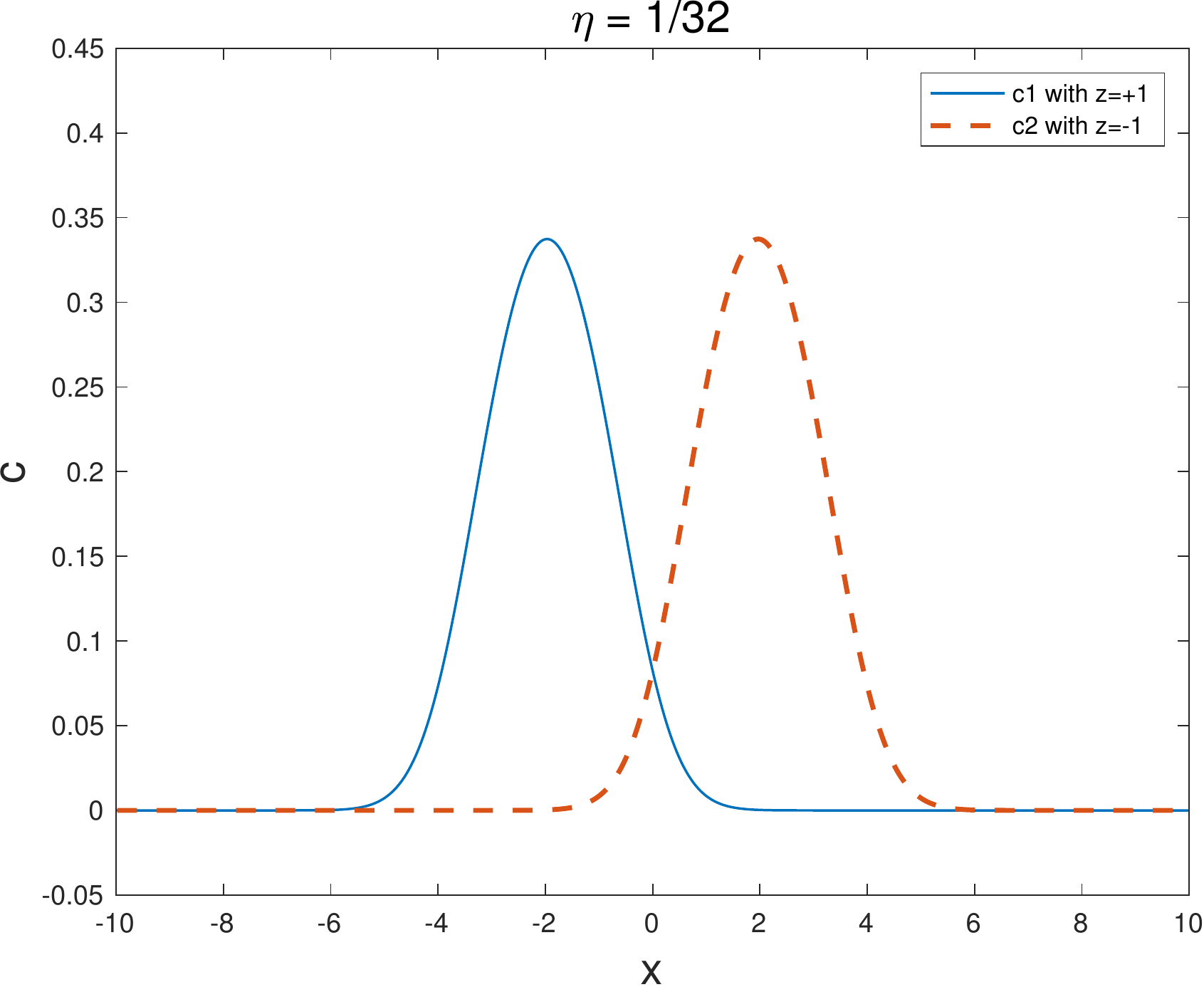}
		\end{minipage}
		\begin{minipage}[t]{0.33\linewidth}
			\centering
			\includegraphics[width=1.0\linewidth]{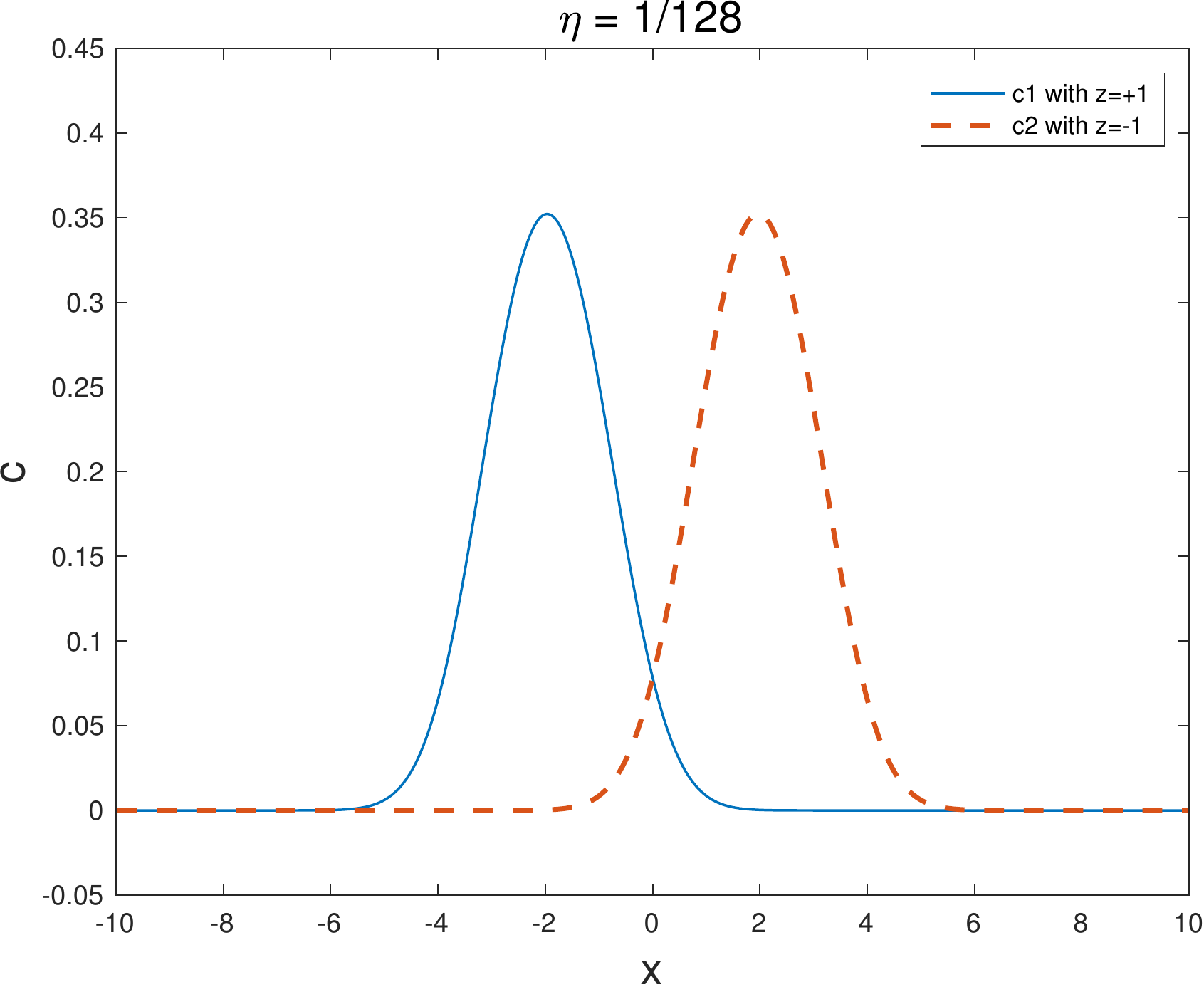}
		\end{minipage}
		\begin{minipage}[t]{0.33\linewidth}
			\centering
			\includegraphics[width=1.0\linewidth]{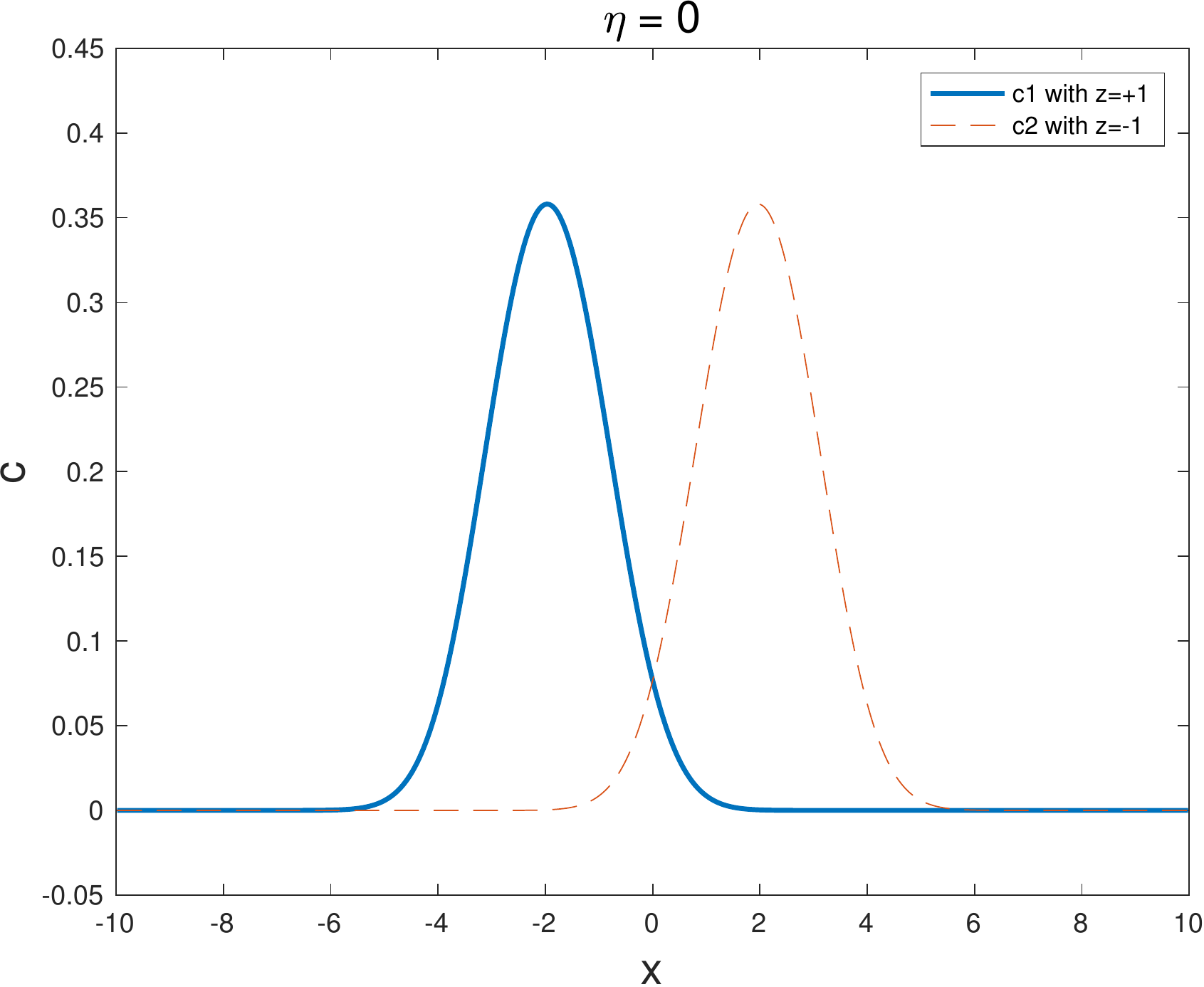}
		\end{minipage}
	}
	\caption{Example 3: The steady state density solutions $c_m$ with different $\eta$ with the mesh size $\Delta x$ being 0.0390625}
	\label{ex3eta}
\end{figure}

\subsection{Example 4}
As for another way in  in \cite{Eisenberg2020, Eisenberg2017} to model ionic and water flows  which includes voids, polarization effect of water, and ion-ion and ion-water correlations in electrolyte solutions, the correlated electric potential $\Phi(\boldsymbol{x})$ can be considered as a convolution of potential  ${\Phi}_{\mathcal{K}}(\boldsymbol{x})$ obtained from (\ref{phik}) with the exponential van der Waals potential kernel \cite{PhysRevLett2004, Eisenberg2017, Translation1979, Eisenberg2016} 
$$
W\left(\boldsymbol{x}\right)=\frac{e^{-\left|\boldsymbol{x}\right| / l_{c}}}{\left|\boldsymbol{x}\right| / l_{c}},
$$
where $l_c$ is the correlation length, i.e.
\begin{equation}
\label{phi1}
\Phi(\boldsymbol{x})=\int_{\mathbb{R}^d} \frac{1}{l_{c}^{2}} W\left(\boldsymbol{x}-\boldsymbol{x}^{\prime}\right) {\Phi}_{\mathcal{K}}\left(\boldsymbol{x}^{\prime}\right) d \boldsymbol{x}^{\prime}.
\end{equation}
We rewrite ${\Phi}_{\mathcal{K}}\left(\boldsymbol{x}\right)$ (\ref{K}) (\ref{phik}) in the convolution form and then (\ref{phi1}) becomes
\begin{equation}
\label{phi2}
\Phi(\boldsymbol{x})=\int_{\mathbb{R}^d} \frac{1}{l_{c}^{2}} W\left(\boldsymbol{x}-\boldsymbol{x}^{\prime}\right) \int_{\mathbb{R}^d} \mathcal{K}\left(\boldsymbol{x}^{\prime}-\boldsymbol{x}^{\prime \prime}\right) \rho\left(\boldsymbol{x}^{\prime \prime}\right) d \boldsymbol{x}^{\prime \prime}
d \boldsymbol{x}^{\prime}.
\end{equation}

Next we apply our numerical method to a one-dimension example to show a simple numerical exploration on such modeling phenomenon. The test model is not physically relevant in one-dimension, and thus this numerical example is a toy model. However, it is easy to extend it to high-dimensional cases, where we omit it in this paper.

Consider the kernel $\mathcal{K}(x) = \exp(-|x|)$, $W\left(x\right)=\frac{e^{-\left|x\right| / l_{c}}}{\sqrt{(x^2 + \epsilon^2)} / l_{c}},\epsilon = \frac{1}{10}$, $V_{\text{ext}}(x) = \frac{1}{2} x^2$ and the initial conditions (\ref{model4}) are given by (\ref{iniex1}).
In this part, we take the parameter $l_c = 7.44, 1, \frac{1}{64}$, the computation domain as $[-2L, 2L], \ L = 10$ and the uniform mesh size $\Delta x = 0.0390625\ (N = 2^{10})$. 
The results with which we are concerned are on the domain $[-L, L]$. Then Figure \ref{lc} shows different steady state solutions with different values of $l_c$. The concentrations of the positive ions and the negative ions move towards each other due to the electrostatic attraction with time $t$ and the concentrations move more closely to each other near $x = 0$ for smaller $l_c = \frac{1}{64}$ when they converge to the equilibrium and on the contrary, the concentrations of the steady state move a little further away from each other near $x = 0$ for larger $l_c = 7.44$.

\begin{figure}[htp] 
	\subfigure{
		\begin{minipage}[t]{0.33\linewidth}
			\centering
			\includegraphics[width=1.0\linewidth]{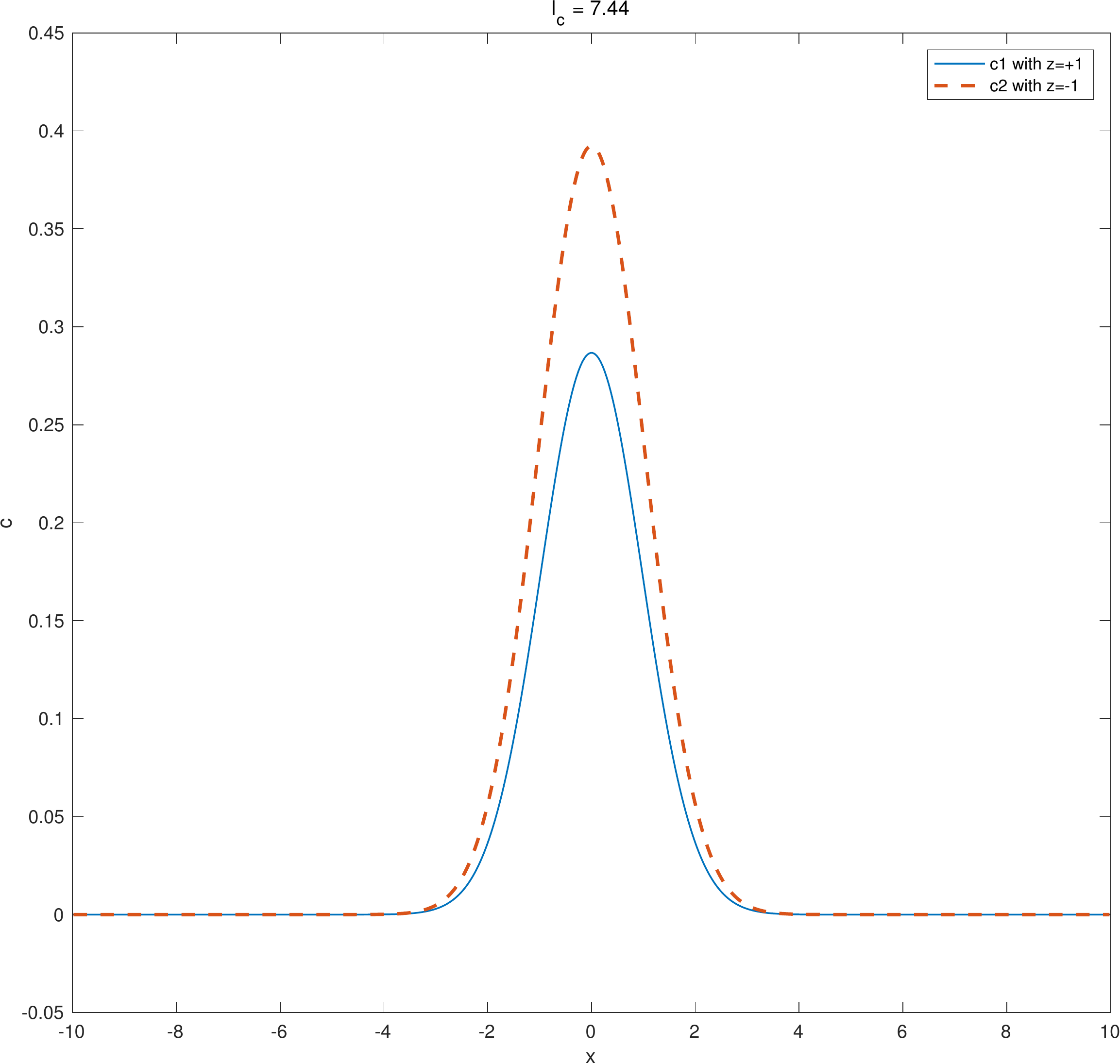}
		\end{minipage}
		\begin{minipage}[t]{0.33\linewidth}
			\centering
			\includegraphics[width=1.0\linewidth]{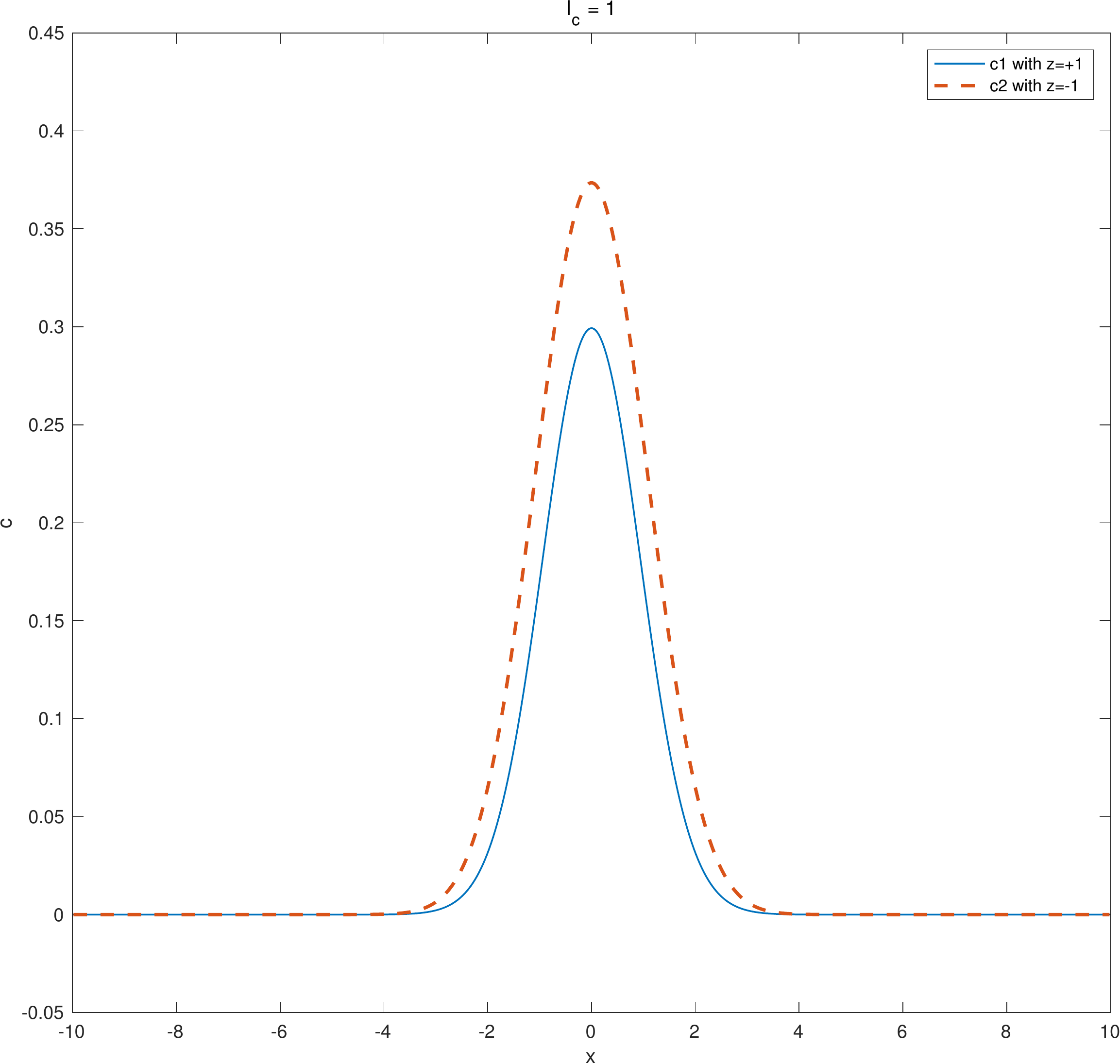}
		\end{minipage}
		\begin{minipage}[t]{0.33\linewidth}
			\centering
			\includegraphics[width=1.0\linewidth]{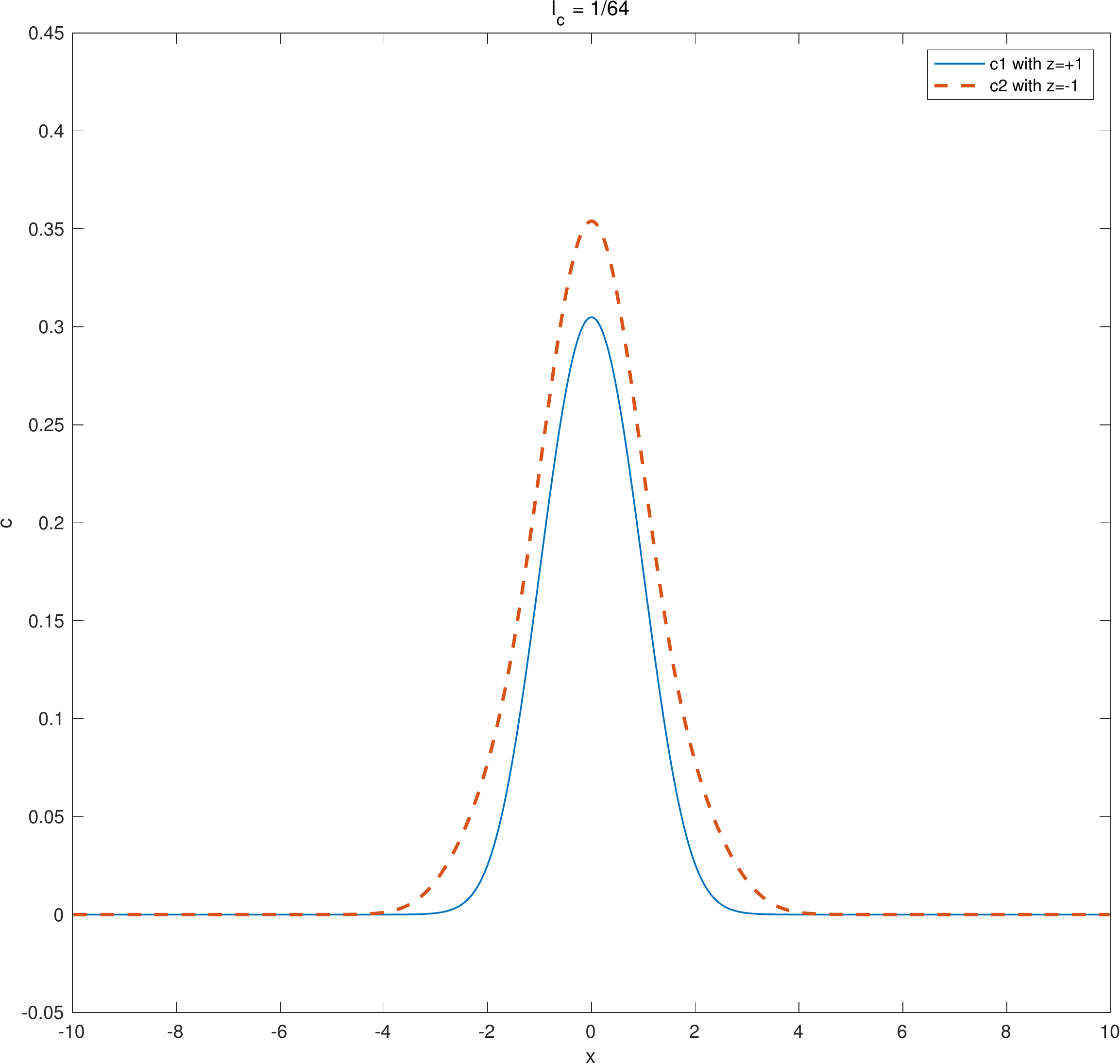}
		\end{minipage}
	}
	\caption{Example 4: The steady state density solutions $c_m$ with different $l_c$}
	\label{lc}
\end{figure}

\section{Conclusion}
\label{sec5}
In this paper, we focus on the model for complex ionic fluids proposed by EnVarA method and analyze the basic properties of the Cauchy problem of it different forms of the electrostatic potential and the steric repulsion of finite size effect and capture the well-posedness with certain regularized kernel. Then a finite volume scheme to the field system in 1D and 2D cases is proposed to observe the transport of the ionic species and verify the basic properties, such as positivity-preserving, mass conservation and discrete free energy dissipation. We also provide series of numerical experiments to demonstrate the small size effect in the model. More appropriate methods will be discussed in the following papers.

\bibliographystyle{plain}
\bibliography{paper01_new}
	
\end{document}